\documentclass[12pt]{article}
\usepackage[margin=1in]{geometry}
\usepackage{amsmath, amssymb, amsthm}
\usepackage{graphicx}
\usepackage{ytableau}

\newtheorem{prop}{Proposition}[section]
\newtheorem{cor}[prop]{Corollary}
\newtheorem{thm}[prop]{Theorem}
\newtheorem{lemma}[prop]{Lemma}
\newtheorem{conj}[prop]{Conjecture}

\newenvironment{casework}{
    \begin{enumerate}
    \renewcommand{\labelenumi}{\bf{\text{Case } \arabic{enumi}\text{.}}}}
{\end{enumerate}}

\theoremstyle{remark}

\usepackage[bookmarks=false, naturalnames=true]{hyperref}

\theoremstyle{definition}
\newtheorem{definition}[prop]{Definition}
\newtheorem{remark}[prop]{Remark}

\newtheorem{eg}[prop]{Example}
\newtheorem{qn}[prop]{Question}

\usepackage{paralist}

\newcommand{\oddeq}{\underset{\text{odd}}{\sim}}
\newcommand{\oddreq}{\underset{\text{odd}}{\sim^r}}
\newcommand{\eveneq}{\underset{\text{even}}{\sim}}
\newcommand{\evenreq}{\underset{\text{even}}{\sim^r}}

\newcommand{\sea}[1]{\underset{#1-\text{ASE}}{\sim}}
\newcommand{\sesa}[1]{\underset{#1-\text{SASE}}{\sim}}

\DeclareMathOperator\asc{Asc}
\DeclareMathOperator\des{Des}

\title{Beyond alternating permutations: Pattern avoidance in Young diagrams and tableaux}
\author{Nihal Gowravaram, Ravi Jagadeesan\\
Mentor: Joel Brewster Lewis}
\date{November 28, 2012}

\begin{document}

\maketitle

\begin{abstract}
We investigate pattern avoidance in alternating permutations and generalizations thereof.
First, we study pattern avoidance in an alternating analogue of Young diagrams. In particular,
we extend Babson-West's notion of shape-Wilf equivalence to apply to alternating
permutations and so generalize results of Backelin-West-Xin and Ouchterlony to alternating
permutations. Second, we study pattern avoidance in the more general context of
permutations with restricted ascents and descents.
We consider a question of Lewis regarding
permutations that are the reading words of thickened staircase Young tableaux,
that is, permutations that have $k-1$ ascents followed by a descent, followed
by $k-1$ ascents, et cetera.  We determine
the relative sizes of the sets of pattern-avoiding $(k-1)$-ascent permutations
in terms of the forbidden pattern.  Furthermore, inequalities in the
sizes of sets of pattern-avoiding permutations in this context
arise from further extensions of shape-equivalence type enumerations.
\end{abstract}

\section{Introduction}
The theory of pattern avoidance in permutations
is concerned with enumerative problems and has connections to computer science,
algebraic combinatorics, algebraic geometry, and representation theory.
The fundamental question is to determine the number of permutations
of a given length that avoid a certain type of forbidden subsequence.
For example, the only permutations that avoid 21 are the identity permutations.
The theory first arose in the study of stack-sortable permutations; for
example, Knuth~\cite{Knuth}
showed that stack-sortable permutations are exactly those that avoid the pattern 231.
Additionally, generalized stack-sortable permutations are characterized by the avoidance
of longer patterns; for an exposition, see \cite[Chapter 8]{BonaBook}.
MacDonald~\cite{MacD} demonstrated that vexillary permutations, objects of interest
in algebraic combinatorics, are characterized by 2143-avoidance.
Furthermore, Lakshmibai and Sandhya \cite{SchubertV} proved that permutations
that simultaneously avoid 3412 and 4231
index smooth Schubert varieties, which are studied in algebraic geometry.
Billey and Warrington~\cite{BilleyWarrington} showed that an interesting
class of Kazhdan-Lusztig polynomials, which arise in representation
theory, are indexed by permutations that simultaneously avoid 321 and
four longer patterns.  These applications motivate the study of permutations
that avoid patterns of arbitrary length.

Herb Wilf asked the question of when two patterns are equally difficult to avoid.
The first non-trivial result of this type is the remarkable fact that all patterns of length 3 are equally difficult to avoid.
Simion and Schmidt~\cite{SimionSchmidt} gave a particularly elegant bijective proof.
The bijections in Section~\ref{sec:AD-Young diagrams} 
can be viewed as generalizations of \cite{SimionSchmidt}.

Pattern-avoiding alternating permutations were first studied by Mansour~\cite{Mansour} and
by Deutsch and Reifegerste (as documented in \cite[Problem h$^7$]{CatalanEC})
who proved that the number of alternating permutations of a given length that avoid a pattern of length $3$ is
a Catalan number.  The enumeration is particularly interesting in that
the number of permutations of a given length that avoid a pattern of length 3
is also a Catalan number.  This suggests that pattern-avoiding alternating permutations have interesting enumerative
properties both independently and in relation to ordinary pattern avoidance.  In this paper, we develop
further connections between the pattern avoidance of ordinary and alternating permutations while
also generalizing beyond alternating permutations.

We build on the work of Backelin, West, and Xin; their result is the following theorem.
\begin{thm}[\cite{BWX}, Theorem 2.1]\label{BWXThm2.1}
For all $t \ge k$ and permutations $q$ of $\{k+1,k+2,\cdots,t\}$,
the patterns $(k-1)(k-2)(k-3)\cdots 1kq$ and $k(k-1)(k-2)\cdots1q$ are
Wilf-equivalent.\end{thm}
B\'{o}na \cite{Bona1221} proved a variant of Theorem~\ref{BWXThm2.1} for alternating permutations in the case
of $k = 2$ and $q = 345\cdots t,$ while Ouchterlony \cite{Ouch}
proved a similar result for
doubly alternating permutations (alternating permutations whose inverse is alternating)
in the case of $k = 2$.  The organization of the paper is as follows.
In Section~\ref{sec:Definitions}, we recall the basic definitions.
In Section~\ref{sec:AD-Young diagrams}, we generalize the method of \cite{BW, BWX}
to apply to permutations with restricted ascents and descents, using
objects that we call \emph{AD-Young diagrams}.  This provides
a framework of alternating shape-equivalence that attempts to fully extend Theorem~\ref{BWXThm2.1}
to alternating permutations.  In Section~\ref{sec:ShapeEquivs}, we use AD-Young diagrams
to prove our main results: Theorems~\ref{1221Alt} and~\ref{213321AltRAlt}, which
are variants of Theorem~\ref{BWXThm2.1}
for alternating and reverse alternating permutations in the cases of $k = 2, 3$.
We also consider patterns of short length, and in Section~\ref{sec:Nonequivalences} we prove certain non-equivalences
of patterns that relate to our AD-Young diagram equivalences.

In \cite{Lewis2011}, Lewis proved basic enumerations of pattern-avoiding
generalized alternating permutations; in particular, he considered pattern
avoidance in permutations that have
$k-1$ ascents followed by a descent, followed by $k-1$ ascents, et cetera.
He computed the number of such
permutations that avoid certain identity patterns.
Lewis asked questions about this and further generalizations of alternating
permutation pattern avoidance in \cite{Lewis2012}.
In Section~\ref{Generalized alternating permutations},
we determine the relative sizes of the sets of pattern-avoiding
$(k-1)$-ascent permutations of lengths $n$ and $n+1$ in terms of the forbidden pattern;
the results of this section constitute our main results regarding $(k-1)$-ascent permutations.
In Section~\ref{sec:ShapeEquivGenAlt}, we give applications of the AD-Young diagram framework to generalizations of alternating permutations,
and we conclude the paper by posing open questions.

\subsection*{Acknowledgements}
We would like to thank the
PRIMES program of the MIT Math Department, where this research was done.
Furthermore, we would like to thank our mentor Dr. Joel Lewis of the University
of Minnesota for his incredibly helpful guidance and insight, and for
suggesting the topic of pattern-avoiding alternating permutations to us.
Lastly, we would like to thank our parents for helpful discussions and research guidance.

\section{Definitions and background}
\label{sec:Definitions}

For a nonnegative integer $n$, let $[n]$ denote the set $\{1,2,3,\ldots,n\},$
and let $S_n$ denote the set of permutations of $[n]$.
We treat a permutation $w \in S_n$ as a sequence $w_1w_2w_3\cdots w_n$ that contains
every element of $[n]$ exactly once.  A permutation $w$ is said to \emph{contain}
a permutation $q$ if there is a subsequence of $w$ that is order-isomorphic to $q$;
for example, the subsequence $246$ of $214536$ shows that $214536$ contains $123$.
If $w$ does not contain $q$, we say that $w$ \emph{avoids} $q$.  Given
a pattern $q$, let $S_n(q)$ denote the set of permutations of length $n$
that avoid $q$.  If patterns $p$ and $q$ are such that $|S_n(p)| = |S_n(q)|$
for all $n$, we say that $p$ and $q$ are \emph{Wilf-equivalent}.

A permutation $w \in S_n$ is called \emph{alternating} if
$w_1 < w_2 > w_3 < \cdots$
and \emph{reverse alternating} if
$w_1 > w_2 < w_3 > \cdots.$
Reverse alternating permutations can be transformed into alternating permutations
(and vice versa) by the \emph{complementation} map that sends
a permutation $w = w_1w_2\cdots w_n$ to $w^c = (n+1-w_1)(n+1-w_2)\cdots (n+1-w_n)$.
Given a pattern $q$, let $A_n(q)$ (resp. $A'_n(q)$) denote the set of alternating
(resp. reverse alternating)
permutations of length $n$ that avoid $q$.  If $p$ and $q$ are such that
$|A_n(p)| = |A_n(q)|$ (resp. $|A'_n(p)| = |A'_n(q)|$) for all even $n$, we say
that $p$ and $q$ are \emph{equivalent for even-length alternating} (resp. \emph{reverse alternating})
permutations and we write $p \eveneq q$ (resp. $p \evenreq q$).  We make similar definitions for
odd-length permutations.  Furthermore, because $|A_n(q)| = |A'_n(q^c)|$ for all $n, q,$
we have that $p \eveneq q$ if and only if $p^c \evenreq q^c$, and similarly
for the odd length.

A permutation $w$ is said to have \emph{descent type $k$} if
\[w_1 < w_2 < \cdots < w_k > w_{k+1} < w_{k+2} < \cdots < w_{2k} > w_{2k+1} < w _ {2k+2} < \cdots.\]
Thus, such a permutation may be thought of as a series of \emph{rows} of length $k$ with values in strictly increasing order, with a possibly \emph{incomplete} final row, as in Figure~\ref{DnkSSYT}.
\begin{figure}
\center{
\begin{ytableau}
\none & \none & \none & \none & 1 & 6\\
\none & \none & 3 & 7 & 8 \\
2 & 4 & 5
\end{ytableau}}
\caption{The permutation $24537816,$ which has descent type 3, is obtained by reading
the entries of a skew standard Young tableau of shape $(6,5,3)/(4,2)$ from
left to right and bottom to top.  Because
the final (top) row has only 2 entries, it is incomplete.}
\label{DnkSSYT}
\end{figure}
Given pattern $q$, let $D^k_n(q)$ denote the set of permutations with descent type $k$ that avoid $q$. For example, alternating permutations have descent type 2,
and therefore we have $D^2_n(q) = A_n(q)$ for all $q$. In Section~\ref{Generalized alternating permutations}, we consider the enumerations of permutations of a fixed descent type that avoid a fixed pattern, and we study the relative sizes of $D^k_n(q)$ and $D^k_{n+1}(q)$.

\section{The AD-Young diagram framework}
\label{sec:AD-Young diagrams}
Given a permutation $p$, let $M(p)$ denote its permutation matrix,
and given matrices $A$ and $B$, let $A \oplus B = \begin{bmatrix}
A & 0 \\
0 & B
\end{bmatrix}.$
We assume that the reader is familiar with the basic terminology
of Young diagrams and tableaux; see, for example, \cite[Chapters 2 and 6]{BonaBook}.
We draw Young diagrams in English notation and use matrix coordinates, and
for example $(1,2)$ is the second square in the first row of a Young diagram.  Furthermore,
we require all Young diagrams to have the same number of rows and columns.

In \cite{BW, BWX}, the notion of pattern avoidance is extended
to transversals of a Young diagram, and analogue of the Wilf-equivalence
of permutations is the shape-Wilf equivalence of permutation matrices.
The critical theorem of \cite{BW}
is that if $M$ and $N$ are permutation
matrices that are shape-Wilf
equivalent, and $C$ is any permutation matrix, then the matrices
$M \oplus C$ and $N \oplus C$
are shape-Wilf equivalent.  We generalize the idea of a transversal of a Young diagram
and refine shape-Wilf equivalence to apply to alternating permutations.

\begin{definition}
\label{ADYoung}
Let $Y$ be a Young diagram with $k$ rows. 
If $A$ and $D$ are disjoint subsets of $[k-1]$
such that if $i \in A \cup D$, then the $i$th and $(i+1)$st rows of $Y$
have the same length, then we call the triple $\mathcal{Y} = (Y, A, D)$ an
\emph{AD-Young diagram}.
We call $Y$ the Young diagram of $\mathcal{Y}$,
$A$ the \emph{required ascent set} of $\mathcal{Y}$, and $D$ the \emph{required descent set} of $\mathcal{Y}$.
\begin{figure}
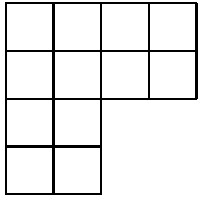

\center{
\begin{ytableau}
*(clear) & & &\\
& & &\\
&\\
&
\end{ytableau}}
\caption{If $Y = (4^2,2^2)$, $A = \emptyset$, and $D = \{3\}$, then $(Y, A, D)$ is an AD-Young diagram.}
\label{fig:ADYoungDef}
\end{figure}
See Figure~\ref{fig:ADYoungDef}.
\end{definition}

As in \cite{BW, BWX, StanWest}, a \emph{transversal} of Young diagram $Y$ is a set of squares $T = \{(i,t_i)\}$
such that
every row and every column of $Y$ contains exactly one member of $T$.
\begin{definition}
Given a transversal $T = \{(i,t_i)\}$, let
$\asc(T) = \{i \in [k-1] \mid t_i < t_{i+1}\}$ and $
\des(T) = \{i \in [k-1] \mid t_i > t_{i+1}\}.$
We call $\asc(T)$ the \emph{ascent set} of $T$ and $\des(T)$ the \emph{descent set} of $T$.
If $A \subseteq A'$ and $D \subseteq D'$, then we say that $T$ a \emph{valid transversal} of $\mathcal{Y}$.
\end{definition}
\begin{eg}
If $T$ is a transversal of a Young diagram $Y$, then $T$ is a valid transversal of
the AD-Young diagram $(Y, \emptyset, \emptyset)$.
\end{eg}

Except for a brief digression
in Section~\ref{sec:ShapeEquivGenAlt}, we restrict ourselves to the AD-Young
analogues of alternating and reverse alternating permutations.

\begin{definition}
Given positive integers $x,y$ and an AD-Young diagram $(Y, A, D)$ such
that $Y$ has $k$ rows, we say that $(Y, A, D)$ is $x,y$\emph{-alternating}
if $A,D$ satisfy the property that
if $x - 1 \le i \le k-y$, then $i \in A$ if and only if $i + 1 \in D$.
\end{definition}

If $\mathcal{Y}$ is $x,y$-alternating, then $\mathcal{Y}$
is $a,b$-alternating for all $a,b$ with $a \ge x$ and $b \ge y$.
If $\mathcal{Y}$ is $1,y$-alternating, then we say that $\mathcal{Y}$ is $y$-alternating,
while if $\mathcal{Y}$ is $2,y$-alternating, then we say that $\mathcal{Y}$ is $y$-semialternating.
Alternating AD-Young diagrams will be the counterpart of alternating permutations, while
semialternating AD-Young diagrams allow reverse alternating permutations.

\begin{eg}
Let $Y = (4^4)$.  Then, $(Y, \{1\}, \{2\})$ is 1-alternating, while
$(Y, \{1,3\}, \{2\})$ is 2-alternating but not 1-alternating.
Furthermore, $(Y, \{2,4\}, \{1, 3\})$ is 1-semialternating but not $y$-alternating
for $y \le 4$.
\end{eg}

The notion of pattern avoidance is exactly as in \cite{BW, BWX, StanWest}; if a transversal
$T = \{(i,t_i)\}$ of a
Young diagram $Y$ \emph{contains} a $r \times r$ permutation matrix $M$ if there are rows $a_1 < a_2 < \cdots < a_r$
and columns $b_1 < b_2 < \cdots < b_r$ of $Y$ such that $(a_r,b_r) \in Y$ and
the restriction of $T$ to the rows $a_i$ and the columns
$b_i$ has 1's exactly where $M$ has 1's.
If $T$ does not contain $M$, then we say that $T$ \emph{avoids} $M$.
\begin{figure}
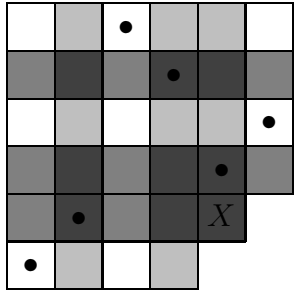

\center{
\begin{ytableau}
*(clear) & *(lightgray) & \bullet & *(lightgray) & *(lightgray) & \\
*(gray) & *(darkgray) & *(gray) & *(darkgray) \bullet & *(darkgray) & *(gray) \\
 & *(lightgray) & & *(lightgray) & *(lightgray) & \bullet \\
*(gray) & *(darkgray) & *(gray) & *(darkgray) & *(darkgray) \bullet & *(gray) \\
*(gray) & *(darkgray) \bullet & *(gray) & *(darkgray) & *(darkgray) X\\
 \bullet & *(lightgray) & & *(lightgray)
\end{ytableau}}
\caption{The transversal $T = \{(1,3),(2,4),(3,6),(4,5),(5,2),(6,1)\}$ of $Y = (6^4,5,4)$
contains $M(231)$ because the restriction of $T$ to the gray columns and light gray
rows is a copy of $M(231)$ in $T$; we require that $X \in Y$.  However, $T$ does not
contain $M(4321)$; for example, the restriction of $T$ to rows $3,4,5,6$
and columns $1,2,5,6$ is not a copy of $M(4321)$ in $T$ because $(6,6) \notin Y$.}
\label{fig:MatrixContainment}
\end{figure}
See Figure~\ref{fig:MatrixContainment}.  Given an AD-Young diagram
$\mathcal{Y}$ and a permutation matrix $M$, let $S_{\mathcal{Y}}(M)$ denote the set of valid transversals
of $\mathcal{Y}$ that avoid $M$.
\begin{definition}
If $M$ and $N$ are permutation matrices such that $|S_{\mathcal{Y}}(M)| = |S_{\mathcal{Y}}(N)|$
for all $x$-alternating AD-Young diagrams $\mathcal{Y}$, we say that $M$ and $N$ are
\emph{shape-equivalent for $x$-alternating AD-Young diagrams}; we write $M \sea{x} N$.
If we have $|S_{\mathcal{Y}}(M)| = |S_{\mathcal{Y}}(N)|$ for all $x$-semialternating AD-Young diagrams
$\mathcal{Y}$, then we say that $M$ and $N$ are \emph{shape-equivalent for $x$-semialternating AD-Young diagrams};
we write $M \sesa{x} N$.
\end{definition}

If $M \sea{y} N$, then we have $M \sea{x} N$ for all positive integers $x \le y$,
while if $M \sesa{y} N$, then we have that $M \sea{x} N$ and $M \sesa{x} N$ for
all positive integers $x \le y$.
Because $(Y, \emptyset, \emptyset)$ is an alternating AD-Young diagram for every Young diagram
$Y$, we have that if $M$ and $N$ are shape-equivalent for $1$-alternating AD-Young diagrams,
then $M$ and $N$ are shape-Wilf equivalent;
that is, for all Young diagrams $Y$, the number of transversals of $Y$ that avoid $M$ is the same as the
number of transversals of $Y$ that avoid $N$.
We explicitly give the connection of alternating and semialternating AD-Young diagrams to
alternating and reverse alternating permutations, respectively.

\begin{prop}
\label{AltPermToDiag}
Let $p$ and $q$ be permutations.
\begin{enumerate}[(a)]
\item If $M(p) \sea{1} M(q)$,
then $p \oddeq q$.
\item If $M(p) \sesa{1} M(q)$,
then $p \evenreq q$.
\item If $M(p) \sea{2} M(q)$,
then $p \eveneq q$.
\item If $M(p) \sesa{2} M(q)$,
then $p \oddreq q$.
\end{enumerate}
\end{prop}
\begin{proof}
We prove the first part; the remaining parts are similar.
Fix a nonnegative integer $n$, and we will show that $|A_{2n+1}(p)| = |A_{2n+1}(q)|$.
Consider the AD-Young diagram $\mathcal{Y} = (Y, A, D)$ given by
$Y = (2n+1^{2n+1})$, $A = \{1,3,5,\cdots,2n-1\}$, and $D = \{2,4,6,\cdots,2n\}$.
It is clear that $\mathcal{Y}$ is $1$-alternating.
Furthermore, a set $T = \{(i,b_i)\}$ is a valid transversal of $\mathcal{Y}$ if and only
if $b = b_1b_2\cdots b_{2n+1} \in A_{2n+1}$, and $T$ avoids $M(p)$ if and only if
$b$ avoids $p$.  Hence, we have
\[|A_{2n+1}(p)| = |S_{\mathcal{Y}}(M(p))| = |S_{\mathcal{Y}}(M(q))| = |A_{2n+1}(q)|,\]
as desired.
\end{proof}

\subsection{Generalization of Babson-West}
\label{sec:Extension}
The extension of shape-equivalences from $M \sim N$ to
$M \oplus C \sim N \oplus C$
is the analogue of \cite[Theorems 1.6 and 1.9]{BW}.  It is critical in generating infinite sets of nontrivial shape-equivalences.
We have two variants, one for alternating AD-Young diagrams and one for semialternating AD-Young diagrams.

\begin{thm}[Extension Theorem]
\label{ShapeExtend}
If permutation matrices $M$ and $M'$ are shape-equivalent for $x$-alternating (resp. $x$-semialternating) AD-Young diagrams and $C$
is an $r \times r$ permutation matrix, then we have
$M \oplus C \sea{(x+r)} M' \oplus C$
(resp. $\sesa{(x+r)}$).
\end{thm}

The remainder of this section will be devoted to the proof of Theorem~\ref{ShapeExtend}.
The first idea in the proof is to pass avoidance of
$M \oplus C$
by a transversal of a large parent AD-Young diagram $\mathcal{Y}$ to avoidance of $M$
by a transversal of a smaller successor AD-Young diagram; this idea stems
from the proof of \cite[Theorems 1.6 and 1.9]{BW}.  The successor map
mostly preserves
the alternating property of AD-Young diagrams
in the sense that if the parent is $(x+r)$-alternating
and $C$ is an $r \times r$ matrix, then the successor
is $x$-alternating.  Furthermore, it sends valid transversals to valid transversals.

The successor AD-Young diagram depends on the choice of transversal of $\mathcal{Y}$.
However, similar to \cite{BW}, we give an injection of the set of $M$-avoiding
transversals of the successor diagram into the set of $M \oplus C$-avoiding transversals of $\mathcal{Y}$;
the ascent set and the descent set of the successor diagram are chosen to facilitate this
reinsertion procedure.  We can then completely reduce the proof that $|S_\mathcal{Y}(M\oplus C)|
= |S_\mathcal{Y}(N \oplus C)|$ to a statement about $M$-avoiding transversals of $x$-alternating
AD-Young diagrams.

Fix an $r\times r$ permutation matrix $C$, a permutation matrix $P$, and an AD-Young diagram
$\mathcal{Y} = (Y, A, D)$ ($\mathcal{Y}$ need not be alternating).
Let $T$ be a valid transversal of $\mathcal{Y}$. In the language of~\cite{BW},
call a square $(a,b) \in Y$
\emph{dominant with respect to }$T$ if the restriction of $T$ to the
region of squares $(x, y) \in Y$ with $x > a$ and $y > b$
contains $C$.  Let $\mathcal{N}^C(T)$ denote the set of elements of $T$ that are
not dominant with respect to $T$, and let $\mathcal{D}^C$ denote the family
of sets $\mathcal{N}^C(T)$ as $T$ ranges over the valid transversals of $\mathcal{Y}$.

\begin{lemma}
\label{BW1}
The set of dominant squares of $\mathcal{Y}$ form a Young diagram.
Furthermore, given the set $\mathcal{N}^C(T)$ and the permutation matrix $C$, one can recover
the Young diagram of dominant squares.
\end{lemma}
\begin{proof}
See the proof of \cite[Theorems 1.6 and 1.9]{BW}.
\end{proof}

Given a set $N = \mathcal{N}^C(T)$, let $d^C(N)$ denote the set of squares of $Y$
that are dominant with respect to $T$; the fact that $d$ is well-defined follows from Lemma~\ref{BW1}.
For a set $N$ of squares of $Y$ and a permutation
matrix $P$, let $S^{N,C}_{\mathcal{Y}}(P)$ denote the set of valid transversals $T$ of $\mathcal{Y}$
that avoid the matrix
$\begin{bmatrix}
P & 0\\
0 & C\end{bmatrix}$
such that $\mathcal{N}^C(T) = N$.  It is clear that we have
\begin{equation}\left|S_{\mathcal{Y}}\left(\begin{bmatrix}
P & 0\\
0 & C\end{bmatrix}\right)\right|
=
\sum_{N \in \mathcal{D}} \left|S^N_{\mathcal{Y}}(P)\right|.
\label{SummingWithoutf}
\end{equation}

We will define a function $f^C$ on $\mathcal{D}^C$ (the value of $f$ is independent of $P$) with the following key properties,
to be proven after defining $f^C$.
The value $f(N)$ is the successor diagram.
\begin{lemma}
~\label{fIsAD}
For all $N \in \mathcal{D}^C$, $f^C(N)$ is an AD-Young diagram.
\end{lemma}
\begin{lemma}
\label{AltExtend}
If $\mathcal{Y}$ is $(x+r)$-alternating, then $f^C(N)$ is $x$-alternating.
\end{lemma}
\begin{lemma}
\label{Embed1}
For all $N^C \in \mathcal{D}^C$, we have
\[|S^{N}_\mathcal{Y}(P)| = |S_{f(N)}(P)|.\]
\end{lemma}

From the Young diagram of dominant squares $d(\mathcal{N}^C(T))$, delete
every row or column that contains a non-dominant
square of $T$, and call the resulting Young diagram $Y'$.
Each row and column of $Y'$ contains exactly 1 dominant member of $T$, and
thus $Y'$ has the same number of rows and columns.
Suppose that $Y'$ has $k$ rows,
and that for all $1 \le i \le k$ the $i$th column of $Y'$ was the $c_i$th column of $Y$
before the row and column deletion; similarly, suppose that for all $1 \le i \le k$, the $i$th
row of $Y'$ was the $r_i$th row of $Y$.
Let
\[A' = \{i \in [k-1] \mid r_i \in A \text{ and } r_{i+1} = r_i+1\},\]
and let
\[D' = \{i \in [k-1] \mid r_i \in D \text{ and } r_{i+1} = r_i+1\}.\]
By construction, the triple $(Y',A',D')$ depends only on $\mathcal{N}^C(T)$, and the set of dominant
squares $d^C(\mathcal{N}^C(T))$ (the value $d^C(N)$ does not depend on $P$).  Given a set $N = d^C(\mathcal{N}^C(T) \in \mathcal{D}$, let $f^C(N)$
denote the corresponding triple $(Y',A',D')$.  We prove that $f^C$ has the desired properties.
Lemmata~\ref{fIsAD} and~\ref{AltExtend} will be immediate from the following
lemma and proposition.

\begin{lemma}
\label{ShiftingC}
Let $T = \{(i,b_i)\}$ be a valid transversal of $\mathcal{Y}$.  If
the square $(j,y)$ is dominant with respect to $T$ and $b_{j+1} \le j$,
then the square $(j+1,y)$ is dominant with respect to $T$.
\end{lemma}
\begin{proof}  See Figure~\ref{fig:ForLemmaShiftingC}.
\begin{figure}
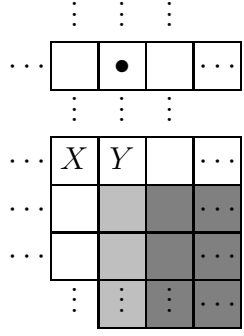

\center{
\begin{ytableau}
\none & \none [\vdots] & \none [\vdots] & \none [\vdots]\\
\none [\cdots] & & \bullet & & \cdots\\
\none & \none[\vdots] & \none[\vdots] & \none[\vdots]\\
\none[\cdots] & X & Y & & \cdots\\
\none[\cdots] & & *(lightgray) & *(gray) & *(gray) \cdots\\
\none[\cdots] & & *(lightgray) & *(gray) & *(gray) \cdots\\
\none & \none[\vdots] & *(lightgray)\vdots & *(gray)\vdots & *(gray) \cdots
\end{ytableau}}
\caption{If the bullet point is an element of $T$, then there are no elements
of $T$ among the light gray squares.  Thus, if square $X$ is dominant, there must be
a copy of $C$ among the dark gray squares, which implies that $Y$ is dominant.}
\label{fig:ForLemmaShiftingC}
\end{figure}
By the definition of dominant squares and because $(j,y)$ is dominant, there are rows
$j < e_1 < e_2 < \cdots < e_r$ and columns $y < f_1 < f_2 < \cdots < f_r$ such that the restriction
of $Y$ to the rows $e_i$ and the columns $f_k$ has members of $T$ exactly where $C$ has 1's.
If $j + 1 < e_1$, then the rows $e_i$ and the columns $f_k$ demonstrate that $(j+1,y)$
is dominant.  Otherwise, we have $j+1 \ge e_1$, which implies that $e_1 = j+1$.  The only
element of $T$ in row $j+1$ is $(j+1,b_{j+1})$, and it follows that $b_{j+1} = f_k$ for some $k$.
Regardless of $k$, we have $b_{j+1} \ge f_1 > y$, which implies that $(j+1,y)$ is dominant
by Lemma~\ref{BW1}.
\end{proof}

\begin{prop}
\label{AltTechnical}
Let $N \in \mathcal{D}^C$ and let $f^C(N) = (Y',A',D')$.
If $i \in A'$ satisfies $r_i + 1 \in D$, then $r_i+1 \in D'$.  If $i \in D'$
satisfies $r_i - 1 \in A$, then $r_i-1 \in A'$.
\end{prop}
\begin{proof}
Let $T = \{i,b_i\}$ be a valid transversal of $\mathcal{Y}$ with $\mathcal{N}^C(T) = N$.
If $i \in A'$ satisfies $r_i \in D$, then we have that $(r_i,b_{r_i})$ is dominant
and $b_{r_i+1} < b_{r_i}$.  By Lemma~\ref{ShiftingC}, $(r_i+1,b_{r_i+1})$ is dominant,
and it follows that $r_{i+1} = r_i + 1$ and $i + 1 \in D'$.

To prove the second part, we first prove that $(b_{r_i-1},r_i-1)$ is dominant.
Because $r_i-1 \in A$, we have that $b_{r_i-1} < b_{r_i}$.  Lemma~\ref{BW1}
implies that $(r_i-1,b_{r_i-1})$ is dominant.  This yields that $r_{i-1} = r_i - 1$,
and the fact that $i-1 \in A'$ follows by the definition of $A'.$
\end{proof}

\begin{proof}[Proof of Lemma~\ref{fIsAD}]
Let $f^C(N) = (Y',A',D')$.
By construction and Lemma~\ref{BW1}, $Y'$ is a Young diagram.
Because $A$ and $D$ are disjoint,
we have that $A'$ and $D'$ are disjoint. Let
$T = \{(i,b_i)\}$ be a valid transversal of $\mathcal{Y}$ with $\mathcal{N}^C(T) = N$.
and suppose that $j \in A' \cup D'$.  Let $y$ be the length of the $r_j$th row of $d(N)$.
By Lemma~\ref{BW1} and because $(r_j+1,b_{r_j+1})$ is dominant, we have that $y \ge b_{r_j+1}$.  Lemma~\ref{ShiftingC}
yields that $(r_j+1,y)$ is dominant, and thus the $r_j$th and $r_j+1$st rows of $d^C(N)$ have the same length.
It follows that the $j$th and $j+1$st rows of $Y'$ have the same length, as desired.
\end{proof}

\begin{proof}[Proof of Lemma~\ref{AltExtend}]
Let $f^C(N) = (Y',A',D')$ and suppose that $Y'$ has $k$ rows.
Let $i \in D'$ with $i \le k-x$, and we prove that $i - 1\in A'$. We have
$r_i \le r_{k-x} \le r_k - x \le n-r-x$.  Because $\mathcal{Y}$ is $x+r$-alternating, this implies that
that $r_i - 1 \in A$.  Proposition~\ref{AltTechnical} yields that $i - 1 \in A'$,
as desired.  The proof that $i \in A'$ with $i \le k-x$ implies $i + 1 \in D'$ is similar.
\end{proof}

\begin{proof}[Proof of Lemma~\ref{Embed1}]
We prove the equality by establishing a bijection.
Define the function $h: S^{N,C}_\mathcal{Y}(P) \rightarrow S_{f^C(N)}(P)$ by
mapping a transversal $T \in S^{N,C}_\mathcal{Y}(P)$ to the image of $T$
after deleting any row or column that contains a non-dominant member of $T$.
By definition of $N$, it is clear that $h(T)$ is a valid transversal
of $f^C(N)$.  Furthermore, if $h(T)$ contains $P$, then the set of dominant squares
of $T$ contain $P$, which implies that $T$ contains
$\begin{bmatrix}
P & 0\\
0 & C\\
\end{bmatrix}.$  Hence, we can conclude that if $T \in S^{N,C}_\mathcal{Y}(P)$, then
$h(T) \in S_{f^C(N)}(P)$.  To show that $h$ is a bijection, we will show that it
has an inverse.
Consider the function $h_2: S_{f^C(N)}(P) \rightarrow S^{N,C}_\mathcal{Y}(P)$
given by mapping a valid transversal $T  = \{(i,b_i)\}\in S_{f^C(N)}(P)$ to the transversal
$T' = N \cup \{(r_i,c_{b_i}) \mid (i,b_i) \in T\}$ of $Y$ (a priori, $h_2(T)$ is not
necessarily an element of $S^{N,C}_\mathcal{Y}(P)$).

We claim that if $T \in S_{f^C(N)}(P)$, then
$h_2(T)$ is a valid transversal of $\mathcal{Y}$.
Let $T_1$ be a valid transversal of $\mathcal{Y}$ such that $\mathcal{N}^C(T_1) = N$;
we introduce $T_1$ in order to exploit the fact that $N \in \mathcal{D}^C$.
Let $h_2(T) = \{(i,a_i)\},$ and let $T_1 = \{(i,a'_i)\}$.
Suppose that $j \in A$ and we will do
casework on which of $j, j+1$ are among the rows $r_i$ to prove that $j$ is in the ascent set of $h_2(T)$.

\begin{casework}
\item Neither $j$ nor $j+1$ are among the rows $r_i$.  Then, we have $a_j = a'_j < a'_{j+1} = a_{j+1}$,
as desired.

\item $j$ is among the rows $r_i$ but $j+1$ is not.  Assume for sake of contradiction
that $a_j > a_{j+1} = a'_{j+1}$.
Because $(j,a_j)$ is dominant with respect to $T_1$, by Lemma~\ref{ShiftingC} the square
$(j+1,a_j)$ is dominant with respect to $T_1$, which implies that $(j+1, a'_{j+1})$ is dominant
with respect to $T_1$ by Lemma~\ref{BW1}.  This contradicts the fact that $j+1$ is not among the rows
$r_i$.

\item We have that $j+1$ is among the rows $r_i$ but $j$ is not.  We claim that
this is impossible. Because $a'_j < a'_{j+1}$ and $j+1$ is among the rows $r_i$, we have that
$(j,a'_j)$ is dominant, which implies that $j$ is among the rows $r_i$.

\item Both $j$ and $j+1$ are among the rows $r_i$.  Suppose that $r_x = j$; then we have $x \in A'$,
which implies that $b_x < b_{x+1}$.  Therefore, we have $a_j = c_{b_x} < c_{b_{x+1}} = a_{j+1}$,
as desired.
\end{casework}

The casework proves that $j$ is in the ascent set of $h'(T)$.  Suppose that $j \in D$, and we will
prove that $j$ is in the descent set of $h_2(T)$ by dividing into the same cases.

\begin{casework}
\item Neither $j$ nor $j+1$ are among the indices $r_i$.  Then, we have $a_j = a'_j > a'_{j+1} = a_{j+1}$,
as desired.

\item $j$ is among the indices $r_i$ but $j+1$ is not.  Because $a'_j > a'_{j+1}$,
by Lemma~\ref{ShiftingC} the square $(j+1,a'_j)$ is dominant with respect to $T_1$.  This implies
that $(j+1,a'_{j+1})$ by Lemma~\ref{BW1}, which implies that $j+1$ is among the rows $r_i$.

\item $j+1$ is among the rows $r_i$ but $j$ is not.
By Lemma~\ref{BW1} and because $(j,a'_j)$ is not dominant with respect to $T_1$,
we have that $a_j = a'_j > a_{j+1}$, as desired.

\item Both $j$ and $j+1$ are among the indices $r_i$.  Suppose that $r_x = j$; then,
we have $x \in D'$, which implies that $b_x > b_{x+1}$.  Therefore, we have
$a_j = c_{b_x} > c_{b_{x+1}} = a_{j+1}$, as desired.
\end{casework}

The casework establishes that every element of $D$ is in the descent set of $h'(T)$, and it follows
that $h_2(T)$ is a valid transversal of $\mathcal{Y}$.  Because $T$ avoids $P$ and by the definition
of dominant squares, $h_2(T)$ avoids
$\begin{bmatrix}
P & 0\\
0 & C\end{bmatrix}.$
It is clear that $\mathcal{N}^C(h'(T)) = N$, and this implies that $h_2(T) \in S^{N,C}_{\mathcal{Y}}(P)$
for all $T \in S_{f^C(N)}(P)$.  Hence, $h$ and $h'$ are inverses, and thus
$h$ is a bijection.  The lemma follows.
\end{proof}

The following proposition is immediate from Lemma~\ref{Embed1} and Equation \ref{SummingWithoutf},
and we use it to prove Theorem~\ref{ShapeExtend} in the alternating case.
\begin{prop}
\label{Embed2}
For all permutation matrices $P, C$, we have that
\[\left|S_{\mathcal{Y}}\left(\begin{bmatrix}
P & 0\\
0 & C\end{bmatrix}\right)\right|
= \sum_{N \in \mathcal{D}^C} \left|S_{f^C(N)}(P)\right|.\]
\end{prop}

\begin{proof}[Proof of Theorem~\ref{ShapeExtend} in the alternating case]
Let $\mathcal{Y}$ be an $(x+r)$-alternating AD-Young diagram.
By Lemma~\ref{AltExtend} and because $M \sea{x} M'$, we have that
$|S_{f^C(N)}(M)| = |S_{f^C(N)}(M')|$
for all $N \in \mathcal{D}^C$.  Proposition~\ref{Embed2} applied to $P = M$
and $P = M'$ then yields that
\[
\left|S_{\mathcal{Y}}\left(\begin{bmatrix}
M & 0\\
0 & C\end{bmatrix}\right)\right|=
\sum_{N \in \mathcal{D}^C} \left|S_{f^C(N)}(M)\right|
= \sum_{N \in \mathcal{D}^C} \left|S_{f^C(N)}(M')\right|
= \left|S_{\mathcal{Y}}\left(\begin{bmatrix}
M' & 0\\
0 & C\end{bmatrix}\right)\right|,\]
as desired.
\end{proof}

In fact, the
alternating AD-Young diagrams
arose as an attempt to provide a neat description for a superset
of the closure of the set of AD-Young diagrams of the form
$(Y, A, D)$ with $Y$ an $n \times n$ square, $A = [n-1] \cap (2 \mathbb{Z} + 1)$,
and $D = [n-1] \cap 2\mathbb{Z}$ under such a successor map.  The need to account for required ascents and descents
significantly complicates both the definition
of the successor map and the resulting proof of the Extension Theorem~\ref{ShapeExtend}.

The proof of Theorem~\ref{ShapeExtend} in the semialternating case is almost identical.  We simply
replace Lemma~\ref{AltExtend} by the following lemma.

\begin{lemma}
\label{SemiAltExtend}
If $\mathcal{Y}$ is $(x+r)$-semialternating and $N \in \mathcal{D}$, then
$f^C(N)$ is $x$-semialternating.
\end{lemma}
\begin{proof}
Let $f^C(N) = (Y',A',D')$ and suppose that $Y'$ has $k$ rows.
Let $i \in D'$ with $1 < i \le k-x$, and we prove that $i - 1\in A'$. We have
$1 < r_i \le r_{k-x} \le r_k - x \le n-r-x$.  Because $\mathcal{Y}$ is $x+r$-alternating, this implies that
that $r_i - 1 \in A$.  Proposition~\ref{AltTechnical} yields that $i - 1 \in A'$,
as desired.  The proof that $i \in A'$ with $1 \le i \le k-x$ implies $i + 1 \in D'$ is similar.
\end{proof}

\section{Shape-equivalences for AD-Young diagrams}
\label{sec:ShapeEquivs}

We now prove two shape-equivalences.  For all positive integers $r$, let $I_r = M(123\cdots r)$ and let $J_r = M(r(r-1)(r-2) \cdots 1)$.  We will prove that $I_2 \sea{1} J_2$ and $J_3 \sesa{1} F_3$.  Using
the Extension Theorem~\ref{ShapeExtend}, we will obtain infinitely many pairs of patterns
that are equivalent for alternating and reverse alternating permutations in Theorems~\ref{1221Alt}
and~\ref{213321AltRAlt}.

\subsection{The matrices $M(12)$ and $M(21)$ are shape-equivalent}
\label{sec:M12}

We will prove that $I_2 \sea{1} J_2$; this will be the analogue
of \cite[Lemma 1.11]{BW}, which proves that $I_2$ an $J_2$ are shape-Wilf equivalent.  First, we prove an explicit enumeration of $S_{\mathcal{Y}}(I_2)$
and $S_{\mathcal{Y}}(J_2)$.

\begin{prop}
\label{Shape2Spec}
For all AD-Young diagrams $\mathcal{Y} = (Y, A, D)$ such that $Y$ has $n$ rows, we have
\[
|S_\mathcal{Y}(I_2)| =
\begin{cases}
1 & \text{if } Y \supseteq (n,n-1,n-2,\ldots,1) \text{ and } A = \emptyset\\
0 & \text{ otherwise}\end{cases}\]
and
\[
|S_\mathcal{Y}(J_2)| =
\begin{cases}
1 & \text{if } Y \supseteq (n,n-1,n-2,\ldots,1) \text{ and } D = \emptyset\\
0 & \text{ otherwise}.\end{cases}\]
\end{prop}
The analogous result for ordinary Young diagrams is in the proof of \cite[Lemma 1.11]{BW}.
\begin{proof}
It is shown in the proof of \cite[Lemma 1.11]{BW},
that if $Y \not\supseteq (n,n-1,n-2,\ldots,1)$, then $Y$ has no valid transversals,
which implies that $|S_\mathcal{Y}(I_2)| = |S_\mathcal{Y}(J_2)|=0$.

To prove the first part, suppose that $x \in A$,
and suppose that $T = \{(i,t_i)\}$ is a valid transversal of $Y$.
Then, we have $t_x < t_{x+1}$ and the $x$th and
$x+1$st rows of $Y$ have the same length.  The restriction of $T$
to the $x$th and $x+1$st rows and the $b_x$th and $b_{x+1}$th columns
of $Y$ demonstrates that $T$ contains $I_2$.  Let $Y$ have $k$ rows.
If $A = \emptyset$, then as in the proof of \cite[Lemma 1.11]{BW},
$T = \{(1,n),(2,n-1),\ldots,(n,1)\}$ is the only element of $S_\mathcal{Y}(I_2)$.

To prove the second part, suppose that $x \in D$,
and suppose that $T = \{(i,t_i)\}$ is a valid transversal of $Y$.
Because the $x$th and $x+1$st rows of $Y$ have the same length, $T$
contains $J_2$.  Suppose that $D = \emptyset$,
and as in the proof of \cite[Lemma 1.11]{BW}, let $T = \{(i,b_i)\}$
be the transversal obtained by moving from the right column to the left column;
for column $y$, select for $T$ a square in the lowest unoccupied row with
at least $y$ squares.  See Figure~\ref{fig:J2Placement}
for an example.
\begin{figure}
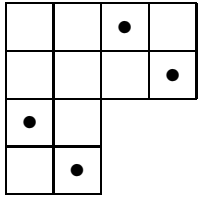

\center{
\begin{ytableau}
*(clear) & & \bullet &\\
& & & \bullet\\
\bullet &\\
& \bullet
\end{ytableau}}
\caption{Suppose that $Y = (4^2,2^2)$.  In the fourth column, we select $(2,4)$ for $T$; then, we select
$(1,3), (4,2), (3,1)$ in that order.}
\label{fig:J2Placement}
\end{figure}
Babson-West, in the proof of~\cite[Lemma 1.11]{BW}, prove that this process returns
the unique transversal of $Y$ that avoids $J_2$.  We prove that it is a valid transversal
of $\mathcal{Y}$.  If $x \in A$, then the $x$th and $(x+1)$st rows of $Y$ have the same
length, and let $m = \max\{b_x,b_{x+1}\}$.  When we selected a square for the $m$th column
of $Y$, the $x+1$st row of $Y$ was unoccupied, and by definition it has at least $m$ squares.
Thus, we have $b_{x+1} = m$ and $x$ is an ascent of $T$.  The fact that $T$ is a valid transversal
of $\mathcal{Y}$ follows.
\end{proof}

The following lemma is immediate from Proposition~\ref{Shape2Spec},
and the subsequent theorem follows easily from Proposition~\ref{AltPermToDiag},
the Extension Theorem~\ref{ShapeExtend},
and Lemma~\ref{Shape2}.

\begin{lemma}
\label{Shape2}
We have that $I_2 \sea{1} J_2$.
\end{lemma}

\begin{remark}
In Definition~\ref{ADYoung}, we require that if $i \in D$,
then the $i$th and $i+1$st rows of $Y$ to have the same length in order
for $(Y, A, D)$ to be an AD-Young diagram.
For the necessity of this condition, consider the AD-Young diagram $\mathcal{Y}'$ given by $Y' = (3^2,1)$,
$A = \{1\}$, $D = \{2\}$.  We have $|S_{\mathcal{Y}'}(M(12))| = 0$ but $|S_{\mathcal{Y}'}(M(21))| = 1$.
\end{remark}

\begin{thm}
\label{1221Alt}
For all $t > 2$ and all permutations $q$ of $[t]\setminus [2]$, the patterns $12q$ and $21q$
are equivalent for even- and odd-length alternating permutations.
\end{thm}

\begin{remark}
An alternate proof of Theorem~\ref{1221Alt} via an isomorphism of generating trees is possible; see \cite{Lewis2012, West1995}
for an exposition of generating trees.  However, such an isomorphism does not exist
in the case of Theorem~\ref{213321AltRAlt}, even in the alternating case.
\end{remark}

\subsection{The matrices $M(213)$ and $M(321)$ are shape-equivalent}
\label{sec:M213}

For a positive integer $r$, let $F_r$ denote the permutation matrix $M((r-1)(r-2)\cdots 1r)$.  We will prove
the following proposition.

\begin{prop}
\label{JF3Strong}
We have $F_3 \sesa{1} J_3$.
\end{prop}

The following theorem is immediate from Propositions~\ref{AltPermToDiag} and~\ref{JF3Strong},
and the Extension Theorem~\ref{ShapeExtend}.

\begin{thm}
\label{213321AltRAlt}
For all $t > 3$ and all permutations $q$ of $[t]\setminus [3]$, the patterns $213q$,
and $321q$ are equivalent for even- and odd-length reverse alternating permutations.
The patterns $123q$, $213q$,
and $321q$ are equivalent for even- and odd-length alternating permutations.
\end{thm}

\begin{remark}
Our proof that $|A_n(123q)| = |A_n(321q)|$ is essentially bijective,
but the bijection is not the restriction of Backelin-West-Xin \cite{BWX}'s bijection
to alternating permutations.
\end{remark}

Taking complements in the statement of Theorem~\ref{213321AltRAlt} for reverse alternating permutations
yields the following corollary.

\begin{cor}
\label{213321RAlt}
For all $t \ge 3$ and all permutations $q$ of $[t]$, the patterns $(t-1)t(t-2)q$
and $(t-2)(t-1)tq$ are equivalent for even- and odd-length alternating permutations.
\end{cor}

The idea of the proof of Proposition~\ref{JF3Strong} is to establish a bijection between $S_\mathcal{Y}(F_3)$
and $S_\mathcal{Y}(J_3)$ for $\mathcal{Y}$ a 1-alternating AD-Young diagram.
The bijection is based on the first proof of \cite[Proposition 3.1]{BWX}
in that it selects a copy of $J_3$ (resp. $F_3$) in a transversal
and removes it, but significant complications arise due to the required ascent
and descent sets.
We split into cases based on the locations of required ascents and descents
near the rightmost entry of the copy of $J_3$ (resp. $F_3$) and remove the copy in
a way that maintains required ascents and descents.  The fact that
rows of $Y$ have equal size at required ascents and descents of $\mathcal{Y}$
plays a critical role in showing that the replacement algorithm returns
a valid transversal of $\mathcal{Y}$.
Furthermore, as in \cite{BWX}, we restrict ourselves to so-called \emph{separable}
transversals (a class of transversals that contains any transversal
that avoids $F_3$ or $J_3$) because the two replacement procedures
are not inverse in general for non-separable transverals.  Due to the more elaborate process of removing
copies of $J_3$ and $F_3$, our notion of separability becomes slightly more technical than
the notion implicitly used in \cite{BWX}.  Before we state the bijection,
we must introduce the notation of cyclic shifts.

\begin{remark}
If $\mathcal{Y}$ has empty required ascent and required descent sets,
then our bijection agrees with that of \cite{BWX}.
\end{remark}

\subsubsection{Cyclic Shifts}
Fix a Young diagram $Y$ with $n$ columns for the entirety of this section and
let $T = \{(i,b_i)\}$ be a transversal of $Y$.
We define a function $\omega_{M}^{P}(T)$, for
sets $M, P \subseteq [1,n]$ with $m = \max M$ and $p = \max P$,
such that the $m$th row of $Y$ has at least $p$ squares.  Let
$i_1<i_2<\cdots<i_k$ denote the indices $i_j \in M$ with $b_{i_j} \in P$.
Take the index $j$ of $i_j$ modulo $k$, and let $\Gamma_{M}^{P}(T) = \{i_j \mid j \in [k]\}$.
Then, we define
\[\omega_{M}^{P}(T) = T \setminus \{(i_j,b_{i_j}) \mid i \in [k+1]\}
    \cup \{(i_j,b_{i_{j-1}}) \mid j \in [k]\}.\]
We define the function $\theta_{M}^{P}(T)$ taking the same arguments as $\Omega$, which will be
proven to be the inverse of $\omega$.  Let
$\{i_1<i_2<\cdots<i_k\} = \Gamma_M^P(T)$, and
we define
\[\theta_{M}^{P}(T) = T \setminus \{(i_j,b_{i_j}) \mid i \in [k+1]\}
    \cup \{(i_j,b_{i_{j+1}}) \mid j \in [k]\}.\]

Because the $m$th row of $Y$ has at least $p$ boxes, $\omega$ and $\theta$
return transversals of $Y$.  From the fact that
\[\Gamma_{M}^{P}(T) = \Gamma_{M}^{P} \left(\omega_{M}^{P}(T)\right),\]
it follows that that $\omega_{M}^{P}(\cdot)$ and $\theta_{M}^{P}(\cdot)$ are inverses.
Furthermore, if $i \notin M$, the position of the element of $T$ in the $i$th column of $Y$
is the same as that of $\omega_{M}^{P}(T)$ and in $\theta_{M}^{P}(T)$.  If $M \times P$
and $M' \times P'$ are disjoint, then it is clear that $\omega_M^P(\cdot)$ and
$\theta_M^P(\cdot)$ each commute with $\omega_{M'}^{P'}(\cdot)$ and $\theta_{M'}^{P'}(\cdot)$.
The functions $\omega$ and $\theta$ cyclically alter certain entries $b_i$ of a transversal
$T = \{(i,b_i)\}$.
\begin{figure}
\center{
\begin{ytableau}
*(clear) & & \bullet &  & & \\
& *(lightgray) & *(lightgray) & *(lightgray) & *(lightgray) & \bullet\\
& *(lightgray)\diamond & *(lightgray) & *(lightgray) \bullet & *(lightgray)\times \\
\bullet  & & & & \\
& *(lightgray) \bullet & *(lightgray) & *(lightgray)\times & *(lightgray)\diamond\\
& *(lightgray) \times & *(lightgray) & *(lightgray)\diamond & *(lightgray) \bullet\\
\end{ytableau}
}
\caption{Let $Y = (6^2,5^4)$ and let $T = \{(1,3),(2,6),(3,4),(4,1),(5,2),(6,5)\}$.
Then, $\Gamma_{[2,3] \cup [5,6]}^{[2,5]}(T) = \{3,5,6\},$ and thus $\omega_{[2,3] \cup [5,6]}^{[2,5]}(T)
= \{(1,3),(2,6),(3,5),(4,1),(5,4),(6,2)\}$ and $\theta_{[2,3] \cup [5,6]}^{[2,5]}(T)
= \{(1,3),(2,6),(3,2),(4,1),(5,4),(6,5)\}$.
Bullets mark elements of $T$ and crosses mark
elements of $\omega_{[2,3] \cup [5,6]}^{[2,5]}(T) \setminus T$, while diamonds
mark elements of $\theta_{[2,3] \cup [5,6]}^{[2,5]}(T) \setminus T$.}
\label{fig:CyclicShifts}
\end{figure}
See Figure~\ref{fig:CyclicShifts} for an example of cyclic shifts.

\subsubsection{Statement of the Bijection}

We first prove that $F_3 \sea{1} J_3$.
Fix a 1-alternating AD-Young diagram $\mathcal{Y} = (Y, A, D)$, and suppose that $Y$ has $n$ rows (columns).
We define inverse bijections $\Phi: S_\mathcal{Y}(F_3) \leftrightarrow S_\mathcal{Y}(J_3) : \Psi$.
We first define $\phi$ and $\psi$; $\Phi$ and $\Psi$ will be obtained by iterating $\phi$ and $\psi$, respectively.
Let $T = \{(i,b_i)\}$ be a transversal of $\mathcal{Y}$.  If $i < j < k \in [n]$,
then we say that $(a_1,a_2,a_3)$ is a copy of $J_3$ (resp. $F_3$) in $T$ if
$\{(a_i,b_{a_i}) \mid i \in [3]\}$ is a copy of $J_3$ (resp. $F_3$) in $T$.

Let $T = \{(i,b_i)\}$ be a transversal of $\mathcal{Y}$ that contains $J_3$.
Suppose that $(a_1,a_2,a_3)$ is a copy of $J_3$ in $T$.  We define auxiliary functions
$\phi_\ell^{(a_1,a_2,a_3)}(T)$ for $\ell \in [3]$ (the functions take arguments $(a_1,a_2,a_3)$ and $T$, and return
a priori transversals of the Young diagram $(n^n)$).
We define
\begin{align*}
\phi_1^{(a_1,a_2,a_3)}(T) &= \theta_{\{a_1,a_2,a_3\}}^{[1,b_{a_1}]}(T)\\
\phi_2^{(a_1,a_2,a_3)}(T) &= \omega_{\{a_1,a_3-1\}}^{[1,b_{a_1}]}(T)\\
\phi_3^{(a_1,a_2,a_3)}(T) &= \omega_{[1,a_1] \cup \{a_3+1\}}^{[b_{a_3},b_{a_1}]}\left(\omega_{[a_2,a_3]}^{[b_{a_3},b_{a_1}]}(T)\right).
\end{align*}
The operation $\phi_1$ is the one used by Backelin-West-Xin in their proof of
\cite[Proposition 3.1]{BWX}.

Let $U(T)$ denote the set of triples $a \in [n]^3$ that are copies of $J_3$ in $T$.
If $a = (a_1,a_2,a_3) \in U(T)$, then
define the $J$-type of $u$ in the following cases.
\begin{casework}
\item If ($a_3 - 1 \notin D$ or $b_{a_1} < b_{a_3-1}$) and $a_3 \notin A$,
    we say that $(a_1,a_2,a_3)$ is of $J$-type 1.
\item If $a_3 -1 \in D$ and $b_{a_3-1} < b_{a_1},$ we say that $(a_1,a_2,a_3)$ is of $J$-type 2.
\item If ($a_3 -1 \notin D$ or $b_{a_3-1} > b_{a_1}$), and $a_3 \in A$
    we say that $(a_1,a_2,a_3)$ is of $J$-type 3.
\end{casework}
See Figures~\ref{fig:phi1},~\ref{fig:phi2}, and~\ref{fig:phi3} for geometric
descriptions of the functions $\phi_\ell$.
\begin{figure}
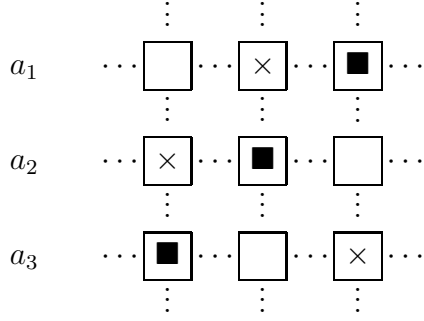

\center{
\begin{ytableau}
\none & \none & \none & \none[\vdots] & \none & \none[\vdots] & \none & \none[\vdots]\\
\none[a_1] & \none & \none[\cdots] & & \none[\cdots] &\times & \none[\cdots] &\blacksquare & \none[\cdots]\\
\none & \none & \none & \none[\vdots] & \none & \none[\vdots] & \none & \none[\vdots]\\
\none[a_2] & \none & \none[\cdots] &\times & \none[\cdots] &\blacksquare & \none[\cdots] & & \none[\cdots]\\
\none & \none & \none & \none[\vdots] & \none & \none[\vdots] & \none & \none[\vdots]\\
\none[a_3] & \none & \none[\cdots] &\blacksquare & \none[\cdots] & & \none[\cdots] &\times & \none[\cdots]\\
\none & \none & \none & \none[\vdots] & \none & \none[\vdots] & \none & \none[\vdots]\\
\end{ytableau}}
\caption{We show the effect of $\phi_1^{(a_1,a_2,a_3)}$ on a transversal $T$. Black boxes mark the
 selected elements of $T$ while crosses mark elements of $\phi_1^{(a_1,a_2,a_3)}(T) \setminus T$.}
\label{fig:phi1}
\end{figure}
\begin{figure}
\center{
\begin{ytableau}
\none & \none & \none & \none[\vdots] & \none & \none[\vdots] & \none & \none[\vdots] & \none & \none[\vdots] \\
\none[a_1] & \none & \none[\cdots] & & \none[\cdots] & & \none[\cdots] &\times & \none[\cdots] &\blacksquare & \none[\cdots]\\
\none & \none & \none & \none[\vdots] & \none & \none[\vdots] & \none & \none[\vdots] & \none & \none[\vdots] \\
\none[a_2] & \none & \none[\cdots] & & \none[\cdots] &\blacksquare & \none[\cdots] & & \none[\cdots] & & \none[\cdots]\\
\none & \none & \none & \none[\vdots] & \none & \none[\vdots] & \none & \none[\vdots] & \none & \none[\vdots] \\
\none[a_3-1] & \none & \none[\cdots] & & \none[\cdots] & & \none[\cdots] &\bullet& \none[\cdots] &\times& \none[\cdots]\\
\none[a_3] & \none & \none[\cdots] &\blacksquare & \none[\cdots] & & \none[\cdots] & & \none[\cdots]\\
\none & \none & \none & \none[\vdots] & \none & \none[\vdots] & \none & \none[\vdots] & \none & \none[\vdots] \\
\end{ytableau}}
\caption{We show the effect of $\phi_2^{(a_1,a_2,a_3)}$ on a transversal $T$. Black boxes mark the
 selected elements of $T$ and bullets mark other elements of $T$, while crosses mark elements of $\phi_1^{(a_1,a_2,a_3)}(T) \setminus T$.}
\label{fig:phi2}
\end{figure}
\begin{figure}
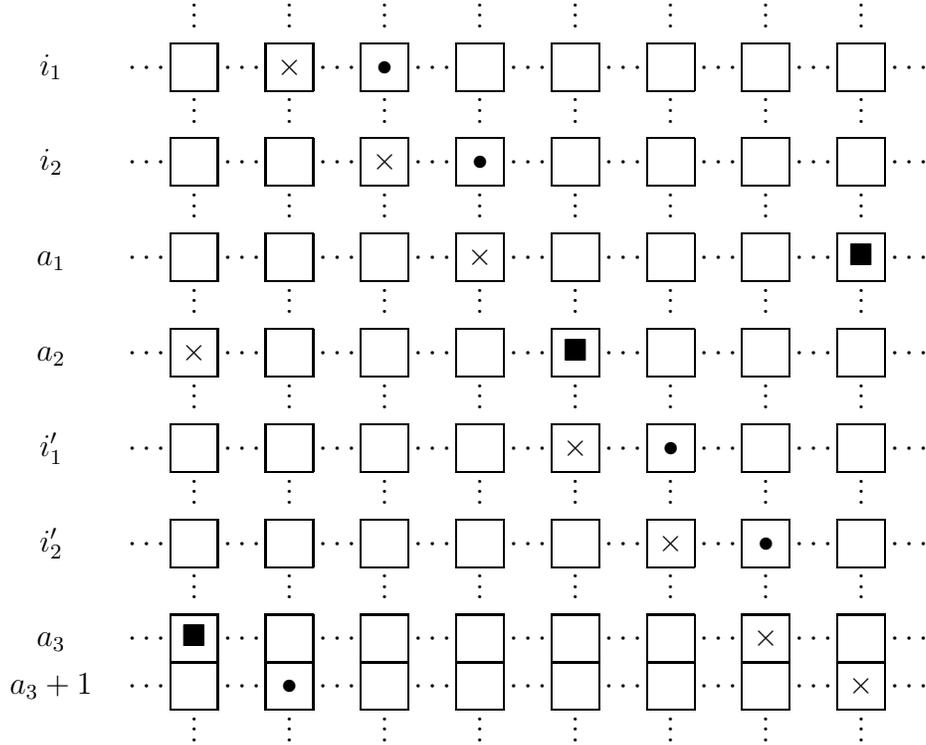

\center{
\begin{ytableau}
\none & \none & \none & \none[\vdots] & \none & \none[\vdots] & \none & \none[\vdots] & \none & \none[\vdots] & \none & \none[\vdots] & \none & \none[\vdots] & \none & \none[\vdots] & \none & \none[\vdots] & \none \\
\none[i_1] & \none & \none[\cdots] & & \none[\cdots] & \times & \none[\cdots] & \bullet & \none[\cdots] & & \none[\cdots] & & \none[\cdots] & & \none[\cdots] & & \none[\cdots] & & \none[\cdots] \\
\none & \none & \none & \none[\vdots] & \none & \none[\vdots] & \none & \none[\vdots] & \none & \none[\vdots] & \none & \none[\vdots] & \none & \none[\vdots] & \none & \none[\vdots] & \none & \none[\vdots] & \none \\
\none[i_2] & \none & \none[\cdots] & & \none[\cdots] & & \none[\cdots] & \times & \none[\cdots] & \bullet & \none[\cdots] & & \none[\cdots] & & \none[\cdots] & & \none[\cdots] & & \none[\cdots] \\
\none & \none & \none & \none[\vdots] & \none & \none[\vdots] & \none & \none[\vdots] & \none & \none[\vdots] & \none & \none[\vdots] & \none & \none[\vdots] & \none & \none[\vdots] & \none & \none[\vdots] & \none \\
\none[a_1] & \none & \none[\cdots] & & \none[\cdots] & & \none[\cdots] & & \none[\cdots] & \times & \none[\cdots] & & \none[\cdots] & & \none[\cdots] & & \none[\cdots] & \blacksquare & \none[\cdots] \\
\none & \none & \none & \none[\vdots] & \none & \none[\vdots] & \none & \none[\vdots] & \none & \none[\vdots] & \none & \none[\vdots] & \none & \none[\vdots] & \none & \none[\vdots] & \none & \none[\vdots] & \none \\
\none[a_2] & \none & \none[\cdots] & \times & \none[\cdots] & & \none[\cdots] & & \none[\cdots] & & \none[\cdots] & \blacksquare & \none[\cdots] & & \none[\cdots] & & \none[\cdots] & & \none[\cdots] \\
\none & \none & \none & \none[\vdots] & \none & \none[\vdots] & \none & \none[\vdots] & \none & \none[\vdots] & \none & \none[\vdots] & \none & \none[\vdots] & \none & \none[\vdots] & \none & \none[\vdots] & \none \\
\none[i'_1] & \none & \none[\cdots] & & \none[\cdots] & & \none[\cdots] & & \none[\cdots] & & \none[\cdots] & \times & \none[\cdots] & \bullet & \none[\cdots] & & \none[\cdots] & & \none[\cdots] \\
\none & \none & \none & \none[\vdots] & \none & \none[\vdots] & \none & \none[\vdots] & \none & \none[\vdots] & \none & \none[\vdots] & \none & \none[\vdots] & \none & \none[\vdots] & \none & \none[\vdots] & \none \\
\none[i'_2] & \none & \none[\cdots] & & \none[\cdots] & & \none[\cdots] & & \none[\cdots] & & \none[\cdots] & & \none[\cdots] & \times & \none[\cdots] & \bullet & \none[\cdots] & & \none[\cdots] \\
\none & \none & \none & \none[\vdots] & \none & \none[\vdots] & \none & \none[\vdots] & \none & \none[\vdots] & \none & \none[\vdots] & \none & \none[\vdots] & \none & \none[\vdots] & \none & \none[\vdots] & \none \\
\none[a_3] & \none & \none[\cdots] & \blacksquare & \none[\cdots] & & \none[\cdots] & & \none[\cdots] & & \none[\cdots] & & \none[\cdots] & & \none[\cdots] & \times & \none[\cdots] & & \none[\cdots] \\
\none[a_3+1] & \none & \none[\cdots] & & \none[\cdots] & \bullet & \none[\cdots] & & \none[\cdots] & & \none[\cdots] & & \none[\cdots] & & \none[\cdots] & & \none[\cdots] & \times & \none[\cdots] \\
\none & \none & \none & \none[\vdots] & \none & \none[\vdots] & \none & \none[\vdots] & \none & \none[\vdots] & \none & \none[\vdots] & \none & \none[\vdots] & \none & \none[\vdots] & \none & \none[\vdots] & \none \\
\end{ytableau}}
\caption{We show the effect of $\phi_3^{(a_1,a_2,a_3)}(T)$ on a transversal $T$. Suppose
that $\{i_1,i_2\} = \Gamma_{[1,a_1)}^{[b_{a_3},b_{a_1}]}(T)$, and
$\{i'_1,i'_2\} = \Gamma_{(a_2,a_3)}^{[b_{a_3},b_{a_1}]}(T)$.  Black boxes mark the
 selected elements of $T$ and bullets mark other elements of $T$, while crosses mark elements of $\phi_1^{(a_1,a_2,a_3)} \setminus T$.}
\label{fig:phi3}
\end{figure}

For $u = (u_1,u_2,u_3) \in \mathbb{N}^3$, let $\#(u) = (u_3,u_1,u_2)$.
Let $h_J(T)$ be the triple $a \in U(T)$ that minimizes
$\#(a)$ in the lexicographic order; this is exactly
the way in which a copy of $J_3$ is chosen to be removed in the proof
of \cite[Proposition 3.1]{BWX}.  If $h_J(T)$ is of $J$-type $t$,
let $\phi(T) = \phi_t^{h_J(T)}(T)$, and we say that $T$ is of $J$-type $t$.

We define the functions $\psi_\ell,$ which take the same arguments as the $\phi_\ell$
and return a priori transversals of $(n^n)$.
For $\ell \in [3]$, let
\begin{align*}
\psi_1^{(a_1,a_2,a_3)} &= \omega_{\{a_1,a_2,a_3\}}^{[1,b_{a_3}]}(T)\\
\psi_2^{(a_1,a_2,a_3)} &= \theta_{\{a_1,a_3\}}^{[1,b_{a_3}]}(T)\\
\psi_3^{(a_1,a_2,a_3)} &= \theta_{[a_2,a_3-1]}^{[b_{a_2},b_{a_3}]}\left(\theta_{[1,a_1] \cup \{a_3\}}^{[b_{a_2},b_{a_3}]}(T)\right).
\end{align*}
The operation $\psi_1$ is the one used by Backelin-West-Xin in their proof
of \cite[Proposition 3.1]{BWX}.

Let $V(T)$ denote the set of triples $a = (a_1,a_2,a_3) \in [n]^3$ that
are copies of $F_3$ in $T$ such that $a_3 \notin A$.
We convert copies of $F_3$ into where their copy of $J_3$ should be via a function $S(u)$ for $u \in V(T)$;
the value $S(a)$ is independent of the choice of $T$ such that $a \in V(T)$.
For $a = (a_1,a_2,a_3) \in V(T)$, define the $F$-type of $u$ in the
in the following cases.
\begin{casework}
\item If $a_3 - 1 \notin A$,
    let $S(a) = (a_3,a_1,a_2).$  We say that $a$ is of $F$-type 1.
\item If $a_3-1 \in A$ and $a_2 = a_3 - 1$,
    let $S(a) = (a_3+1, a_1, 0)$.
    We say that $a$ is of $F$-type 2.
\item If $a_3 - 1 \in A$ and $a_2 \not= a_3-1$, let $S(a) = (a_3-1,a_1,a_2)$.
We say that $a$ is of $F$-type 3.
\end{casework}
See Figures~\ref{fig:psi1},~\ref{fig:psi2}, and~\ref{fig:psi3} for geometric
descriptions of the functions $\psi_\ell$.
\begin{figure}
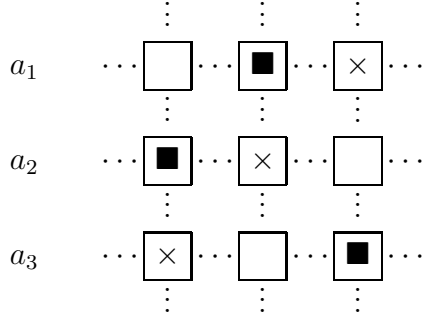

\center{
\begin{ytableau}
\none & \none & \none & \none[\vdots] & \none & \none[\vdots] & \none & \none[\vdots]\\
\none[a_1] & \none & \none[\cdots] & & \none[\cdots] &\blacksquare & \none[\cdots] &\times & \none[\cdots]\\
\none & \none & \none & \none[\vdots] & \none & \none[\vdots] & \none & \none[\vdots]\\
\none[a_2] & \none & \none[\cdots] &\blacksquare & \none[\cdots] &\times & \none[\cdots] & & \none[\cdots]\\
\none & \none & \none & \none[\vdots] & \none & \none[\vdots] & \none & \none[\vdots]\\
\none[a_3] & \none & \none[\cdots] &\times & \none[\cdots] & & \none[\cdots] &\blacksquare & \none[\cdots]\\
\none & \none & \none & \none[\vdots] & \none & \none[\vdots] & \none & \none[\vdots]\\
\end{ytableau}}
\caption{We show the effect of $\psi_1^{(a_1,a_2,a_3)}$ on a transversal $T$. Black boxes mark the
 selected elements of $T$ while crosses mark elements of $\psi_1^{(a_1,a_2,a_3)}(T) \setminus T$.}
\label{fig:psi1}
\end{figure}
\begin{figure}
\center{
\begin{ytableau}
\none & \none & \none & \none[\vdots] & \none & \none[\vdots] & \none & \none[\vdots]\\
\none[a_1] & \none & \none[\cdots] & & \none[\cdots] &\blacksquare & \none[\cdots] &\times & \none[\cdots]\\
\none & \none & \none & \none[\vdots] & \none & \none[\vdots] & \none & \none[\vdots]\\
\none[a_2] & \none & \none[\cdots] &\blacksquare & \none[\cdots] & & \none[\cdots] & & \none[\cdots]\\
\none[a_3] & \none & \none[\cdots] & & \none[\cdots] &\times & \none[\cdots] &\blacksquare & \none[\cdots]\\
\none & \none & \none & \none[\vdots] & \none & \none[\vdots] & \none & \none[\vdots]\\
\end{ytableau}}
\caption{We show the effect of $\psi_2^{(a_1,a_2,a_3)}$ on a transversal $T$. Black boxes mark the
 selected elements of $T$, while crosses mark elements of $\psi_2^{(a_1,a_2,a_3)}(T) \setminus T$.}
\label{fig:psi2}
\end{figure}
\begin{figure}
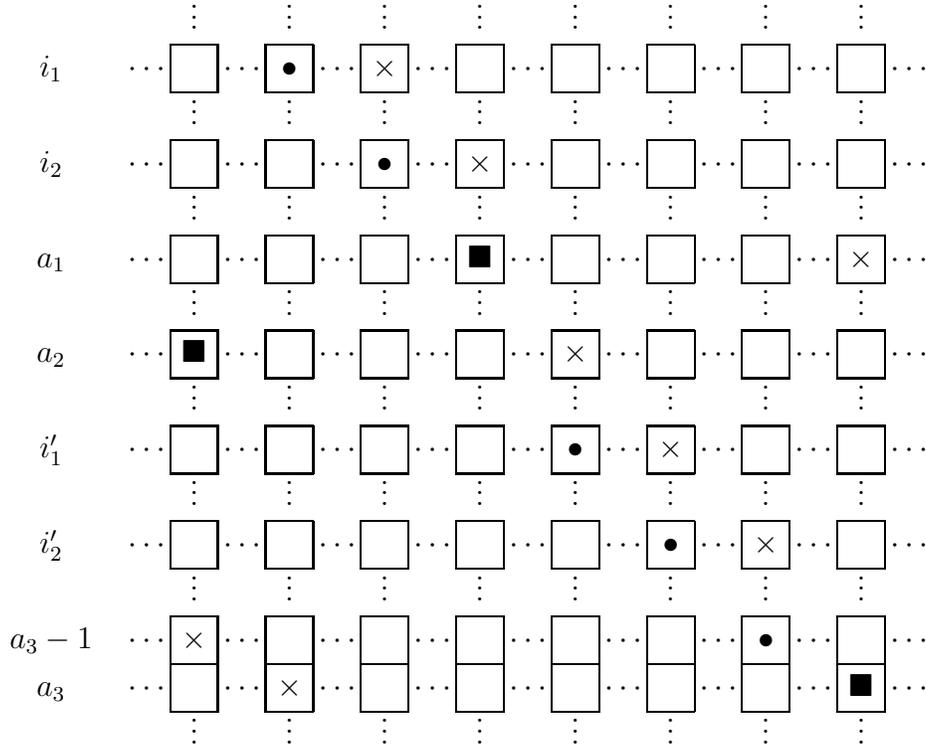

\center{
\begin{ytableau}
\none & \none & \none & \none[\vdots] & \none & \none[\vdots] & \none & \none[\vdots] & \none & \none[\vdots] & \none & \none[\vdots] & \none & \none[\vdots] & \none & \none[\vdots] & \none & \none[\vdots] & \none \\
\none[i_1] & \none & \none[\cdots] & & \none[\cdots] & \bullet & \none[\cdots] & \times & \none[\cdots] & & \none[\cdots] & & \none[\cdots] & & \none[\cdots] & & \none[\cdots] & & \none[\cdots] \\
\none & \none & \none & \none[\vdots] & \none & \none[\vdots] & \none & \none[\vdots] & \none & \none[\vdots] & \none & \none[\vdots] & \none & \none[\vdots] & \none & \none[\vdots] & \none & \none[\vdots] & \none \\
\none[i_2] & \none & \none[\cdots] & & \none[\cdots] & & \none[\cdots] & \bullet & \none[\cdots] & \times & \none[\cdots] & & \none[\cdots] & & \none[\cdots] & & \none[\cdots] & & \none[\cdots] \\
\none & \none & \none & \none[\vdots] & \none & \none[\vdots] & \none & \none[\vdots] & \none & \none[\vdots] & \none & \none[\vdots] & \none & \none[\vdots] & \none & \none[\vdots] & \none & \none[\vdots] & \none \\
\none[a_1] & \none & \none[\cdots] & & \none[\cdots] & & \none[\cdots] & & \none[\cdots] & \blacksquare & \none[\cdots] & & \none[\cdots] & & \none[\cdots] & & \none[\cdots] & \times & \none[\cdots] \\
\none & \none & \none & \none[\vdots] & \none & \none[\vdots] & \none & \none[\vdots] & \none & \none[\vdots] & \none & \none[\vdots] & \none & \none[\vdots] & \none & \none[\vdots] & \none & \none[\vdots] & \none \\
\none[a_2] & \none & \none[\cdots] & \blacksquare & \none[\cdots] & & \none[\cdots] & & \none[\cdots] & & \none[\cdots] & \times & \none[\cdots] & & \none[\cdots] & & \none[\cdots] & & \none[\cdots] \\
\none & \none & \none & \none[\vdots] & \none & \none[\vdots] & \none & \none[\vdots] & \none & \none[\vdots] & \none & \none[\vdots] & \none & \none[\vdots] & \none & \none[\vdots] & \none & \none[\vdots] & \none \\
\none[i'_1] & \none & \none[\cdots] & & \none[\cdots] & & \none[\cdots] & & \none[\cdots] & & \none[\cdots] & \bullet & \none[\cdots] & \times & \none[\cdots] & & \none[\cdots] & & \none[\cdots] \\
\none & \none & \none & \none[\vdots] & \none & \none[\vdots] & \none & \none[\vdots] & \none & \none[\vdots] & \none & \none[\vdots] & \none & \none[\vdots] & \none & \none[\vdots] & \none & \none[\vdots] & \none \\
\none[i'_2] & \none & \none[\cdots] & & \none[\cdots] & & \none[\cdots] & & \none[\cdots] & & \none[\cdots] & & \none[\cdots] & \bullet & \none[\cdots] & \times & \none[\cdots] & & \none[\cdots] \\
\none & \none & \none & \none[\vdots] & \none & \none[\vdots] & \none & \none[\vdots] & \none & \none[\vdots] & \none & \none[\vdots] & \none & \none[\vdots] & \none & \none[\vdots] & \none & \none[\vdots] & \none \\
\none[a_3-1] & \none & \none[\cdots] & \times & \none[\cdots] & & \none[\cdots] & & \none[\cdots] & & \none[\cdots] & & \none[\cdots] & & \none[\cdots] & \bullet & \none[\cdots] & & \none[\cdots] \\
\none[a_3] & \none & \none[\cdots] & & \none[\cdots] & \times & \none[\cdots] & & \none[\cdots] & & \none[\cdots] & & \none[\cdots] & & \none[\cdots] & & \none[\cdots] & \blacksquare & \none[\cdots] \\
\none & \none & \none & \none[\vdots] & \none & \none[\vdots] & \none & \none[\vdots] & \none & \none[\vdots] & \none & \none[\vdots] & \none & \none[\vdots] & \none & \none[\vdots] & \none & \none[\vdots] & \none \\
\end{ytableau}}
\caption{We show the effect of $\psi_3^{(a_1,a_2,a_3)}$ on a transversal $T$. Suppose
that $\{i_1,i_2\} = \Gamma_{[1,a_1)}^{[b_{a_2},b_{a_3}]}(T)$, and
$\{i'_1,i'_2\} = \Gamma_{(a_2,a_3)}^{[b_{a_2},b_{a_3}]}(T)$.  Black boxes mark the
 selected elements of $T$ and bullets mark other elements of $T$, while crosses mark elements of $\psi_3^{(a_1,a_2,a_3)}(T) \setminus T$.}
\label{fig:psi3}
\end{figure}

For $u, u' \in V(T)$, we write $u \triangleright u'$ if $S(u) \ge S(u')$ in the
lexicographic order.  We will select a copy of $F_3$ to eliminate by
treating $\triangleright$ as a total order on $V(T)$; to do so, we require
the following lemma.
\begin{lemma}
$S$ is injective, and thus $\, \triangleright$ is a total order on $V(T)$.
\end{lemma}
\begin{proof}
Suppose that $S(u) = (d_3,d_1,d_2)$.  If $d_2 = 0$, then $u$ and $u'$ are of $F$-type
2, and thus $u = (d_1,d_3-2,d_3-1)$.  Otherwise, if $d_3 \in A$, then $u$ and $u'$
are of $F$-type 3, and hence $u = (d_1,d_2,d_3+1)$.  Otherwise, we have $d_2 \not= 0$
and $d_3 \notin A$, which implies that $u$ is of $F$-type 1 and $u = (d_1,d_2,d_3)$.  Therefore,
$S$ has a left inverse and is thus injective.  The fact that $\, \triangleright$ is a total order
follows.
\end{proof}
Let $h_F(T)$ denote the maximum of $V(T)$ with respect to the restriction of $\, \triangleright$.
If $A = D = \emptyset$, then $h_F$ agrees with Backelin-West-Xin's selection of a copy of $F_3$ to remove
in their proof of \cite[Proposition 3.1]{BWX}, but in general, $h_F$ differs
from Backelin-West-Xin's selection.
If $h_F(T)$ is of $F$-type $t$,
let $\psi(T) = \psi_t^{h_F(T)}(T)$, and we say that $T$ is of $F$-type $t$.

A transversal $T$ is said to be \emph{separable}
if it satisfies the property that if $u \in U(T)$ and $u' = S\left(V(T)\right)$,
then $\#(u) \ge u'$ in the lexicographic order.  Any
element of $S_\mathcal{Y}(J_3)$ (resp. $S_{\mathcal{Y}}(F_3)$) is separable,
as $U(T)$ (resp. $V(T)$) is empty.  We restrict our attention
to separable transversals.

The critical properties of $\phi$ and $\psi$ are the following two propositions.

\begin{prop}
\label{FullPhi}
If $T$ is a separable valid transversal of $\mathcal{Y}$ that
contains $J_3$, then $\phi(T)$ is a separable valid transversal of $\mathcal{Y}$
and we have $\psi(\phi(T)) = T$.  Furthermore, if $T = \{(i,b_i)\}$ and $\phi(T) = \{(i, c_i)\}$,
then we have $(b_1,b_2,\ldots, b_n) > (c_1,c_2,\cdots,c_n)$ in the lexicographic order.
\end{prop}

\begin{prop}
\label{FullPsi}
If $T$ is a separable valid transversal of $\mathcal{Y}$ that
contains $F_3$, then $\psi(T)$ is a separable valid transversal of $\mathcal{Y}$
and we have $\phi(\psi(T)) = T$.  Furthermore, if $T = \{(i,b_i)\}$ and $\psi(T) = \{(i, c_i)\}$,
then we have $(b_1,b_2,\ldots, b_n) < (c_1,c_2,\cdots,c_n)$ in the lexicographic order.
\end{prop}

We defer the proofs of Propositions~\ref{FullPhi}
and~\ref{FullPsi} to Appendices~\ref{sec:PropFullPhi} and~\ref{sec:PropFullPsi}, respectively.
To complete the proof of Proposition~\ref{JF3Strong},
we require a simple technical lemma to extend from alternating
to semialternating AD-Young diagrams.  It follows immediately from the definitions
of $\phi$ and $\psi$, and so we omit the proof.

\begin{lemma}
\label{AddA1}
Let $T = \{(i,b_i)\}$ be a valid transversal of $\mathcal{Y}$ with $b_1 = 1$.
If $T$ contains $J_3$ and $\phi(T) = \{(i,c_i)\}$, then we have $c_1 = 1$.
If $T$ contains $F_3$ and $\psi(T) = \{(i,c_i)\}$, then we have $c_1 = 1$.
\end{lemma}

\begin{proof}[Proof of Proposition~\ref{JF3Strong} assuming Propositions~\ref{FullPhi} and~\ref{FullPsi}]
Let $\mathcal{Y}$ be a 1-alternating AD-Young diagram and let $T \in S_\mathcal{Y}(F_3)$; because
$V(T) = \emptyset$, the transversal $T$ is separable.  Let
$M$ be the smallest integer $m$ such that $\phi^m(T) \in S_\mathcal{Y}(J_3)$; such an integer
$M$ exists because applying $\phi$ repeatedly yields valid transversals of $\mathcal{Y}$,
because $(b_1,b_2,\ldots,b_n)$ is monotonically decreasing in the lexicographic order,
the sequence $\{\phi^m(T)\}_{m \in \mathbb{N}}$ must eventually terminate.  Then, let $\Phi(T) = \phi^m(T)$,
and $\Phi$ defines a function from $S_\mathcal{Y}(F_3)$ to $S_\mathcal{Y}(J_3)$.  Define $\Psi:S_\mathcal{Y}(J_3)
\rightarrow S_\mathcal{Y}(F_3)$ analogously.  We claim that $\Phi$ and $\Psi$ are inverses.  Let $T \in S_\mathcal{Y}(F_3)$, and suppose that $\Phi(T) = \phi^m(T)$.  By $m$ applications of Proposition~\ref{FullPhi},
we have $\psi^m(\Phi(T)) = T$, and because $\psi(W)$ is defined only for $W \notin S_\mathcal{Y}(F_3)$,
we have $\Psi(\Phi(T)) = \psi^m(\Phi(T)) = T$.  A similar argument using Proposition~\ref{FullPsi}
demonstrates that $\Phi(\Psi(T)) = T$ for all $T \in S_\mathcal{Y}(J_3)$, and therefore $\Phi$ and $\Psi$
are inverse bijections.

Suppose that $\mathcal{Y'} = (Y',A',D')$ is a 1-semialternating AD-Young diagram, and let
$Y' = (Y'_1,Y'_2,\ldots,Y'_n)$.  If $1 \notin D'$,
then it is clear that $\left|S_\mathcal{Y}(F_3)\right| = \left|S_\mathcal{Y}(J_3)\right|.$  If
$1 \in D$, then let $Y = (Y'_1+1,Y'_1+1,Y'_2+1,Y'_3+1,\ldots,Y'_n+1$, $A = \{1\} \cup (A' + 1)$,
and let $D' = D + 1$.  The AD-Young diagram $\mathcal{Y} = (Y, A, D)$ is 1-alternating.
For $T' = \{(i,b'_i)\}$ a valid transversal of $\mathcal{Y'}$, let $\alpha(T) = \{(1,1)\} \cup \{(i+1,b'_i+1) \mid i \in [n]\}$; it is clear that $\alpha(T)$ is a valid transversal of $\mathcal{Y}$, and that $\alpha$ is injective.
Furthermore, if $T' \in S_{\mathcal{Y}'}(F_3)$ (resp. $S_{\mathcal{Y}'}(J_3)$), then
$\alpha(T') \in S_\mathcal{Y}(F_3)$ (resp. $S_\mathcal{Y}(F_3)$) because $(1,1)$ cannot be an element
of any copy of $F_3$ (resp. $J_3$) in $\alpha(T')$.
Define $\Phi' = \alpha^{-1} \circ \Phi \circ \alpha$.  Lemma~\ref{AddA1}
implies that $\Phi$ sends the range of $\alpha$ to the range of $\alpha$.
Together with the fact that $\alpha$ is injective, it follows that $\Phi'(T)$ is defined (and well-defined) for all $T \in S_\mathcal{Y}(F_3)$.  It is clear that
$\Phi'$ sends $S_{\mathcal{Y}'}(F_3)$ to $S_{\mathcal{Y}'}(J_3)$.  We define $\Psi'$
analogously.  Because $\Phi$ and $\Psi$ are inverses, so are $\Phi'$ and $\Psi'$. Hence, we have
that $\left|S_{\mathcal{Y}'}(F_3)\right| = \left|S_{\mathcal{Y}'}(J_3)\right|$, and the fact
that $F_3 \sesa{1} J_3$ follows.
\end{proof}

\subsection{Applications of shape-equivalence to equivalences of short patterns}
For a permutation $w = w_1w_2\cdots w_n,$
we define its reverse $w^r$ by $w_nw_{n-1}\cdots w_1$.
Because reversal is an involution on odd-length alternating
permutations, we have that $w \oddeq w^r$
for all $w$, and likewise we have $w \eveneq w^{rc}$.  Such
equivalences are called \emph{trivial equivalences.}
We consider non-trivial equivalences among patterns of length 4 and 5.

By Theorem~\ref{213321AltRAlt}, we have $1234 \oddeq 2134 \oddeq 3214$.
By Theorem~\ref{1221Alt}, we have $2143 \oddeq 1243$, which by reversal
is equivalent to $3421$, which is in turn equivalent to
$2341$ by Corollary~\ref{213321RAlt}.  These equivalences
constitute all possible equivalences for odd-length
alternating permutations among patterns of length 4 due
to the data of \cite{Lewis2012},
thereby rederiving results of \cite{Lewis2012, XuYan}.
Similar logic yields that
$1234 \eveneq 3214 \eveneq 2134 \eveneq 2143$ and $2341 \eveneq 3421$,
which recovers results of \cite{Bona1221, Lewis2012, XuYan}.

For patterns of length 5, we settled all possible equivalences
except for $23451 \oddeq 34521$, $43215 \oddeq 32145$,
and $32145 \eveneq 43215 \eveneq 23451 \eveneq 34521$;
this is 9 out of 11 possible equivalences for odd-length alternating permutations
and 9 out of 12 possible equivalences for even-length alternating permutations.
Except for $12345 \oddeq 21345$ and $12345 \eveneq 21345,$
which are proven in \cite{Bona1221}, the equivalences among
patterns of length 5 are new.
Brute-force enumerations that describe all possible nontrivial equivalences
among length 5 patterns are given in \cite{Lewis2012}.

For patterns of length 6, we described all possible nontrivial equivalences
for both odd-length and even-length alternating permutations by
brute-force enumeration. We present the list of possible equivalences
in Appendix~\ref{sec:enumerate}.  Theorems~\ref{1221Alt} and ~\ref{213321AltRAlt}
imply 35 out of 39 possible nontrivial equivalences for odd-length alternating permutations
among patterns of length 6, and 35 out of 45 possible nontrivial equivalences for even-length
alternating permutations.  Combinatorial blowup precludes the thorough examination
of equivalences between patterns of length 7.

\section{The matrix $J_k$ and the difference between 1-alternating and 2-alternating AD-Young diagrams}
\label{sec:Nonequivalences}

An interesting phenomenon is that fact that $J_3 \sesa{1} F_3$, but $321$ and $213$
are not equivalent for even-length alternating permutations.  Proposition~\ref{AltPermToDiag}
states that if $M(p) \sesa{2} M(q)$ implies that $p \eveneq q$, but equivalence for 2-alternating
AD-Young diagrams is stronger than equivalence for 1-alternating AD-Young diagrams.  We demonstrate
that $J_k$ is not equivalent $F_k$ for 2-alternating AD-Young
diagrams by proving a stronger statement: Corollary~\ref{decreasing}.  To do so, we give a coarse
method to prove non-equivalences.

Given a positive integer $k$ and a permutation $p \in S_k$, let
$d(p) = \{i \in [k-1] \mid p_{i-1} > p_i > p_{i+1} \text{ or } p_{i-1} < p_i < p_{i-1}\},$
where we let $p_0 = \infty$.  Call $d(p)$ the \emph{doubling set of } $p$, and call
$|d(p)|$ the \emph{doubling number} of $p$.  The doubling number of a permutation
is the number of double descents and double ascents of that permutation.
Furthermore, it is a measure of how far a permutation is from alternating;
alternating permutations are exactly those permutations with a doubling number of 0.

\begin{thm}
\label{Doubling}
Given a permutation $p$ of length $k$ with doubling number $t$, the length of the shortest
alternating permutation that contains $p$ is $k+t$.
\end{thm}

Theorem~\ref{Doubling} is immediate from Lemmata~\ref{NoSmall} and~\ref{Construction}.
Lemma~\ref{QuickFact} is a useful fact that will be used in the proof of Lemma~\ref{Construction}.

\begin{lemma}
\label{NoSmall}
If $p$ is a permutation of length $k$ with doubling number $t$ and $w$ is an alternating
permutation of length less than $k+t$, then $w$ does not contain $p$.
\end{lemma}
\begin{proof}
Assume for sake of contradiction that $w$ contains $p$ and let $w_{i_1}w_{i_2}\cdots w_{i_k}$
be a subsequence of $w$ that is order-isomorphic to $p$.  Let $i_0 = 0$.  For every $j \in [k]$, we have
$i_{j} - i_{j-1} \ge 1$.  However, if $j \in d(p)$ for all $\ell \le j \le m$, we have $i_{m+1} - i_{\ell-1} \ge 2(m - \ell) + 1$.  Suppose that $d(p) = \cup_{1 \le x \le X} \{\ell_x,\ell_x+1,\ldots,m_x\}$ with
$1 \le \ell_x \le m_x < \ell_{x+1} - 1$ for all $x$.  Adding the inequalities
in the appropriate fashion, we have
\begin{align*}
n &\ge i_k = i_k - i_0 = \sum_{j\in [k] \setminus d(p) \setminus \{m_1+1,m_2+1,\ldots,m_X+1\}} (i_j - i_{j-1})
+ \sum_{x=1}^X (i_{m_x+1} - i_{\ell_x - 1})\\
&\ge k - |d(p)| - X + \sum_{x=1}^X (2(m_x - \ell_x) + 1) = k - |d(p)| - X + 2|d(p)| + X = k+t,\end{align*}
as desired.
\end{proof}

\begin{lemma}
\label{QuickFact}
Let $p$ be a permutation of length $k$.  Then, for all $1 \le i < k$, we have $p_i < p_{i+1}$
if and only if $i + |d(p) \cap [i-1]|$ is odd.
\end{lemma}
\begin{proof}
For $1 \le i < k$, let $f(i) = i + |d(p) \cap [i-1]|$.  It is clear that
\[f(i) - f(i-1) = 1 + \mathbb{I}[i-1 \in d(p)].\]
We proceed by induction on $i$ for a fixed permutation $p$.  The base case $i = 1$
is obvious.  Assume the result for $i = l$, and we will prove it for $i = l+1$.  If $p_ip_{i+1}$
and $p_{i+1}p_{i+2}$ are either both ascents or both descents, we have that $i-1 \in d(p)$, which yields
that $f(i) = f(i-1) + 2$.  The result follows by the induction hypothesis.  If one of $p_ip_{i+1}$
and $p_{i+1}p_{i+2}$ is an ascent and the other is a descent, we have $f(i) = f(i-1) + 1$ and the
result follows by the induction hypothesis. This completes the proof of the inductive step,
and the induction is complete.
\end{proof}

\begin{lemma}
\label{Construction}
For every permutation $p$ of length $k$ with doubling number $t$, there is an alternating permutation
$w$ of length $k + t$ that contains $p$.
\end{lemma}
The idea of the proof is that we place a $p$ as a subsequence of $w$ with consecutive terms
of the subsequence consecutive in $w$ when possible, skipping entries of $w$ only when we have
elements of $d(p)$.
\begin{proof}
Firstly, we will define a function $f: [k] \rightarrow [k+t]$ that places $p$ into $w$;
let $f(m) = m + |d(p) \cap [m-1]|$ as in Lemma~\ref{QuickFact}.
It is clear that $f(m) \in [k+t]$ for all $m \in [k]$.  Let $S = f([k])$, let
$T_0 = ([k+t]\setminus S) \cap 2\mathbb{Z}$, and let $T_1 = ([k+t]\setminus S) \cap (2\mathbb{Z}+1)$.
The sets $T_0$ and $T_1$ are places where we need to add ``filler values"
that will not be part of our subsequence of $w$ order-isomorphic to $p$.
Define the permutation $w \in S_{k+t}$ as
\[w_i = \begin{cases}
p_{f^{-1}(i)} + |T_0| & \text{ if } i \in S\\
|T_0 \cap [i]| & \text{ if } i \in T_0\\
k + t - |T_1 \cap [i-1]| & \text{ if } i \in T_1.
\end{cases}\]
It is clear that the restriction of $w$
to $S$ forms a bijection between $S$ and $[|T_0| + k] \setminus T_0$, the restriction of $w$
to $T_0$ forms a bijection between $T_0$ and $[T_0]$,
and the restriction of $w$ to $T_1$ forms a bijection between $T_1$ and $[k+t] \setminus [k+t-|T_1|]$.
Hence, $w$ is a permutation of $[k]$.

To prove that $w$ is alternating, we consider consecutive entries of $w$.
If $i \in T_0$ and $i < k+t$, we have $w_i \le |T_0| < w_{i+1}$ because $i+1 \notin T_0$;
if $i \in T_0$ and $i > 1$, we have $w_i \le |T_0| < w_{i-1}$ because $i-1 \notin T_0$.
If $i \in T_1$, then we have $w_i > k + t - |T_1| \ge w_{i-1}$ because $i-1 \notin T_1$; if $i \in T_1$
and $i < k+t$, we have $w_i > k + t - |T_1| \ge w_{i+1}$ because $i+1 \notin T_1$.  Hence, it suffices
to consider the case in which $i \in [k+t-1]$ such that $i,i+1 \in S$.  Then, we have
$w_i < w_{i+1}$ if and only if $p_{f^{-1}(i)} < p_{f^{-1}(i+1)}$, which happens if and only if $i$
is odd by Lemma~\ref{QuickFact}.  It follows that $w$ is alternating, as desired.  The subsequence
$w_{f(1)}w_{f(2)}\cdots w_{f(k)}$ of $w$ demonstrates that $w$ contains $p$, so the proof of the proposition
is complete.
\end{proof}

We give a useful corollary that follows immediately from Theorem~\ref{Doubling}, and after that
we obtain the desired result.

\begin{cor}
\label{GoodFormOfMain}
Suppose that a pattern $p$ of length $k$ and a pattern $q$ of length $k'$ have doubling numbers
$t$ and $t'$, respectively.  If $\left\lceil \frac{k+t}{2} \right\rceil \not= \left\lceil \frac{k'+t'}{2}\right\rceil$,
then $p$ and $q$ are not equivalent for even-length alternating permutations.  If
$\left\lceil \frac{k+t-1}{2}\right\rceil \not= \left\lceil\frac{k'+t'-1}{2}\right\rceil$,
then $p$ and $q$ are not equivalent for odd-length alternating permutations.
\end{cor}
\begin{proof}
Firstly, suppose that $\left\lceil \frac{k+t}{2} \right\rceil \not= \left\lceil \frac{k'+t'}{2}\right\rceil$.
Assume without loss of generality that $\left\lceil \frac{k+t}{2} \right\rceil > \left\lceil \frac{k'+t'}{2}\right\rceil$. Let $n = 22\left\lceil \frac{k'+t'}{2}\right\rceil$.
Then, by Theorem~\ref{Doubling} and because $k'+t' \le n < k+t,$
we have
\[A_{n}(q) \subsetneq A_n = A_{n}(p).\]
It follows that $p$ and $q$ are not equivalent for even-length alternating permutations.

Secondly, suppose that $\left\lceil \frac{k+t-1}{2}\right\rceil \not= \left\lceil\frac{k'+t'-1}{2}\right\rceil$.
Assume without loss of generality that $\left\lceil \frac{k+t-1}{2}\right\rceil > \left\lceil\frac{k'+t'-1}{2}\right\rceil$.
Let $n = 2\left\lceil\frac{k'+t'-1}{2}\right\rceil+1$.
Then, by Theorem~\ref{Doubling} and because $k'+t' \le n < k + t$,
we have
\[A_{n}(q) \subsetneq A_n = A_n(p),\]
and it follows that $p$ and $q$ are not equivalent for odd-length alternating permutations.

\end{proof}

\begin{cor}
\label{decreasing}
For all positive integers $k$, the pattern $k(k-1)(k-2)\cdots1$ is not
equivalent to any other pattern of length $k$ for even-length
alternating permutations.  Furthermore, the permutation matrix $J_k$
is not equivalent to any other permutation matrix for 2-alternating AD-Young diagrams.
\end{cor}
\begin{proof}
Let $\nu_k = k(k-1)(k-2)\cdots1$ and let $q \in S_k \setminus \{\nu_k\}$.
It is clear that we have $|d(\nu_k)| = k-1$ and $|d(q)| < k-1$.  Because
\[\left\lceil\frac{k + k - 1}{2}\right\rceil = k > \left\lceil\frac{k + |d(q)|}{2}\right\rceil,\]
$\nu_k$ and $q$ are not equivalent for even-length alternating permutations by Corollary~\ref{GoodFormOfMain}.
The second claim follows by the contrapositive of Proposition~\ref{AltPermToDiag}.
\end{proof}

\section{Generalized alternating permutations}
\label{Generalized alternating permutations}

Throughout this section let $p = p_1p_2\cdots p_n$ be a permutation of length $n.$
Similarly, let $q$ = $q_1 q_2 q_3 \cdots q_b$ be a pattern of length $b$. In \cite{Lewis2012}, an operation called extension was used to recursively generate pattern-avoiding permutations of length $n+1$ from such permutations of length $n$. The procedure itself involved appending a new value to the end of a permutation. However, in the context of permutation of descent type $k$, this procedure restricts us to only extending values $v \geq p_n$.  We require more flexibility in choosing
which values to add, so we define a new method to add a value.
\begin{definition}
Let $p \in S_n$ be a permutation of descent type $k$ and let $v \in [n+1]$.
Define $v \mapsto p$, the \emph{injection} of $v$ into $p$ as follwos: we first increment all values of $p$
that are greater than or equal to $v$ and then append $v$ to get a permutation
$p'$.  Then if $p$ had an incomplete final row, we rearrange the elements of the
final row of $p'$ to be in increasing order.  If $p$ had a complete final row and $v \le p_n$,
we simply define $v \mapsto p$ as $p'$.  However, if $v > p_n$, we swap the last two entries
of $p'$.If $w = v \mapsto p$ for some $v$, then we say that $w$ is a \emph{child} of $v$ and
$v$ is a \emph{parent} of $w$.  However, permutations need not have a unique parent.
\end{definition}
\begin{eg}
Consider the permutation $35624718$ of descent type $k = 3$. We have $4 \mapsto 35624718 = 367258149,$ because $p' = 36725819$, and appending 4
to $p'$ gives $367258194$.  Rearranging the final row then yields $367258149$.
\end{eg}

We omit the proof that if $p$ has descent type $k$, then every child of $p$ also descent type $k$.
It is clear, however, that every child of $p$ contains $p$, and therefore
if a child of $p$ avoids a pattern $q$, then so does $p$.

In Section~\ref{nonstrict case}, the primary nontrivial result is that $|D^k_n(q)| \leq |D^k_{n+1}(q)|$ for all patterns $q$ except for the trivial counterexample of the identity permutation when $k \geq b$. Additionally, in Section~\ref{equality case}, we investigate \emph{repetitive patterns}, patterns which are characterized by pattern avoidance of a particular triplet of patterns. What is especially interesting about these patterns, as we show in Section~\ref{equality case}, is that $|D^k_n(q)| = |D^k_{n+1}(q)|$ for particular values of $n$ and repetitive patterns $q$. In conjunction with this, for all non-repetitive patterns, in Section~\ref{strict case} we show that $|D^k_n(q)| < |D^k_{n+1}(q)|$. Since $D_n^k(q)$ is trivial to understand when $n < k$, as $|D_n^k(q)| = 1$ (or 0 for short idenity permutations), it shall be assumed throughout this section $n \geq k$. For similar reasons, since $q = 12, 21$ are trivial cases as well, we shall assume that $b \geq 3$.

\subsection{Nonstrict Case: $|D^k_n(q)| \leq |D^k_{n+1}(q)|$}
\label{nonstrict case}

We shall show that $|D^k_n(q)| \leq |D^k_{n+1}(q)|$; this is fairly intuitive for as we consider longer length permutations, we would expect more permutations to avoid the fixed pattern. We prove the following theorem.

\begin{thm}
	\label{nonstrict case theorem}
Let $k \le n$ be positive integers and let $q \notin \{21, 1, 12, 123, \ldots, 123\cdots k\}$.
Then, there is an injection
$f: D^k_n(q) \hookrightarrow D^k_{n+1}(q)$ such that for all $p$,
$f(p)$ is a child of $p$. In particular, we have $|D^k_n(q)| \leq |D^k_{n+1}(q)|$.
\end{thm}

In other words, for each parent, we are choosing a different child.
In the case where $q = 123\cdots b $ with $b \le k,$ $k$ is so large that any
sufficiently long permutation with descent type $k$ contains $q$ when $n \geq b$, since the first $k$ values of the permutation are in strictly increasing order.

\begin{proof}
Fix $q$ satisfying the theorem conditions.

Let $p \in D_n^k(q)$. Define a \emph{consecutive block} to be a subset of consecutive cells that are consecutive in value as well; i.e. $a_i < a_{i+1} < \cdots < a_j$ for $i < i+1 < \cdots < j$, and $a_{s} - a_{s-1} = 1$ when $i < s \leq j$. We call the value $a_{j}$ the \emph{anchor} of the consecutive block. We define the \emph{block function} $B(q)$ to return the length of the consecutive block anchored at $q_b$. (Note that this function is only defined for patterns with $q_b = b$.) Note that if the pattern is the identity pattern, then the function returns $b$.
The following algorithm defines $f(p)$.

If $q_b = b$:

\hspace{1 cm} If $n = km + s$, $0 \leq s < B(q)$, inject 1.

\hspace{1 cm} If $n = km + s$, $B(q) \leq s < k$, inject $p_{n-B(q)+1}$.

If $q_b \neq b$:

\hspace{1 cm} If $n = km + s$, $0 < s < k$, inject $n+1$.

\hspace{1 cm} If $n = km$:

\hspace{2 cm} If $q_b = 1$:

\hspace{3 cm} If $q_{b-1} = 2$, inject $n+1$.

\hspace{3 cm} If $q_{b-1} \neq 2$, inject $p_{km}$.

\hspace{2 cm} If $q_b \neq 1$, inject 1.

\begin{eg}
Let $q = 2134$. Then $B(q) = B(2134) = 2$. Consider permutation $p = 23514$ with $k = 3$. Then, by the algorithm, the child permutation $p' = f(p) = 346125$ since we inject the value $p_{4} = 1$. Similarly by the algorithm, $f(p') = 4572361$ as we inject the value $1$ into $p'$.
\end{eg}

We now prove $f$ is injective via casework.

\noindent{\bf Case 1: $q_b = b$}
\begin{itemize}

\item[$n = km + s$, $0 \leq s < B(q)$]\

We claim that the injection of the value 1 into the final row will result in a $p' \in D^k_{n+1}(q)$. Since 1 is the smallest value in the permutation, it must be the first value in the final row. Since $n < km + B(q)$, there exist at most $B(q)-1$ values to the right of 1 in $p'$. However, by defintion of $B(q)$, there are at least $B(q)$ values to the right of 1 in $q$. Thus, $p'$ avoids $q$ as desired.

\item[$n = km + s$, $B(q)\leq s < k $]\

Let $p_{n-B(q)+1}$ = $f$. We claim that in this subcase, the injection of the value $f$ into the final row will result in a valid $p' \in D^k_{n+1}(q)$. Assume, for the sake of contradiction, $p'$ contains $q$. As a result of the injection, $p_{n-B(q)+2} = f+1$. If only $p_{n-B(q)+1}$ is part of the subsequence, this is a contradiction because then the same
subsequence is in p, indicating that p contains q. Similarly, if only $p_{n-B(q)+2}$ is part of the subsequence, since $p_{n-B(q)+1}$ and $p_{n-B(q)+2}$ are consecutive integers, we can simply swap the corresponding position of $p_{n-B(q)+2}$ for $p_{n-B(q)+1}$, resulting in another contradiction. So, both of $p_{n-B(q)+1}$ or $p_{n-B(q)+2}$ must be in the subsequence that is order-isomorphic to $q$. There remain two nontrivival cases to consider: if both $p_{n-B(q)+1}$ and $p_{n-B(q)+2}$ are isomorphic to values part of the consecutive block or if both $p_{n-B(q)+1}$ and $p_{n-B(q)+2}$ are isomorphic to values not part of the consecutive block. If the former case were true, then $B(q) > 1$, implying that $p_{n-B(q)+3}$ exists. However since $q_b = b$, $p_{n-B(q)+2}$ is the largest value in the subsequence. But, then, we could substitute the $p_{n-B(q)+3}$ term for the $p_{n-B(q)+2}$, which is a contradiction since this then implies that the original $p$ contained $q$. The latter possibility similarly leads to contradiction as well.
\end{itemize}

\noindent {\bf Case 2: $q_b \neq b$}
\begin{itemize}

\item[$n = km + s$, $0 < s < k$]\

In this subcase, the final row has at least one cell (i.e. a value), but must still be incomplete (since $s < k$). Thus, there are no restrictions on what can be injected into the row. So, the algorithm is simply to inject $n+1$ (in this case, this is simply appending $n+1$ to the end). Clearly, since this must be the largest value in the permutation, and the value is in the last row, $p_{n+1} = n+1$. Thus, the $p'$ that results must avoid $q$, since clearly $p_{n+1}$ cannot be part of a subsequence order-isomorphic to $q$ since $q_b \neq b$ and since the original $p$ avoided $q$.


\item[$n = km$]\

This subcase is slightly more complicated. Here, he have an added restriction; by definition, $p_{km} > p_{km+1}$. We will proceed with further casework.


\textbf{If $q_b \neq 1$}, we simply inject 1 into the final row (i.e. $p_{km+1} = 1$). Clearly, then, since 1 is the smallest value in the permutation, $p_{km+1}$ cannot be part of a subsequence that is order-isomorphic to $q$ since $q_b \neq 1$. Thus, since the original permutation $p$ avoided $q$, $p'$ avoids $q$ as well, and so, $p' \in D^k_{n+1}(q)$.


\textbf{If $q_b = 1$ and $q_{b-1} = 2$}, the algorithm is to simply inject $n+1$. As a result, $p_{km} = n+1$, and then the old value of $p_{km}$ is bumped up into the next row. This swapping is essential because $p_{km} > p_{km+1}$. Clearly, $p_{km}$ cannot be part of any subsequence order-isomorphic to $q$, because the value of $n+1$ cannot correspond to the 1 nor 2 in $q$. Thus, since the original $p$ avoids $q$, and the relative positions of the values in $p'$ are invariant from $p$, $p'$ avoids $q$ as well, and so, $p' \in D^k_{n+1}(q)$.


\textbf{If $q_b = 1$ and $q_{b-1} \neq 2$}, the algorithm is to inject/append $f$ to the new row. Let $p_{km}$ = $f$. So, upon the injection, $p_{km}$ = $f+1$ and $p_{km}$ = $p_{km+1} + 1$. We claim that the resulting $p'$ avoids $q$. Assume, for the sake of contradiction, $p'$ contains $q$ and so there exists some subsequence of $p'$ that is order-isomorphic to $q$. Since $p$ avoids $q$, the only situations to consider are if only one of $p_{km}$ and $p_{km+1}$ are part of the subsequence, or if both $p_{km}$ and $p_{km+1}$ are in the subsequence. These situations are easily tractable, yielding contradictions in a manner similar to the proofs above.  Thus, $p'$ avoids $q$ as well, and so, $p' \in D^k_{n+1}(q)$.

\end{itemize}

The above procedures are all reversible as we can easily undo the injection. Thus, the casework shows that the algorithm is indeed injective and that the children $f(p)$ are pairwise distinct.
\end{proof}

\subsection{Strict Case: $|D^k_n(q)| < |D^k_{n+1}(q)|$}
\label{strict case}

Call a pattern \emph{repetitive} if it avoids $321, 132, 231$. Similarly, a pattern is \emph{non-repetitive} if it contains at least one of $321, 132, 231$.
We prove that
\begin{thm}
\label{theorem: strict case}
	For all non-repetitive patterns $q$ all all $k, n$, we have $|D^k_n(q)| < |D^k_{n+1}(q)|$.
    If $q$ is repetitive and $k \mid n$, then $|D^k_n(q)| < |D^k_{n+1}(q)|$ as well.
\end{thm}
Our approach will be an inductive one. Lemma~\ref{inductive step between D_n and D_n+1} captures the overall nature of induction from $|D^k_n(q)|$ to $|D^k_{n+1}(q)|$, while the rest of the section more specifically details our algorithm through casework based on the value of $n$.

The following lemma provides the framework for our inductive argument.

\begin{lemma}
		\label{inductive step between D_n and D_n+1}
		Let $q$ and $q'$ be two patterns such that $q$ contains $q'$. If  $|D^k_n(q')| < |D^k_{n+1}(q')|$, it must also be true that $|D^k_n(q)| < |D^k_{n+1}(q)|$.
		\end{lemma}
		\begin{proof}
The key idea is that any parent of a permutation that avoids $q'$ also avoids $q'$.
Therefore, under the theorem conditions, the assignment of children of Theorem~\ref{nonstrict case theorem} must
miss a permutation in $D^k_{n+1}(q')$.
Let $f: D^k_n(q) \hookrightarrow D^k_{n+1}(q)$ be defined from Theorem~\ref{nonstrict case theorem}.
If $f(p) \in D^k_{n+1}(q)$,
then a child of $p$ does not contain $q$, and therefore $p$ does contain $q$
and $p \in D^k_n(q)$.  Because  $|D^k_n(q')| < |D^k_{n+1}(q')|$, there is an
element of $D^k_{n+1}(q')$ that is outside $f(D^k_n(q')) \supseteq D^k_{n+1}(q') \cap f(D^k_n(q))$.
Hence, $f$ is not surjective, and therefore $|D^k_n(q)| < |D^k_{n+1}(q)|$.
\end{proof}

\begin{proof}[Proof of Theorem~\ref{theorem: strict case}]
We now shall proceed with casework based on the value of $n$. In Section~\ref{$n = km + s$, $0 < s < k$}, we deal with the case of continuing an incomplete final row, and so investigate permutations that contain at least one
of 321, 132, 231.  We apply Lemma~\ref{inductive step between D_n and D_n+1} for the desired result. In Section~\ref{$n = km$}, when constructing a new row, we continue with the same structure of our proof.
We first consider permutations that avoid at least one of 213, 312; then,
we do casework to finish the argument. 

\subsubsection{$n = km + s$, $0 < s < k$}
\label{$n = km + s$, $0 < s < k$}


We consider non-repetitive patterns, and we revisit repetitive patterns in the Section~\ref{equality case}.
Let $q$ be a non-repetitive pattern, and we do casework on which of 132, 231, 321 that $q$ contains.

\begin{itemize}[leftmargin=*]
\item[132 or 231]

Much of this case has already been shown by Lewis in \cite{Lewis2011}.
Recall that permutations of descent type $k$ are obtaining from
reading skew Young tableaux of a particular shape (see Figure~\ref{DnkSSYT}).
Suppose that permutations of length $n$ of descent type $k$ are identified with tableaux of
shape $\lambda / \mu$ and such permutations of length $n+1$ are identified with tableaux
of shape $\lambda' / \mu'$.  Let $r = \lceil \frac{n}{k}\rceil$;
then $\lambda$ and $\lambda'$ have $r$ rows.  By~\cite[Corollary 7.3]{Lewis2011}, there is a bijection
between $D^k_n(132)$ and the set of Young diagrams $Y \subseteq (\lambda_1-\lambda_r, \lambda_2-\lambda_r,
\cdots, \lambda_{r-1} - \lambda_r)$, and similarly for $D^k_{n+1}(q)$ and $\lambda'$.
However, we have $\lambda'_1 = \lambda_1 + 1$ and $\lambda_i = \lambda'_i$ for all $i > 1$.
Therefore, $(\lambda_1-\lambda_r, \lambda_2-\lambda_r,
\cdots, \lambda_{r-1} - \lambda_r) \subsetneq (\lambda'_1-\lambda'_r, \lambda'_2-\lambda'_r,
\cdots, \lambda'_{r-1} - \lambda'_r)$ and the fact that $|D^k_{n}(q) < |D^k_{n+1}(q)|$ follows.
A similar argument using~\cite[Corollary 7.6]{Lewis2011} settles the 231 case.

\item[321]
Consider a permutation $p$ in $|D^k_n(321)|$. Clearly, we may simply append the value $n+1$ to $p$
to obtain an element of $D^k_{n+1}(321)$. However, we can also replace $p_{km}$ with the value $n+1$, and inject the $p_{km}$ value into the final row. Since the sets of permutations derived from the two procedures are disjoint
due to different locations of $n+1$, we have
$|D^k_{n+1}(321)| \geq 2|D^k_n(321)|$.
	
\end{itemize}

Lemma~\ref{inductive step between D_n and D_n+1} implies that, for all $k \nmid n$
and $q$ non-repetitive, we have $|D^k_n(q)| < |D^k_{n+1}(q)|$.

\subsubsection{$n = km$}
\label{$n = km$}

This case is slightly more complicated than the previous case. When $k = 2$ (alternating permutations), we firstly
 consider patterns that contain at least one of 123, 213, and 312. However, when $k > 2$, we instead consider patterns
 that contain at least one of 321, 213, and 312. For both these triplets of patterns, there are patterns that avoid all three. So, at the end of this section, we address these patterns by considering all such patterns of length 4. It is important to note that unlike the previous subsection, this case includes repetitive patterns as well.

\begin{itemize}[leftmargin = *]
\item[123]

Since we are only considering the $k = 2$ case, we may clearly inject the values 1 and 2 while
preserving 123-avoidance. Since these injections result in distinct permutations, we have $|D^2_{2m}(123)| < |D^2_{2m+1}(123)|$.

\item[213 or 312]

The idea here is very similar to the preceding one. Since the permutation $p$ has descent type $k$, we may clearly inject any value $v \leq k$ into $p$. One may verify that this operation is reversible.
Therefore, we have $|D^k_{km+1}(q)| \geq k|D^k_{km}(q)| > |D^k_{km}(q)|$ for $q$ = 213 or 312.

\item[321]

In this case we prove the following more interesting result.

\begin{prop}
For all $k, m > 1,$ we have
\[|D^k_{km+1}(321)| = \displaystyle\sum\limits_{i = k(m-1) + 2}^{km} |D^k_{i}(321)|.\]
\end{prop}
\begin{figure}
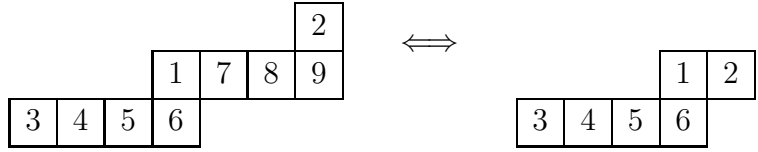

\center{
\begin{ytableau}
\none & \none & \none & \none & \none &\none & 2\\
\none & \none & \none & 1 & 7 & 8 & 9\\
3 & 4 & 5 & 6
\end{ytableau}} \hspace{.5 cm}
$\Longleftrightarrow$ \hspace{.5 cm}
\begin{ytableau}
\none\\
\none & \none & \none & 1 & 2\\
3 & 4 & 5 & 6
\end{ytableau}
\caption{Removing the largest consecutive block (7,8,9) from 345617892 and collapsing the final row into the row beneath it results in another valid permutation avoiding 321. Reversing the deletion of a consecutive block is clear as well, for inserting (7,8,9) into 345612 bumps up the final value 2 into a new row and the consecutive block fills the prior final row.}
\label{consecblock}
\end{figure}

\begin{proof}[Proof Idea]
The essence of this proof is reintroducing the notion of the consecutive block. First, it is important to note that for $p \in D^k_{km+1}(321)$, $p_{km}$ = $km+1$. Thus, a bijection is achieved by simply inserting a consecutive block into a permutation and in the other direction, removing the largest such block from a permutation. We may see this better graphically. Structurally, a consecutive block is a group of consecutive cells of a permutation that are in the same row and also consecutively ascending in value. A graphical example of this bijection is depicted in Figure~\ref{consecblock}.
The proposition follows from this bijection.
\end{proof}

Note that $|D^k_{km+1}(321)|$ = $|D^k_{km}(321)|$ when $k = 2$; however, when $k > 2$, we have $|D^k_{km+1}(321)| > |D^k_{km}(321)|$.

\item[Patterns avoiding 123, 213, and 312]\

Consider the length-4 patterns that avoid 123, 213, and 312; they are 4321, 1432, 2431, and 3421. Note that every pattern that avoids 123, 213, and 312 must contain at least one of these 4 patterns. We shall show that for each of these patterns there exists a second valid injection, distinct from the one provided in Section~\ref{nonstrict case}, for every $p \in D^k_{km}(q)$. For the sake of space, we shall simply provide the second algorithm in each case without proof, for the individual proofs are quite simple.

\begin{itemize}[leftmargin = *]

\item[4321 and 3421]

Inject $n$ if $p_{km} \neq n$. Inject $n-1$ if $p_{km} = n$.

\item[1432 and 2431]

Inject $n+1$.

\end{itemize}

Thus, these length-4 patterns help conclude the case for patterns that avoid 123, 213, and 312.

\item[Patterns avoiding 321, 213, and 312]\

Consider the length-4 patterns that avoid 321, 213, and 312; they are 1234, 1243, 1342, and 2341. Note that every pattern that avoids 321, 213, and 312 must contain at least one of these 4 patterns. As before, we simply provide the algorithm.

\begin{itemize}[leftmargin = *]

\item[1234 and 1243]

Inject 2.

\item[1342]

Inject $p_{km}$

\item[2341]

Inject $n-1$ if $p_{km} = n$. Inject $p_{km}+2$ if $p_{km} \neq n$.

\end{itemize}

Thus, these length 4 patterns help conclude the case for patterns that avoid 321, 213, and 312.
Combining all the cases and applying Lemma~\ref{inductive step between D_n and D_n+1}
yields that $|D^k_n(q)| < |D^k_{n+1}(q)|$ for all $q$ and $k \mid n$. \qedhere
\end{itemize}
\end{proof}
\subsection{Equality Case: $|D^k_n(q)| = |D^k_{n+1}(q)|$ (Repetitive Patterns)}
\label{equality case}

In the previous section, we defined repetitive patterns to be those that avoided 321, 132 and 231 simultaneously. Now,
we determine the structure of such patterns. Fix the location of the ``1''. Since the pattern avoids 231 and 321 concurrently, there can be at most one value to the left of the ``1''. Additionally, since the pattern avoids 132, all values to the right of the ``1'' must be in strictly increasing order. Thus, the pattern $q$ must be the identity pattern, or must be of the form $t123\cdots(t-1)(t+1)\cdots b$, where $q_1 = t$ and $b$ is the length of the pattern.

We consider when $q$ is a non-identity, repetitive pattern. The sequence
$\left\{D^k_n(q)\right\}$ has predictable repetitions among consecutive terms.  We prove the following theorem.

\begin{thm}
	\label{repetitive pattern theorem}
For all $k \ge b-1$ and all non-identity repetitive patterns $q$ of length $b,$
\[|D^k_{km+(b-2)}(q)| = |D^k_{km+(b-1)}(q)| = |D^k_{km+b}(q)| = \cdots = |D^k_{km+k}(q)|.\]
\end{thm}

We divide the proof into two separate cases. Lemmas~\ref{lem:first entry max} and~\ref{repetitive bij in b = t} tackle the case when $q_1 = b$. Similarly, Lemmas~\ref{repetitive adj values} and~\ref{repetitive general case} deal with the general case.

\begin{lemma}
		\label{lem:first entry max}
		For $q$ with $q_1 = b$ and $b-1 \leq x \leq k$, for all $p$ in $D^k_{km+x}(q)$, $p_{km+x} = km+x$.
		\end{lemma}
		\begin{proof}
		Assume, for the sake of contradiction, $km+x$ is not part of the final row in $p$, i.e. it exists in some 	
		earlier block of $p$. Then $\exists i < km+1 $ such that $p_i = km + x$. Since $x \geq t-1$, $p_{km+1} < p_{km+2} <
		\cdots < p_{km + (t-1)}$. However, since $i < km+1$ and $p_i > p_{km+1}$, $p_i, p_{km+1}, p_{km+2}, \cdots , p_{km +
		(b-1)}$ is order-isomorphic to $q$. Thus, $p$ contains $q$. Yet, this is a contradiction since $p$ is defined to
		\emph{avoid} $q$. Consequently,  $km+x$ must be part of the final row in $p$. However, by definition, since
		each block is strictly ordered from least to greatest, $p_{km+1} < p_{km+2} < \cdots < p_{km + x}$, and because $km+x$
		is the largest value in $p$ (and the block as well), $ p_{km+x} = km+x$, thus completing the proof of the lemma.
	\end{proof}
	
\begin{lemma}
		\label{repetitive bij in b = t}
		For $q$ with $q_1 = b$ and $b-2 \leq x \leq k-1$, there exists a bijection from $D^k_{km+x}(q)$ to $D^k_{km+(x+1)}(q)$.
		\end{lemma}
		\begin{proof}
		First we must show $D^k_{km+x}(q) \to D^k_{km+(x+1)}(q)$. For $p \in D^k_{km+x}(q)$, the
		injection of $p_{km+x+1} = km+x+1$ results in a $p' \in D^{k,l=0}_{km+(x+1)}(q)$, since this injection
		clearly maintains all original relations prior to the injection and is valid because then $p_{km+(x+1)}$ holds the
		largest value in the permutation.
		
		Now we must show $D^k_{km+(x+1)}(q) \to D^k_{km+x}(q)$. From Lemma \ref{lem:first entry max}, For $s \in D^k_{km+(x+1)}(q)$,  $s_{km+(x+1)} = km + (x+1)$. Thus, we may simply strip off $s_{km+(x+1)}$ from $s$ to get an $s'
		\in D^k_{km+x}(q)$.
		
		Thus we have established the bijection from $D^k_{km+x}(q)$ to $D^k_{km+(x+1)}(q)$.
	\end{proof}

Now we tackle the more general case.

\begin{lemma}
		\label{repetitive adj values}
		For $q$ with $q_1 = t$, $t \neq b$ and $b-1 \leq x \leq k$, for all $p$ in $D^k_{km+x}(q)$, $p_{km+(x+t-b+1)} = p_{km+(x+t-b)} + 1$.
		\end{lemma}
		\begin{proof}
		Assume, for the sake of contradiction, $p_{km + (x+t-b)}$ and $p_{km+(x+t-b+1)}$ are not consecutive values. So, $\exists i$ such that
		$p_{km + (x+t-b)} < p_i < p_{km+(x+t-b+1)}$. Clearly, $p_i$ is not in the final block of $p$ (it it was, its ``cell'' would be in
		between those of $p_{km + (x+t-b)}$ and $p_{km+(x+t-b+1)}$, which is impossible since $p_{km + (x+t-b)}$ and $p_{km+(x+t-b+1)}$ are 	
		adjacent cells) and so,  $p_i$ is in an earlier block of $p$. So, $i < km+1$ and $p_i > p_{km+(x-b+2)}$ (since $p_{km+(x-b+2)}$ is in the final block, yet $p_{km+(x-b+2)} \leq p_{km + (x+t-b)}$).
		
		However, since
		$p_{km+(x-b+2)} < p_{km+(x-b+3)} < \cdots < p_{km + (x+t-b)} <  p_{km + (x+t-b+1)} < \cdots < p_{km + k}$, and $x \geq
		t-1$,
		$p_i, p_{km+(x-b+2)}, p_{km+(x-b+3)}, \cdots , p_{km + (x+t-b)}, p_{km + (x+t-b+1)}, \cdots, p_{km + x}$
		is order-isomorphic to $q$ (since $p_{km + (x+t-b)} < p_i < p_{km+(x+t-b+1)}$). To see this more clearly note that
		$p_{km+(x-b+2)}, p_{km+(x-b+3)}, \cdots , p_{km + (x+t-b)}$ are order-isomorphic to  $1, 2, 3, \cdots , t-1$. $p_i$
		comprises the $t$ term, while  $p_{km + (x+t-b+1)}, p_{km + (x+t-b+2)} \cdots, p_{km + (x-1)}, p_{km + x}$ is
		order-isomorphic to $t+1, t+2, t+3, \cdots , b$.
		
		Thus, $p$ \emph{contains} $q$. Yet, this is a contradiction since $p$ is defined
		to
		\emph{avoid} $q$. So, $p_{km + (x+t-b)}$ and $p_{km+(x+t-b+1)}$ are consecutive, and so, $p_{km+(x+t-b+1)} = p_{km+(x+t-b)} + 1$.
	\end{proof}
	
\begin{lemma}
		\label{repetitive general case}
		For $q$ with $q_1 = t$, $t \neq b$ and $b-2 \leq x \leq k-1$, there exists a bijection from $D^k_{km+x}(q)$ to $D^k_{km+(x+1)}(q)$.
		\end{lemma}
		\begin{proof}
	First we must show $D^k_{km+x}(q) \to D^k_{km+(x+1)}(q)$. From Section~\ref{nonstrict case}, for $p \in D^k_{km+x}(q)$, the
		injection of $p_{km+(x+t-b+2)} = p_{km+(x+t-b+1)} + 1$ results in a $p' \in D^k_{km+(x+1)}(q)$.
Now we must show $D^k_{km+(x+1)}(q) \to D^k_{km+x}(q)$. This direction is much clearer and straightforward. From Lemma~\ref{repetitive adj values}, $\forall s \in D^k_{km+(x+1)}(q)$,  $p_{km+(x+t-b+2)} = p_{km+(x+t-b+1)} + 1$. Thus, we may simply remove $p_{km+(x+t-b+2)}$ from $s$ to get an $s'
		\in D^k_{km+x}(q)$.
		Thus we have established the bijection from $D^k_{km+x}(q)$ to $D^k_{km+(x+1)}(q)$.
	\end{proof}
	
Combining Lemmas~\ref{repetitive bij in b = t} and~\ref{repetitive general case}, we prove the desired result of Theorem~\ref{repetitive pattern theorem}.

\subsubsection*{The Identity Permutation}

The identity permutation merits mention. When $k = b-1$, the identity pattern has repetitions for the exact same values of $n$ as other repetitive patterns (the argument for this case is identical to the one above). For $n \ge k \geq b$ however, we have $|D^k_n(q)| = 0$. Thus, only short identity patterns behave like other repetitive patterns.

\section{Implications of shape-equivalence for generalized alternating permutations}
\label{sec:ShapeEquivGenAlt}

Proposition~\ref{Shape2Spec} yields inequalities for $12q$ and $21q$-avoiding
generalized alternating permutations.  The following two theorems exploit the generality
of the AD-Young diagram structure.  Their proofs involve considering non-alternating AD-Young diagrams
and applying the key lemmata used in the proof of Theorem~\ref{ShapeExtend}.

\begin{prop}
\label{Extend1}
Let $C$ be an $r\times r$ permutation matrix. If $\mathcal{Y} = (Y, A, D)$ is an AD-Young diagram such that
$Y$ has $n$ rows (columns) and $A \supseteq (D \cap [n-1-r]) + 1$, then we have
\[\left|S_{\mathcal{Y}}\left(\begin{bmatrix}
I_2 & 0 \\
0 & C\end{bmatrix}\right)\right|
\le
\left|S_{\mathcal{Y}}\left(\begin{bmatrix}
J_2 & 0 \\
0 & C\end{bmatrix}\right)\right|.\]
\end{prop}
The constraint on $A, D$ is that every required descent, except possibly those involving
the last $r$ rows, must be immediately preceded by a required ascent.
\begin{proof}
We use the notation of Section~\ref{sec:Extension}.
Suppose that $N \in \mathcal{D}$, and let $f(N) = (Y', A', D')$.  We claim that
if $A' = \emptyset$, then $D' = \emptyset$.

We prove the contrapositive; suppose that $j \in D'$.
Let $f(N)$ have $k$ rows and, for $1 \le i \le k$, suppose that the $i^\text{th}$ row of $Y'$
was the $r_i^\text{th}$ row of $Y$ before row and column deletion.  It is clear that we
have $r_j < r_{j+1} \le n-x$, and hence we have $r_j \le n-x-1$.  This yields that $r_j - 1 \in A$,
and by Proposition~\ref{AltTechnical}, we have that $j-1 \in A'$.  Taking contrapositives, we have
that if $A' = \emptyset$, then $D' = \emptyset$.  By Proposition~\ref{Shape2Spec}, it follows that
$\left|S_{f(N)}(I_2)\right| \le \left|S_{f(N)}(J_2)\right|$ for all $N \in \mathcal{D}$.

Adding these inequalities as $N$ ranges over $\mathcal{D}$ and applying Corollary~\ref{Embed2} yields that
\[\left|S_\mathcal{Y}\left(\begin{bmatrix}
I_2 & 0\\
0 & C\end{bmatrix}\right)\right| = \sum_{N \in \mathcal{T}} \left|S_{f(N)}(I_2)\right|
\le \sum_{N \in \mathcal{T}} \left|S_{f(N)}(J_2)\right| =
\left|S_{\mathcal{Y}}\left(\begin{bmatrix}
J_2 & 0\\
0 & C\end{bmatrix}\right)\right|,\]
as desired.
\end{proof}

\begin{thm}
\label{NoConsecD}
Suppose that $n, t$ are positive integers with $t \ge 2$, and $D \subseteq [n-1]$ such
that $1 \notin D$ and $D \cap [n+1-t]$ does not contain any two consecutive integers.  If
$q$ is a permutation of $[t] \setminus [2]$, then the number of permutations of length $n$ with
descent set $D$ that avoid $12q$ is at most the number of permutations of length $n$ with descent set
$D$ that avoid $21q$.
\end{thm}
\begin{proof}
Apply Proposition~\ref{Extend1} to $\mathcal{Y} = (Y, A, D)$ with $Y = (n^n)$ and $A = [n-1] \setminus D$,
and let $C = M(q)$.
\end{proof}

Exchanging the roles 12 and 21 reverses the inequality sign and yields a similar proposition and theorem.

\begin{prop}
\label{Extend2}
Let $C$ be an $r\times r$ permutation matrix. If $\mathcal{Y} = (Y, A, D)$ is an AD-Young diagram such that
$Y$ has $n$ columns and $D \supseteq (A \cap [n-r]) - 1$, then we have
\[\left|S_{\mathcal{Y}}\left(\begin{bmatrix}
I_2 & 0 \\
0 & C\end{bmatrix}\right)\right|
\ge
\left|S_{\mathcal{Y}}\left(\begin{bmatrix}
J_2 & 0 \\
0 & C\end{bmatrix}\right)\right|.\]
\end{prop}

\begin{thm}
\label{NoConsecA}
Suppose that $n, t$ are positive integers with $t > 2$, and $A \subseteq [n-1]$ such
that $A \cap [n+2-t]$ does not contain any two consecutive integers.  If
$q$ is a permutation of $[t] \setminus [2]$, then the number of permutations of length $n$ with
ascent set $A$ that avoid $12q$ is at least the number of permutations of length $n$ with ascent set
$A$ that avoid $21q$.
\end{thm}

In particular, substituting $D = \left\{k,2k,\ldots,k\left\lfloor\frac{n}{k}\right\rfloor\right\}$
into Theorem~\ref{NoConsecD} yields that $|D_n^k(12q)| \le |D_n^k(21q)|.$  Similarly,
substituting $A = \left\{k,2k,\ldots,k\left\lfloor\frac{n}{k}\right\rfloor\right\}$
into Theorem~\ref{NoConsecA} and complementing yields that
$|D_n^k((t+2)(t+1)w)| \ge |D_n^k((t+1)(t+2)w)|$ for all $w \in S_t$. It is interesting
that the method that yields equalities for the $k=2$ case of alternating permutations can be generalized to yield inequalities for larger $k$.

\section{Future directions and open problems}
\label{conjectures}
The following conjecture would fully extend Theorem~\ref{BWXThm2.1} to alternating and reverse alternating permutations.
It generalizes Lemma~\ref{Shape2} and Proposition~\ref{JF3Strong}.  Recall
that $F_k = M((k-1)(k-2)\cdots1k)$ and $J_k = M(k(k-1)\cdots 1)$.
\begin{conj}
\label{SesaGeneral}
For all $k > 2,$ we have $F_k \sesa{1} J_k$.
\end{conj}
In addition, one may consider an analogue of AD-Young diagrams related
to doubly alternating permutations by also restricting the ascent
and descent sets of the transpose of a transversal. Specifically, we make
the following definition.
\begin{definition}
Let $Y$ be a Young diagram with $n$ rows and columns, and let $A, D, A_2, D_2 \subseteq [n-1]$.
We call $\mathcal{Y} = (Y, A, D, A_2, D_2)$ a \emph{double AD-Young diagram} if $(Y, A, D)$ and
$(Y^t, A_2, D_2)$ are AD-Young diagrams, where $Y^t$ denotes the transpose of $Y$.
\end{definition}
One may then extend Conjecture~\ref{SesaGeneral} to the context of double AD-Young diagrams.

Furthermore, empirical data, which we provide in Appendix~\ref{sec:enumerate}, suggests that most equivalences for alternating permutations
are generated by Conjecture~\ref{SesaGeneral} and trivial equivalences.  In particular, all possible
equivalences for odd-length alternating permutations among patterns of length 5 and 6 are
generated in this manner, as well as all but 5 equivalences for even-length
alternating permutations among patterns of length 6.  This occurrence mimics a similar phenomenon
for ordinary permutations documented in \cite{StanWest}, and ``sporadic" equivalences
occur between patterns of length 4.

We also have a conjecture regarding
the decreasing pattern $k(k-1)(k-2)\cdots 1$, which is once again suggested by brute-force enumerations.
The conjecture follows from explicit enumerations for $k = 3$, and we proved the case of $n = 2k-2$
in the proof of Corollary~\ref{decreasing}.
\begin{conj}
\label{conj:decreasing}
For all positive integers $k, n$ and all $q \in S_k$ with $q \not= k(k-1)(k-2)\cdots 1$.
we have $|A_n(q)| \ge |A_n(k(k-1)(k-2)\cdots 1)|$.  If $n \ge 2k-2$ is even,
then the equality is strict.
\end{conj}

Brute-force enumerations suggest the following conjecture, which would give Wilf-type equivalences
over all descent types.
\begin{conj}
For all $k \ge 0$ and $n \ge 3$, we have $|D^k_n(2134\cdots n)| = |D^k_n(n123\cdots (n-1))|$ as well as
$|D^k_n(123\cdots n(n-1))| = |D^k_n(23\cdots n1)|$.
\end{conj}
Equally interesting are permutations that do not seem to be Wilf-equivalent to any other pattern for any descent type. For length four patterns, we have the following conjecture.
\begin{conj}
For all $p = 1324, 1342, 3124, 3412$ and $p \neq q \in S_4$, $|D^k_n(p)| \neq |D^k_n(q)|$.
\end{conj}
\begin{qn}
Does a similar phenomenon arise for higher length patterns?
\end{qn}

\appendix

\section{Proof of Proposition~\ref{FullPhi}}
\label{sec:PropFullPhi}
Backelin-West-Xin's proof of \cite[Proposition 3.1]{BWX} involves a subboard $E$.  We consider
a similar board, and it plays a substantial role in the following proofs.
Let $T$ be a separable valid transversal of $\mathcal{Y}$ that contains
$J_3$, and let $h_J(T) = (a_1,a_2,a_3)$.  We define a subset of $Y$ called
$E_\phi(T)$ that will be free of elements of $T$ by definition of $h_J$; let
\[E_\phi(T) = \left(\left([1,a_1) \times [b_{a_2}, Y_{a_3}]\right) \cup \left((a_1,a_2) \times [b_{a_3},b_{a_1}]\right)
    \cup \left((a_2,a_3) \times [1,b_{a_2}]\right) \cup \left((a_3, \infty) \times (b_{a_2}, \infty)\right)\right) \cap Y.\]
The critical property of $E_\phi(T)$ is the following lemma, which plays a critical role
in the proof of Proposition~\ref{FullPhi}.

\begin{lemma} \label{EPhi}
If $T$ is a separable valid transversal of $\mathcal{Y}$ that contains $J_3$, then
$E_\phi(T)$ does not contain any element of $T$.
\end{lemma}
\begin{proof}
If $(i,b_i) \in [1,a_1) \times [b_{a_2}, Y_{a_3}]$, then
$(i,a_2,a_3) \in U(T)$.  If $(i,b_i) \in (a_1,a_2) \times [b_{a_3},b_{a_1}]$,
then $(a_1,i,a_3) \in U(T)$.  If $(i,b_i) \in (a_2,a_3) \times [1,b_{a_2}]$,
then $(a_1,a_2,i) \in U(T)$.  All three contradict the definition of $h_J$.
If $(i,b_i) \in (a_3, \infty) \times (b_{a_2}, \infty)$, then
$v = (a_2,a_3,i)$ is a copy of $F_3$ in $T$.  If $i \in A$, replace
$v$ by $(a_2,a_3,i+1)$.  Then, we have $v \in V(T)$, and
$S(v) \ge (a_3,a_2,0) > (a_3,a_1,a_2) = \#(h_J(T))$ in the lexicographic order,
which contradicts the separability of $T$.
\end{proof}

The following lemma will be used in the proof of Proposition~\ref{FullPhi} for
the case in which $T$ is of $J$-type 2.

\begin{lemma}
\label{JType2L1}
Let $T$ be a separable, valid transversal of $\mathcal{Y}$ of $J$-type 2, and let
$h_J(T) = (a_1,a_2,a_3)$.
Then, $b_{a_2} \le b_{a_3-1}$ and $a_3 - a_1 \ge 3$.
\end{lemma}
\begin{proof}
If $b_{a_2} > b_{a_3-1}$, then $(a_1,a_2,a_3-1) \in U(T)$, which contradicts the definition of
$h_J$.  If $a_3 - a_1 \le 2$, then we have $a_3 = a_1 + 1$
and $a_2 = a_1 + 1$.  Because $\mathcal{Y}$ is 1-alternating and $a_3-1 \in D$, we have
$a_1 = a_3 - 2 \in A$, which implies that $b_{a_1} < b_{a_2}$, contradiction.
\end{proof}

The following lemma will be used repeatedly in the proof of Proposition~\ref{FullPhi}
for the case in which $T$ is of $J$-type 3.

\begin{lemma}
\label{JType3L1}
Let $T = \{(i,b_i)\}$ be a separable, valid transversal of $\mathcal{Y}$ of $J$-type 3,
and let $\phi(T) = \{(i,c_i)\}$.
\begin{enumerate}[(a)]
\item If $i \in \Gamma_{[1,a_1)}^{[b_{a_3},b_{a_1}]}(T)$, then $b_{a_3+1} < b_i < b_{a_2}.$
Let $\Gamma_{[1,a_1]}^{[b_{a_3},b_{a_1}]}(T) = \{i_1 < i_2 < \cdots < i_k\}$; then $b_{i_1} < b_{i_2} < \cdots < b_{i_k}$ and $c_{i_1} < c_{i_2} < \cdots < c_{i_k}$.
\item Let $\Gamma_{[a_2,a_3)}^{[b_{a_3},b_{a_1}]}(T) = \{i_1 < i_2 < \cdots < i_k\}$, then $b_{i_1} < b_{i_2} < \cdots < b_{i_k}$ and $c_{i_1} < c_{i_2} < \cdots < c_{i_k} < c_{a_3}$.
In particular, if $i \in \Gamma_{[a_2,a_3)}^{[b_{a_3},b_{a_1}]}(T)$, then
$b_{a_2} \le b_i$.
\end{enumerate}
\end{lemma}
\begin{proof}
First, we prove part (a).  Let $i \in \Gamma_{[1,a_1)}^{[b_{a_3},b_{a_1}]}(T)$.
If $b_i \le b_{a_3+1}$, then $(i,a_3,a_3+1) \in V(T)$ and $S(i,a_3,a_3+1) = (a_3+2,i,0) > (a_3,a_1,a_2)$
in the lexicographic order, which contradicts the separability of $T$.  The fact
that $b_i \ge b_{a_2}$ follows from Lemma~\ref{EPhi}.  If $j < j'$ with $b_{i_j} > b_{i_{j'}}$, then $(i_j,i_{j'},a_3) \in U(T)$, which contradicts the definition of $h_J$.
Because $c_{i_j} = b_{i_{j+1}}$ for $j \in [k-1]$, it suffices to prove that $c_{i_1} < c_{i_2}$.  This follows
from $c_{i_1} = b_{a_3 + 1} < b_{i_1} = c_{i_2}$.

The proof of part (b) is similar.  If $j < j'$ with $b_{i_j} < b_{i_{j'}}$
then $(a_1,i_j,i_{j'}) \in U(T)$, which contradicts the definition of $h_J$.  Once again,
to finish it suffices to prove that $c_{i_1} < c_{i_2}$, but this follows from $c_{i_1} = b_{a_3} < b_{i_1} = c_{i_2}$.
\end{proof}

\begin{proof}[Proof of Proposition~\ref{FullPhi}]
We do casework on the $J$-type of $T$.  Let $h_J(T) = a = (a_1,a_2,a_3)$.

\begin{asparaenum}
\renewcommand{\labelenumi}{\bf{$J$-type \arabic{enumi}.}}
\item \begin{figure}
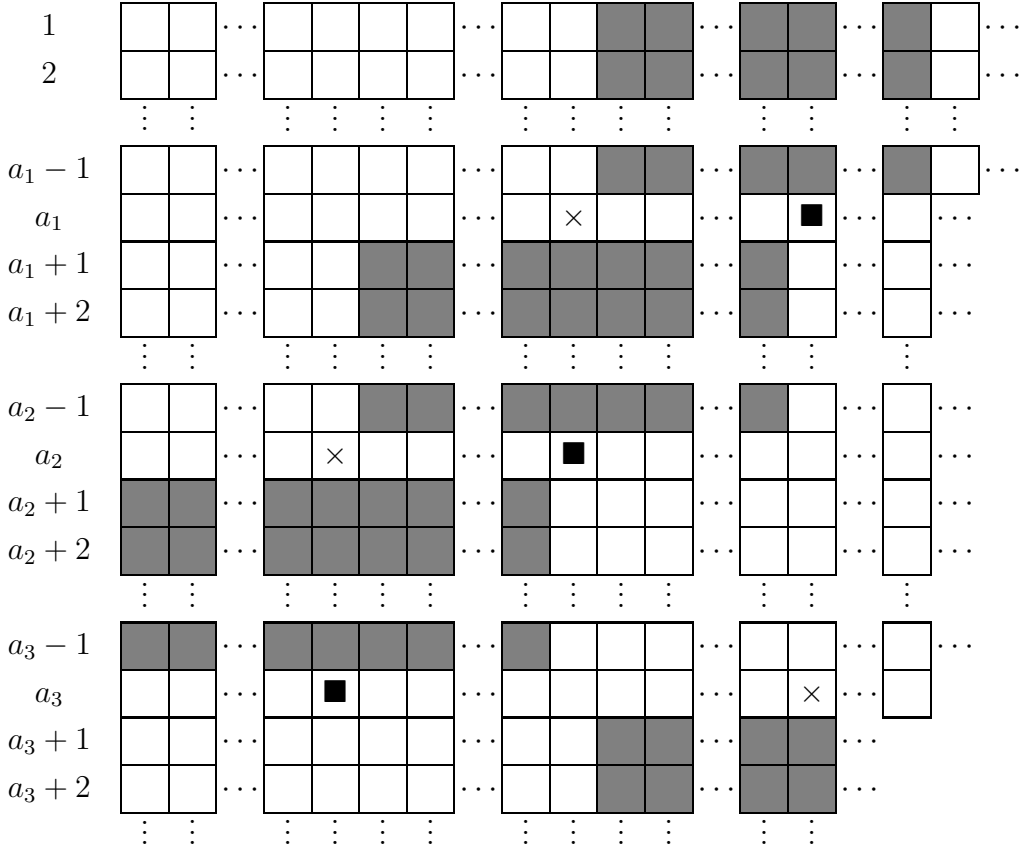

\begin{ytableau}
\none[1] &\none & &&\none[\cdots]& &&&&\none[\cdots] && & *(gray) & *(gray) &\none[\cdots] & *(gray) & *(gray)  & \none[\cdots] & *(gray) & & \none[\cdots]\\
\none[2] &\none & &&\none[\cdots]& &&&&\none[\cdots] && & *(gray) & *(gray) &\none[\cdots] & *(gray) & *(gray)  & \none[\cdots] & *(gray) & & \none[\cdots]\\
\none &\none & \none[\vdots] &\none[\vdots] &\none &\none[\vdots] &\none[\vdots] &\none[\vdots] & \none[\vdots] & \none & \none[\vdots] & \none[\vdots] & \none[\vdots] & \none[\vdots] & \none & \none[\vdots] & \none[\vdots] &\none &\none[\vdots] &\none[\vdots]\\
\none[a_1-1] &\none & &&\none[\cdots]& &&&&\none[\cdots] && & *(gray) & *(gray) &\none[\cdots] & *(gray) & *(gray)  & \none[\cdots] & *(gray) & & \none[\cdots]\\
\none[a_1] &\none &&& \none[\cdots] &&&&&\none[\cdots]&& \times & & & \none[\cdots] & & \blacksquare & \none[\cdots] && \none[\cdots]\\
\none[a_1+1] &\none &&& \none[\cdots] &&& *(gray) & *(gray) &\none[\cdots] & *(gray) & *(gray) & *(gray) & *(gray) & \none[\cdots] & *(gray) & & \none[\cdots] && \none[\cdots]\\
\none[a_1+2] &\none &&& \none[\cdots] &&& *(gray) & *(gray) &\none[\cdots] & *(gray) & *(gray) & *(gray) & *(gray) & \none[\cdots] & *(gray) & & \none[\cdots] && \none[\cdots]\\
\none &\none & \none[\vdots] &\none[\vdots] &\none &\none[\vdots] &\none[\vdots] &\none[\vdots] & \none[\vdots] & \none & \none[\vdots] & \none[\vdots] & \none[\vdots] & \none[\vdots] & \none & \none[\vdots] & \none[\vdots] &\none &\none[\vdots]\\
\none[a_2-1] &\none &&& \none[\cdots] &&& *(gray) & *(gray) &\none[\cdots] & *(gray) & *(gray) & *(gray) & *(gray) & \none[\cdots] & *(gray) & & \none[\cdots] && \none[\cdots]\\
\none[a_2] &\none &&& \none[\cdots] &&\times &&& \none[\cdots] && \blacksquare&&&\none[\cdots] &&& \none[\cdots] && \none[\cdots]\\
\none[a_2+1] &\none &*(gray) &*(gray) & \none[\cdots] &*(gray) &*(gray) &*(gray) &*(gray) & \none[\cdots] &*(gray) &&&&\none[\cdots] &&& \none[\cdots] && \none[\cdots]\\
\none[a_2+2] &\none &*(gray) &*(gray) & \none[\cdots] &*(gray) &*(gray) &*(gray) &*(gray) & \none[\cdots] &*(gray) &&&&\none[\cdots] &&& \none[\cdots] && \none[\cdots]\\
\none &\none & \none[\vdots] &\none[\vdots] &\none &\none[\vdots] &\none[\vdots] &\none[\vdots] & \none[\vdots] & \none & \none[\vdots] & \none[\vdots] & \none[\vdots] & \none[\vdots] & \none & \none[\vdots] & \none[\vdots] &\none &\none[\vdots]\\
\none[a_3-1] &\none &*(gray) &*(gray) & \none[\cdots] &*(gray) &*(gray) &*(gray) &*(gray) & \none[\cdots] &*(gray) &&&&\none[\cdots] &&& \none[\cdots] && \none[\cdots]\\
\none[a_3] &\none &&& \none[\cdots] &&\blacksquare &&&\none[\cdots]&&&&&\none[\cdots] && \times & \none[\cdots] &\\
\none[a_3+1] &\none &&& \none[\cdots] &&&&&\none[\cdots]&& & *(gray) &*(gray) &\none[\cdots] &*(gray) &*(gray) & \none[\cdots] \\
\none[a_3+2] &\none &&& \none[\cdots] &&&&&\none[\cdots]&& & *(gray) &*(gray) &\none[\cdots] &*(gray) &*(gray) & \none[\cdots] \\
\none &\none & \none[\vdots] &\none[\vdots] &\none &\none[\vdots] &\none[\vdots] &\none[\vdots] & \none[\vdots] & \none & \none[\vdots] & \none[\vdots] & \none[\vdots] & \none[\vdots] & \none & \none[\vdots] & \none[\vdots]\\
\end{ytableau}
\caption{The squares marked with a solid black box are the elements of the chosen copy of $J_3$ for a separable, valid transversal $T$ of $J$-type 1.  The crosses mark new elements of $\phi(T)$, i.e. elements of $\phi(T) \setminus T$, while the gray squares are free of elements of $T$ (and $\phi(T)$).}
\label{fig:JType1}
\end{figure}
See Figure~\ref{fig:JType1}.  First, we prove that $\phi(T)$ is a valid transversal of $\mathcal{Y}$.
Because $(a_1,a_2,a_3)$ is a copy of $J_3$ in $T$, we have $(a_3,b_{a_1}) \in Y$,
which implies that $\phi(T)$ is a transversal of $Y$.  If $\{i,i+1\} \cap \{a_1,a_2,a_3\} = \emptyset$,
then we have $b_i = c_i$ and $b_{i+1} = c_{i+1}$, so $i$ is an ascent (resp. descent) of $T$
if and only if $i$ is an ascent (resp. descent) of $\phi(T)$.
By Lemma~\ref{EPhi}, we have that $b_{a_1-1} \notin [b_{a_2}, b_{a_1}]$, which implies
that $a_1-1$ is a ascent (resp. descent) of $\phi(T)$ if and only if it is a (resp. descent) of $T$.
By Lemma~\ref{EPhi},
we have $b_{a_1+1} \notin [b_{a_3}, b_{a_1}]$.  Provided that
$a_2 \not= a_1 + 1,$ this implies
that $a_1$ is an ascent (resp. descent) of $\phi(T)$ if and only if it is an ascent (resp. descent) of $T$; however,
if $a_2 = a_1 + 1$, then it is clear $a_1$ is a descent of both $T$ and $\phi(T)$.
If $a_2 \not= a_1 + 1$, then, by Lemma~\ref{EPhi}, we have
$b_{a_2-1} < b_{a_3} (= c_{a_2}) < b_{a_2}$, which implies that $a_2-1$
is an ascent of both $T$ and $\phi(T)$.  If $a_3 \not= a_2 + 1$, then by Lemma~\ref{EPhi},
we have $b_{a_2+1} > b_{a_2} > b_{a_3} = c_{a_2}$,
which implies that $a_2+1$ is an ascent of both $T$ and $\phi(T)$; if $a_3 = a_2 + 1$,
then we have $a_2 \notin D$ by definition, and because $b_{a_2} > b_{a_3}$, we have
$a_2 \notin A$.
By Lemma~\ref{EPhi}, we have $b_{a_3-1} \ge b_{a_2} > b_{a_3}$, which implies that
$a_3 - 1 \notin A$.
Because $\mathcal{Y}$ is 1-alternating, we have $a_3 \notin D$,
and we also have $a_3 \notin A$ by definition of $J$-type.  It follows
that $\phi(T)$ is a valid transversal of $\mathcal{Y}$.

Next, we prove that $h_F(\phi(T)) = a$.  It is clear that $a \in V(\phi(T))$,
and because $a_3 - 1 \notin A$, we have $S(a) = (a_3, a_1, a_2)$.  Suppose that
$a' = (a'_1,a'_2,a'_3) \in U(T)$ with $S(a') > S(a)$ in the lexicographic order.
Because $a_3 - 1 \notin A$, we must have $a'_3 \ge a_3$.  If $a'_3 > a_3$,
then we have $c_{a'_3} = b_{a'_3} < b_{a_2}$ by Lemma~\ref{EPhi}.  For $i \in [2],$ let
\[d_i = \begin{cases}
a_3 & \text{ if } a'_i = a_2\\
a'_i  & \text{ otherwise.}\end{cases}\]
Because $b_{a'_2} < b_{a_2}$, we have $a'_2 \notin (a_2,a_3)$ by Lemma~\ref{EPhi},
from which it follows that $d_1 < d_2$.  We have $b_{d_i} = b_{a'_i}$ for $i \in [2]$,
which implies that $v = (d_1,d_2,a'_3) \in V(T)$.  Because $a_3 \notin A$, the first
component of $S(v)$ is greater than $a_3$, and this contradicts the definition of $h_F$.
Hence, we may assume that $a'_3 = a_3$, and because
$a_3-1 \notin A$, the first component of $S(a')$ is $a_3$.  If $a'_1 > a_1$, then $(a_1,a'_1,a'_2) \in U(T)$, which contradicts the definition of $h_J$.
If $a'_1 = a_1$, then Lemma~\ref{EPhi} implies that $a'_2 \le a_2$.  The fact
that $h_F(T) = a$ follows by definition of $h_F$.  It is clear that $\phi(T)$
is of $F$-type 1 and that $\psi(\phi(T)) = T$.

We prove that if $e = (e_1,e_2,e_3) \in U(\phi(T))$, then we have $\#(e) > \#(a)$
in the lexicographic order.  If $e_3 < a_3$ and $c_{e_3} > c_{a_1}$, then we have
$c_{e_i} = b_{e_i}$ for all $i$ and thus $(e_1,e_2,e_3) \in U(T)$, which contradicts
the definition of $h_J$.  If $e_3 < a_3$ and $c_{e_3} = c_{a_1}$, then we have
$e_3 = a_1$ and $b_{e_2} = c_{e_2} > Y_{a_3} \ge b_{a_1}$ by Lemma~\ref{EPhi}.  If
$e_3 < a_3$ and $c_{e_3} < c_{a_1}$, then we have $e_3 \le a_2$ by Lemma~\ref{EPhi}.
If $e_3 < a_2$, then we have $b_{e_3} = c_{e_3}$ and thus $(e_1,e_2,e_3) \in U(T)$
(because if $b_{e_i} = c_{e_i}$ for all $e_i \not= a_1$, with $b_{e_1} = c_{e_1}$;
if $e_2 = a_1$, then we have $b_{e_1} \ge Y_{a_3} > b_{a_1}$ by Lemma~\ref{EPhi}),
contradiction.  If $e_3 = a_2$ and $c_{e_1} < c_{a_1}$, then $(e_1,e_2,a_3) \in U(T)$,
contradiction.  If $e_3 = a_2$ and $c_{e_1} > c_{a_1}$, then we have $c_{e_1} = b_{e_1}
> Y_{a_3} \ge b_{a_1}$ by Lemma~\ref{EPhi}, which implies that $(e_1,a_1,a_2) \in U(T)$,
contradiction.  Hence, we may assume that $e_3 = a_3$.  Then, if $e_1 \le a_1$,
we have $c_{e_1} \le c_{a_1} < c_{a_3}$ by Lemma~\ref{EPhi}, contradiction.  The separability
of $\phi(T)$ follows.

We have $b_i = c_i$ for all $i < a_1$, and $b_{a_1} > b_{a_2} = c_{a_1}$.  Therefore,
$(b_1,b_2,\ldots,b_n) > (c_1,c_2,\ldots,c_n)$ in the lexicographic order.

\item\begin{figure}
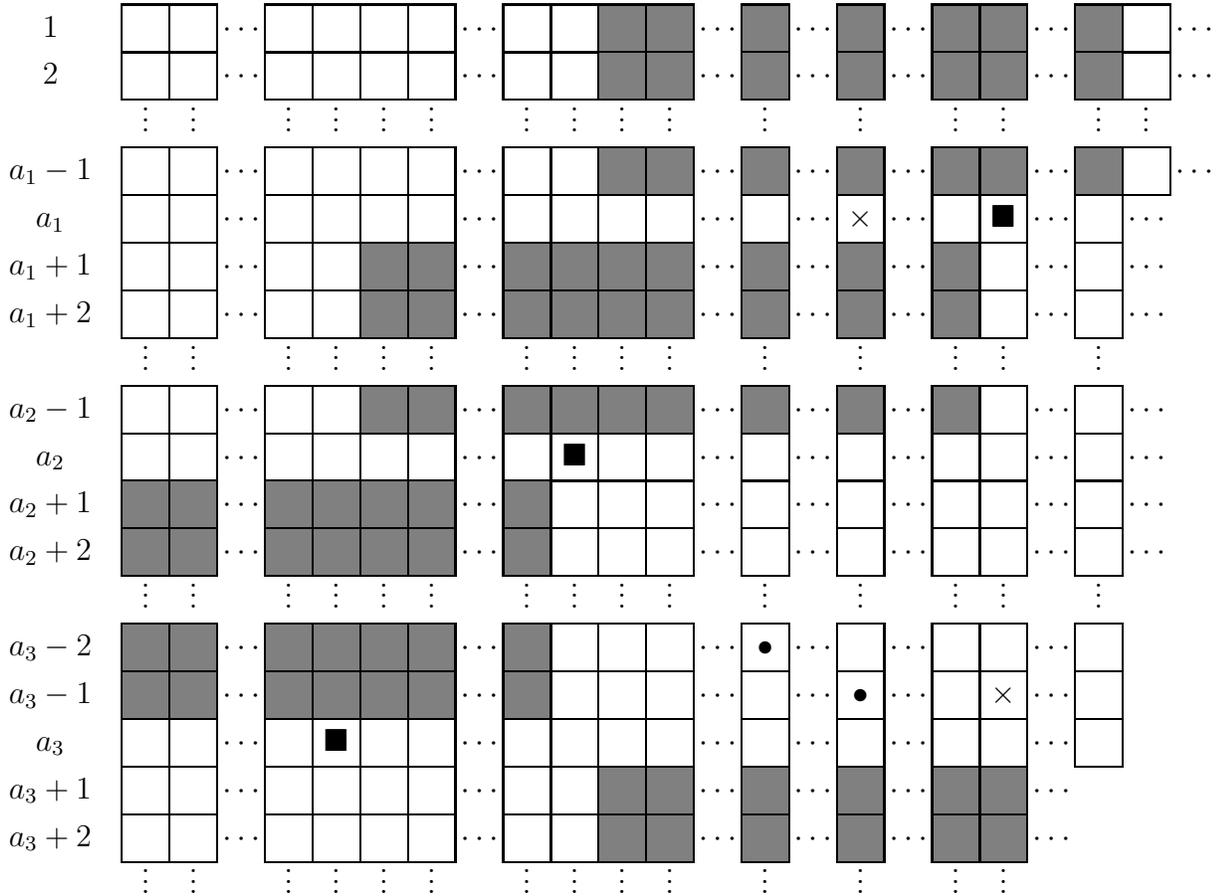

\begin{ytableau}
\none[1] &\none & &&\none[\cdots]& &&&&\none[\cdots] && & *(gray) & *(gray) &\none[\cdots] &*(gray) &\none[\cdots] &*(gray) &\none[\cdots] & *(gray) & *(gray)  & \none[\cdots] & *(gray) & & \none[\cdots]\\
\none[2] &\none & &&\none[\cdots]& &&&&\none[\cdots] && & *(gray) & *(gray) &\none[\cdots] &*(gray) &\none[\cdots] &*(gray) &\none[\cdots] & *(gray) & *(gray)  & \none[\cdots] & *(gray) & & \none[\cdots]\\
\none &\none & \none[\vdots] &\none[\vdots] &\none &\none[\vdots] &\none[\vdots] &\none[\vdots] & \none[\vdots] & \none & \none[\vdots] & \none[\vdots] & \none[\vdots] & \none[\vdots] & \none & \none[\vdots] & \none & \none[\vdots] & \none & \none[\vdots] & \none[\vdots] &\none &\none[\vdots] &\none[\vdots]\\
\none[a_1-1] &\none & &&\none[\cdots]& &&&&\none[\cdots] && & *(gray) & *(gray) &\none[\cdots] &*(gray) &\none[\cdots] &*(gray) &\none[\cdots] & *(gray) & *(gray)  & \none[\cdots] & *(gray) & & \none[\cdots]\\
\none[a_1] &\none &&& \none[\cdots] &&&&&\none[\cdots]&&&&& \none[\cdots] &&\none[\cdots] &\times &\none[\cdots] &&\blacksquare & \none[\cdots] && \none[\cdots]\\
\none[a_1+1] &\none &&& \none[\cdots] &&& *(gray) & *(gray) &\none[\cdots] & *(gray) & *(gray) & *(gray) & *(gray) &\none[\cdots] &*(gray) &\none[\cdots] &*(gray) &\none[\cdots] & *(gray) & & \none[\cdots] && \none[\cdots]\\
\none[a_1+2] &\none &&& \none[\cdots] &&& *(gray) & *(gray) &\none[\cdots] & *(gray) & *(gray) & *(gray) & *(gray) &\none[\cdots] &*(gray) &\none[\cdots] &*(gray) &\none[\cdots] & *(gray) & & \none[\cdots] && \none[\cdots]\\
\none &\none & \none[\vdots] &\none[\vdots] &\none &\none[\vdots] &\none[\vdots] &\none[\vdots] & \none[\vdots] & \none & \none[\vdots] & \none[\vdots] & \none[\vdots] & \none[\vdots] & \none & \none[\vdots] & \none & \none[\vdots] & \none & \none[\vdots] & \none[\vdots] &\none &\none[\vdots]\\
\none[a_2-1] &\none &&& \none[\cdots] &&& *(gray) & *(gray) &\none[\cdots] & *(gray) & *(gray) & *(gray) & *(gray) &\none[\cdots] &*(gray) &\none[\cdots] &*(gray) &\none[\cdots] & *(gray) & & \none[\cdots] && \none[\cdots]\\
\none[a_2] &\none &&& \none[\cdots] &&&&& \none[\cdots] && \blacksquare&&&\none[\cdots] &&\none[\cdots] &&\none[\cdots] &&& \none[\cdots] && \none[\cdots]\\
\none[a_2+1] &\none &*(gray) &*(gray) & \none[\cdots] &*(gray) &*(gray) &*(gray) &*(gray) & \none[\cdots] &*(gray) &&&&\none[\cdots] &&\none[\cdots] &&\none[\cdots] &&& \none[\cdots] && \none[\cdots]\\
\none[a_2+2] &\none &*(gray) &*(gray) & \none[\cdots] &*(gray) &*(gray) &*(gray) &*(gray) & \none[\cdots] &*(gray) &&&&\none[\cdots] &&\none[\cdots] &&\none[\cdots] &&& \none[\cdots] && \none[\cdots]\\
\none &\none & \none[\vdots] &\none[\vdots] &\none &\none[\vdots] &\none[\vdots] &\none[\vdots] & \none[\vdots] & \none & \none[\vdots] & \none[\vdots] & \none[\vdots] & \none[\vdots] & \none & \none[\vdots] & \none & \none[\vdots] & \none & \none[\vdots] & \none[\vdots] &\none &\none[\vdots]\\
\none[a_3-2] &\none &*(gray) &*(gray) & \none[\cdots] &*(gray) &*(gray) &*(gray) &*(gray) & \none[\cdots] &*(gray) &&&&\none[\cdots] &\bullet &\none[\cdots] &&\none[\cdots] &&& \none[\cdots] &\\
\none[a_3-1] &\none &*(gray) &*(gray) & \none[\cdots] &*(gray) &*(gray) &*(gray) &*(gray) & \none[\cdots] &*(gray) &&&&\none[\cdots] &&\none[\cdots] &\bullet &\none[\cdots] &&\times & \none[\cdots] &\\
\none[a_3] &\none &&& \none[\cdots] &&\blacksquare &&&\none[\cdots]&&&&&\none[\cdots] &&\none[\cdots] &&\none[\cdots] &&& \none[\cdots] &\\
\none[a_3+1] &\none &&& \none[\cdots] &&&&&\none[\cdots]&& & *(gray) &*(gray) &\none[\cdots] &*(gray) &\none[\cdots] &*(gray) &\none[\cdots] &*(gray) &*(gray) & \none[\cdots] \\
\none[a_3+2] &\none &&& \none[\cdots] &&&&&\none[\cdots]&& & *(gray) &*(gray) &\none[\cdots] &*(gray) &\none[\cdots] &*(gray) &\none[\cdots] &*(gray) &*(gray) & \none[\cdots] \\
\none &\none & \none[\vdots] &\none[\vdots] &\none &\none[\vdots] &\none[\vdots] &\none[\vdots] & \none[\vdots] & \none & \none[\vdots] & \none[\vdots] & \none[\vdots] & \none[\vdots] & \none & \none[\vdots] & \none & \none[\vdots] & \none & \none[\vdots] & \none[\vdots]\\
\end{ytableau}
\caption{The squares marked with a solid black box are the elements of the chosen copy of $J_3$ for a separable, valid transversal $T$ of $J$-type 2. The bullets mark other elements of $T$, and the crosses mark new elements of $\phi(T)$, i.e. elements of $\phi(T) \setminus T$.  The gray squares are free of elements of $T$ (and $\phi(T)$).}
\label{fig:JType2}
\end{figure}
See Figure~\ref{fig:JType2}.  First, we prove that $\phi(T)$ is a valid transversal of $\mathcal{Y}$.
Because $Y_{a_3-1} \ge Y_{a_3} \ge b_{a_1},$ the set $\phi(T)$ is a transversal of $Y$.
If $\{i,i+1\} \cap \{a_1,a_3-1\} = \emptyset$, then we have $b_i = c_i$ and $b_{i+1} = c_{i+1}$,
and thus $i$ is an ascent (resp. descent) of $T$ if and only if it is an ascent (resp. descent)
of $\phi(T)$.  By Lemma~\ref{EPhi}, we have $b_{a_1-1}, b_{a_1+1} \notin [b_{a_2}, b_{a_1}] \subseteq [b_{a_3-1}, b_{a_1}] = [c_{a_1},b_{a_1}]$, where the subset relation holds by Lemma~\ref{JType2L1}.  By Lemma~\ref{JType2L1}
again, we have $a_3 - 1 > a_1 + 1$, and it follows that $b_{a_1+1} = c_{a_1+1}$
and $b_{a_1-1} = c_{a_1-1}$.  Therefore, $a_1 - 1$ is an ascent (resp. descent) of $\phi(T)$
if and only if it is an ascent (resp. descent) of $T$, and the same for $a_1$.  Furthermore,
we have $c_{a_3-1} = b_{a_1} > b_{a_3-1} > b_{a_3-2},b_{a_3}$.  By Lemma~\ref{JType2L1}, we have that $a_3-2 > a_1$,
which yields that $b_{a_3-2} = c_{a_3-2}$ and $b_{a_3} = c_{a_3}$.  Therefore,
$a_3 - 2$ is an ascent of $\phi(T)$ and $a_3-1$ a descent.  It follows that $\phi(T)$
is a valid transversal of $\mathcal{Y}$, as desired.

Next, we prove that $h_F(\phi(T)) = (a_1, a_3-2, a_3-1)$.  It is clear that $(a_1,a_3-2,a_3-1) \in V(\phi(T))$,
and we have $S(a_1,a_3-2,a_3-1) = (a_3,a_1,0)$.  Suppose for sake of contradiction
that $d = (d_1,d_2,d_3) \in V(T)$ with $S(d) > (a_3,a_1,0)$ in the lexicographic order.
If $d_3 > a_3$, then by Lemma~\ref{EPhi}, we have $c_{d_3} = b_{d_3} < b_{a_2}$; this implies
that $b_{d_i} = c_{d_i}$ for $i = 1, 2$ and it follows that $(d_1,d_2,d_3) \in V(T)$.
If $d_1 < a_1$, then the second component of $S(d)$ is less than $a_1$, which implies that
first component of $S(d)$ is greater than $a_3$; it follows that $S(d) > (a_3,a_1,a_2)$ in
the lexicographic order, which contradicts the separability of $T$.  It is clear that
$d_1 \not= a_1$.  If $d_1 > a_1$, then $(a_1,e_1,e_2) \in U(T)$, which contradicts the
definition of $h_J$.  Thus, we have $h_F(\phi(T)) = (a_1,a_3-2,a_3-1)$.  It follows
that $\phi(T)$ is of $F$-type 2, and it is clear that $\psi(\phi(T)) = T$.

We prove that if $e = (e_1,e_2,e_3) \in h_J(\phi(T))$, then $\#(e) \ge (a_3,a_1,0)$
in the lexicographic order.  If $e_3 < a_3-1$ and $e_2 \not= a_1$,
then $(e_1,e_2,e_3) \in U(T)$ because $b_{e_i} = c_{e_i}$ for $i =2, 3$ and $b_{e_1} \ge c_{e_1}$, contradiction.
If $e_2 = a_1$, the fact that $c_{a_1} < c_{a_3-1}$ implies that $e_3 \not= a_3-1$,
and by Lemma~\ref{EPhi} we have $b_{e_1} > b_{a_1}$, which implies
that $(e_1,e_2,e_3) \in U(T)$, contradiction.  If $e_3 = a_3-1$, then the fact
that $c_{a_1} < c_{a_3}$ implies that $e_1,e_2 \not= a_1$ and $b_{e_i} = c_{e_i}$ for $i = 1, 2$.
Then, because $b_{a_3-1} < c_{a_3-1}$, we have $(e_1,e_2,e_3) \in U(T)$, contradiction.
Hence, we may assume that $e_3 = a_3$.  If $\{e_1,e_2\}$ and $\{a_1,a_3-1\}$ are disjoint,
then clearly we have $(e_1,e_2,e_3) \in U(T)$.  It is impossible for $e_1 = a_3-1$,
and if $e_2 = a_3-1$, then we have $e_1 \not= a_1$ and hence $(e_1,e_2,e_3) \in U(T)$,
contradiction.  If $e_1 = a_1$, then the fact that $b_{a_1} > c_{a_1}$ implies that
$(e_1,e_2,e_3) \in U(T)$, contradiction.  We have already dealt with the case $e_2 = a_1$.
The separability of $T$ follows.

We have $b_i = c_i$ for all $i < a_1$, and $b_{a_1} > b_{a_3-1} = c_{a_1}$.  Therefore,
$(b_1,b_2,\ldots,b_n) > (c_1,c_2,\ldots,c_n)$ in the lexicographic order.

\item
\begin{figure}
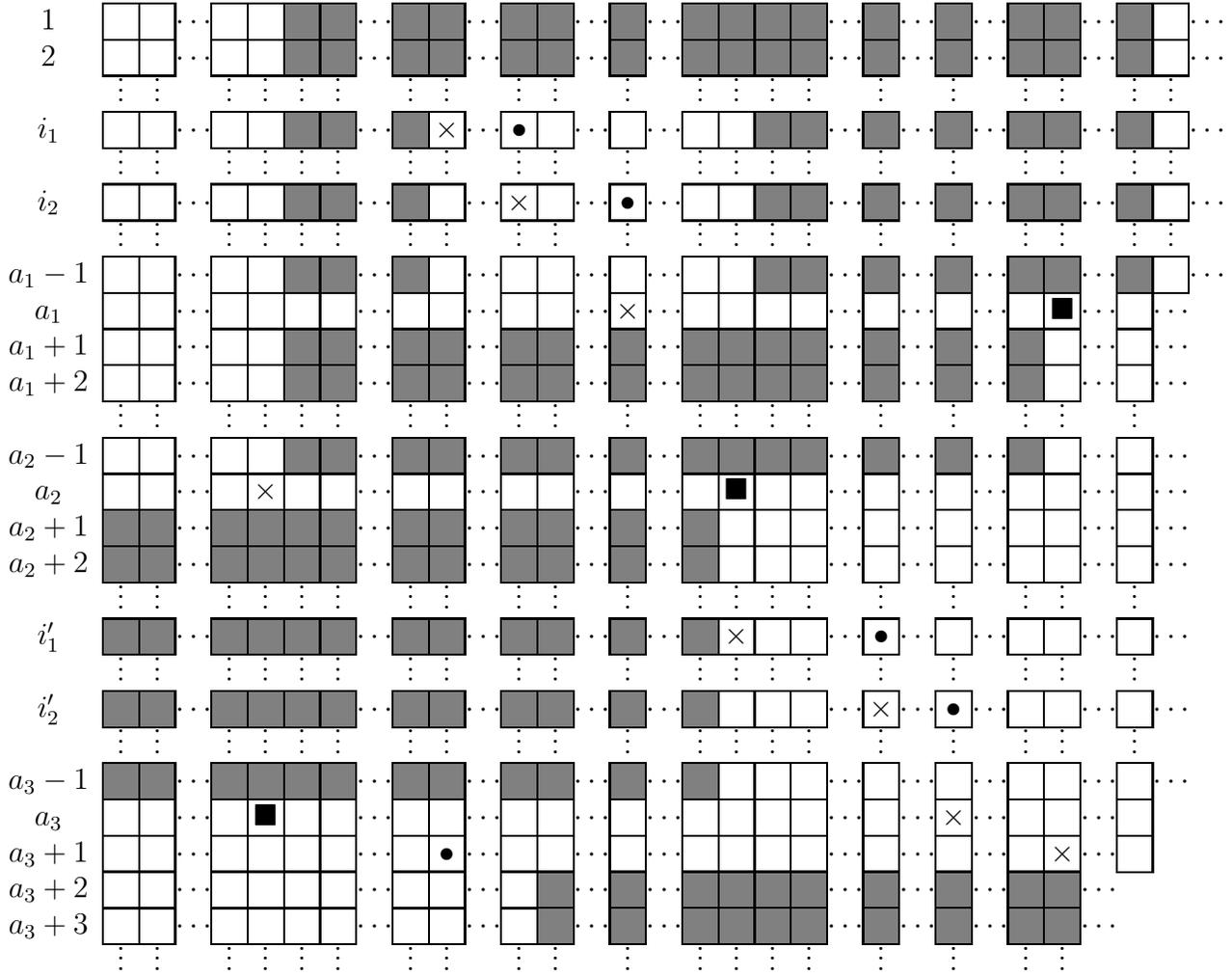

\ytableausetup{boxsize = 1.175em}
\begin{ytableau}
\none[1] &\none & &&\none[\cdots]& &&*(gray)&*(gray)&\none[\cdots] &*(gray)&*(gray)&\none[\cdots]&*(gray)&*(gray)&\none[\cdots]&*(gray)&\none[\cdots] &*(gray)&*(gray) & *(gray) & *(gray) &\none[\cdots] &*(gray) &\none[\cdots] &*(gray) &\none[\cdots] & *(gray) & *(gray)  & \none[\cdots] & *(gray) & & \none[\cdots]\\
\none[2] &\none & &&\none[\cdots]& &&*(gray)&*(gray)&\none[\cdots] &*(gray)&*(gray)&\none[\cdots]&*(gray)&*(gray)&\none[\cdots]&*(gray)&\none[\cdots] &*(gray)&*(gray) & *(gray) & *(gray) &\none[\cdots] &*(gray) &\none[\cdots] &*(gray) &\none[\cdots] & *(gray) & *(gray)  & \none[\cdots] & *(gray) & & \none[\cdots]\\
\none &\none & \none[\vdots] &\none[\vdots] &\none &\none[\vdots] &\none[\vdots] &\none[\vdots] & \none[\vdots] & \none & \none[\vdots]&\none[\vdots]&\none&\none[\vdots]&\none[\vdots]&\none&\none[\vdots]&\none& \none[\vdots] & \none[\vdots] & \none[\vdots] & \none[\vdots] & \none & \none[\vdots] & \none & \none[\vdots] & \none & \none[\vdots] & \none[\vdots] &\none &\none[\vdots] &\none[\vdots]\\
\none[i_1] &\none & &&\none[\cdots]& &&*(gray)&*(gray)&\none[\cdots] &*(gray)&\times &\none[\cdots]&\bullet&&\none[\cdots]&&\none[\cdots] && & *(gray) & *(gray) &\none[\cdots] &*(gray) &\none[\cdots] &*(gray) &\none[\cdots] & *(gray) & *(gray)  & \none[\cdots] & *(gray) & & \none[\cdots]\\
\none &\none & \none[\vdots] &\none[\vdots] &\none &\none[\vdots] &\none[\vdots] &\none[\vdots] & \none[\vdots] & \none & \none[\vdots]&\none[\vdots]&\none&\none[\vdots]&\none[\vdots]&\none&\none[\vdots]&\none& \none[\vdots] & \none[\vdots] & \none[\vdots] & \none[\vdots] & \none & \none[\vdots] & \none & \none[\vdots] & \none & \none[\vdots] & \none[\vdots] &\none &\none[\vdots] &\none[\vdots]\\
\none[i_2] &\none & &&\none[\cdots]& &&*(gray)&*(gray)&\none[\cdots] &*(gray)&& \none[\cdots]&\times&&\none[\cdots]&\bullet&\none[\cdots] && & *(gray) & *(gray) &\none[\cdots] &*(gray) &\none[\cdots] &*(gray) &\none[\cdots] & *(gray) & *(gray)  & \none[\cdots] & *(gray) & & \none[\cdots]\\
\none &\none & \none[\vdots] &\none[\vdots] &\none &\none[\vdots] &\none[\vdots] &\none[\vdots] & \none[\vdots] & \none & \none[\vdots]&\none[\vdots]&\none&\none[\vdots]&\none[\vdots]&\none&\none[\vdots]&\none& \none[\vdots] & \none[\vdots] & \none[\vdots] & \none[\vdots] & \none & \none[\vdots] & \none & \none[\vdots] & \none & \none[\vdots] & \none[\vdots] &\none &\none[\vdots] &\none[\vdots]\\
\none[a_1-1] &\none & &&\none[\cdots]& &&*(gray)&*(gray)&\none[\cdots] &*(gray)&&\none[\cdots]&&&\none[\cdots] &&\none[\cdots] && & *(gray) & *(gray) &\none[\cdots] &*(gray) &\none[\cdots] &*(gray) &\none[\cdots] & *(gray) & *(gray)  & \none[\cdots] & *(gray) & & \none[\cdots]\\
\none[a_1] &\none &&& \none[\cdots] &&&&&\none[\cdots] &&&\none[\cdots]&&&\none[\cdots]&\times&\none[\cdots]&&&&& \none[\cdots] &&\none[\cdots] &&\none[\cdots] &&\blacksquare & \none[\cdots] && \none[\cdots]\\
\none[a_1+1] &\none &&& \none[\cdots] &&& *(gray) & *(gray) &\none[\cdots] &*(gray)&*(gray)&\none[\cdots]&*(gray)&*(gray)&\none[\cdots]&*(gray)&\none[\cdots] & *(gray) & *(gray) & *(gray) & *(gray) &\none[\cdots] &*(gray) &\none[\cdots] &*(gray) &\none[\cdots] & *(gray) & & \none[\cdots] && \none[\cdots]\\
\none[a_1+2] &\none &&& \none[\cdots] &&& *(gray) & *(gray) &\none[\cdots] &*(gray)&*(gray)&\none[\cdots]&*(gray)&*(gray)&\none[\cdots]&*(gray)&\none[\cdots] & *(gray) & *(gray) & *(gray) & *(gray) &\none[\cdots] &*(gray) &\none[\cdots] &*(gray) &\none[\cdots] & *(gray) & & \none[\cdots] && \none[\cdots]\\
\none &\none & \none[\vdots] &\none[\vdots] &\none &\none[\vdots] &\none[\vdots] &\none[\vdots] & \none[\vdots] & \none & \none[\vdots]&\none[\vdots]&\none&\none[\vdots]&\none[\vdots]&\none&\none[\vdots]&\none& \none[\vdots] & \none[\vdots] & \none[\vdots] & \none[\vdots] & \none & \none[\vdots] & \none & \none[\vdots] & \none & \none[\vdots] & \none[\vdots] &\none &\none[\vdots]\\
\none[a_2-1] &\none &&& \none[\cdots] &&& *(gray) & *(gray) &\none[\cdots] &*(gray)&*(gray)&\none[\cdots]&*(gray)&*(gray)&\none[\cdots]&*(gray)&\none[\cdots] & *(gray) & *(gray) & *(gray) & *(gray) &\none[\cdots] &*(gray) &\none[\cdots] &*(gray) &\none[\cdots] & *(gray) & & \none[\cdots] && \none[\cdots]\\
\none[a_2] &\none &&& \none[\cdots] &&\times&&& \none[\cdots] &&&\none[\cdots]&&&\none[\cdots]&&\none[\cdots] && \blacksquare&&&\none[\cdots] &&\none[\cdots] &&\none[\cdots] &&& \none[\cdots] && \none[\cdots]\\
\none[a_2+1] &\none &*(gray) &*(gray) & \none[\cdots] &*(gray) &*(gray) &*(gray) &*(gray) & \none[\cdots] &*(gray)&*(gray)&\none[\cdots]&*(gray)&*(gray)&\none[\cdots]&*(gray)&\none[\cdots] &*(gray) &&&&\none[\cdots] &&\none[\cdots] &&\none[\cdots] &&& \none[\cdots] && \none[\cdots]\\
\none[a_2+2] &\none &*(gray) &*(gray) & \none[\cdots] &*(gray) &*(gray) &*(gray) &*(gray) & \none[\cdots] &*(gray)&*(gray)&\none[\cdots]&*(gray)&*(gray)&\none[\cdots]&*(gray)&\none[\cdots] &*(gray) &&&&\none[\cdots] &&\none[\cdots] &&\none[\cdots] &&& \none[\cdots] && \none[\cdots]\\
\none &\none & \none[\vdots] &\none[\vdots] &\none &\none[\vdots] &\none[\vdots] &\none[\vdots] & \none[\vdots] & \none & \none[\vdots]&\none[\vdots]&\none&\none[\vdots]&\none[\vdots]&\none&\none[\vdots]&\none& \none[\vdots] & \none[\vdots] & \none[\vdots] & \none[\vdots] & \none & \none[\vdots] & \none & \none[\vdots] & \none & \none[\vdots] & \none[\vdots] &\none &\none[\vdots]\\
\none[i'_1] &\none &*(gray) &*(gray) & \none[\cdots] &*(gray) &*(gray) &*(gray) &*(gray) & \none[\cdots] &*(gray)&*(gray)&\none[\cdots]&*(gray)&*(gray)&\none[\cdots]&*(gray)&\none[\cdots] &*(gray) &\times&&&\none[\cdots] &\bullet&\none[\cdots] &&\none[\cdots] &&& \none[\cdots] && \none[\cdots]\\
\none &\none & \none[\vdots] &\none[\vdots] &\none &\none[\vdots] &\none[\vdots] &\none[\vdots] & \none[\vdots] & \none & \none[\vdots]&\none[\vdots]&\none&\none[\vdots]&\none[\vdots]&\none&\none[\vdots]&\none& \none[\vdots] & \none[\vdots] & \none[\vdots] & \none[\vdots] & \none & \none[\vdots] & \none & \none[\vdots] & \none & \none[\vdots] & \none[\vdots] &\none &\none[\vdots]\\
\none[i'_2] &\none &*(gray) &*(gray) & \none[\cdots] &*(gray) &*(gray) &*(gray) &*(gray) & \none[\cdots] &*(gray)&*(gray)&\none[\cdots]&*(gray)&*(gray)&\none[\cdots]&*(gray)&\none[\cdots] &*(gray) &&&&\none[\cdots] &\times&\none[\cdots] &\bullet&\none[\cdots] &&& \none[\cdots] && \none[\cdots]\\
\none &\none & \none[\vdots] &\none[\vdots] &\none &\none[\vdots] &\none[\vdots] &\none[\vdots] & \none[\vdots] & \none & \none[\vdots]&\none[\vdots]&\none&\none[\vdots]&\none[\vdots]&\none&\none[\vdots]&\none& \none[\vdots] & \none[\vdots] & \none[\vdots] & \none[\vdots] & \none & \none[\vdots] & \none & \none[\vdots] & \none & \none[\vdots] & \none[\vdots] &\none &\none[\vdots]\\
\none[a_3-1] &\none &*(gray) &*(gray) & \none[\cdots] &*(gray) &*(gray) &*(gray) &*(gray) & \none[\cdots] &*(gray)&*(gray)&\none[\cdots]&*(gray)&*(gray)&\none[\cdots]&*(gray)&\none[\cdots] &*(gray) &&&&\none[\cdots] &&\none[\cdots] &&\none[\cdots] &&& \none[\cdots] && \none[\cdots]\\
\none[a_3] &\none &&& \none[\cdots] &&\blacksquare &&&\none[\cdots]&&&\none[\cdots]&&&\none[\cdots]&&\none[\cdots] &&&&&\none[\cdots]  &&\none[\cdots] &\times&\none[\cdots] &&& \none[\cdots] &\\
\none[a_3+1] &\none &&& \none[\cdots] &&&&&\none[\cdots]&&\bullet&\none[\cdots]&&&\none[\cdots]&&\none[\cdots] &&&&&\none[\cdots]  &&\none[\cdots] &&\none[\cdots] &&\times& \none[\cdots] &\\
\none[a_3+2] &\none &&& \none[\cdots] &&&&&\none[\cdots]&&&\none[\cdots]&&*(gray)&\none[\cdots]&*(gray)&\none[\cdots]
&*(gray)&*(gray) & *(gray) &*(gray) &\none[\cdots] &*(gray) &\none[\cdots] &*(gray) &\none[\cdots] &*(gray) &*(gray) & \none[\cdots] \\
\none[a_3+3] &\none &&& \none[\cdots] &&&&&\none[\cdots]&&&\none[\cdots]&&*(gray)&\none[\cdots]&*(gray)&\none[\cdots]
&*(gray)&*(gray) & *(gray) &*(gray) &\none[\cdots] &*(gray) &\none[\cdots] &*(gray) &\none[\cdots] &*(gray) &*(gray) & \none[\cdots] \\
\none &\none & \none[\vdots] &\none[\vdots] &\none &\none[\vdots] &\none[\vdots] &\none[\vdots] & \none[\vdots] & \none & \none[\vdots]&\none[\vdots]&\none&\none[\vdots]&\none[\vdots]&\none&\none[\vdots]&\none& \none[\vdots] & \none[\vdots] & \none[\vdots] & \none[\vdots] & \none & \none[\vdots] & \none & \none[\vdots] & \none & \none[\vdots] & \none[\vdots]\\
\end{ytableau}
\caption{The squares marked with a solid black box are the elements of the chosen copy of $J_3$ for a separable, valid transversal $T$ of $J$-type 3.  The bullets mark some other elements of $T$, while the crosses mark new elements of $\phi(T)$, i.e. elements of $\phi(T) \setminus T$.  The gray squares are free of elements of $T$ (and $\phi(T)$).
We suppose that $\Gamma_{[1,a_1)}^{[b_{a_3},b_{a_1}]}(T) = \{i_1,i_2\}$ and $\Gamma_{(a_2,a_3)}^{[b_{a_3},b_{a_1}]}(T) = \{i'_1,i'_2\}$.}
\label{fig:JType3}
\end{figure}
See Figure~\ref{fig:JType3}.  We first prove that $\phi(T)$ is a valid transversal of $\mathcal{Y}$.  Because
$Y_{a_3+1} = Y_{a_3} \ge b_{a_1}$, the set $\phi(T)$ is a transversal of $Y$.
If the sets $\{i,i+1\}$ and $\Gamma_{[1,a_1] \cup [a_2,a_3+1]}^{[b_{a_3},b_{a_1}]}(T)$ are disjoint,
then $b_i = c_i$ and $b_{i+1} = c_{i+1},$ and thus
$i$ is an ascent (resp. descent) of $\phi(T)$ if and only if it is an ascent (resp. descent) of $T$.
If $x, x+1 \in \Gamma_{[1,a_1] \cup [a_2,a_3]}^{[b_{a_3},b_{a_1}]}(T)$ with $x \not= a_1,a_3-1$,
then by Lemma~\ref{JType3L1}, we have $b_x < b_{x+1}$ and $c_x < c_{x+1}$.
If $x \in \Gamma_{[1,a_1] \cup [a_2,a_3]}^{[b_{a_3},b_{a_1}]}(T)$,
but $x+1$ is not (and $x \not= a_1,a_3$), then we have $b_{a_3} < b_x,c_x < b_{a_1}$ and $b_{x+1} \notin [b_{a_3},b_{a_3+1}],$ which implies that $x$ is an ascent (resp. descent) of $\phi(T)$ if and only if
it is (resp. descent) of $T$.  Similar logic holds if we replace $x+1$ by $x-1$ and require that
$x \not= a_2$.  Because $b_{a_1+1}, b_{a_2-1} \notin [b_{a_3},b_{a_1}]$ by Lemma~\ref{EPhi},
we have $b_{a_1+1} < b_{a_1}$ if and only if $b_{a_1+1} < c_{a_1}$, and $b_{a_2-1} < b_{a_2}$
if and only if $b_{a_2-1} < c_{a_2}$.  If $a_2 \not= a_1+1$, then
we have $c_{a_1+1} = b_{a_1+1}$ and $c_{a_2-1} = b_{a_2-1}$, which implies that
that $a_1$ is an ascent (resp. descent) of $\phi(T)$ if and only if it is an ascent (resp. descent)
of $T$, and similarly for $a_2-1$.  If $a_2 = a_1 + 1$, then $a_1 = a_2-1$ is a descent of both $T$ and $\phi(T)$.
Also, we have $a_3 \in A$, and because $\mathcal{Y}$ is 1-alternating, we have $a_3 + 1 \in D$ and $a_3-1 \notin A$.
By definition of $F$-type we have $a_3-1 \notin D$ or ($b_{a_3-1} > b_{a_1} > c_{a_3})$.
Furthermore, we have $c_{a_3} < b_{a_1} = c_{a_3+1}$ and $c_{a_3+1} = b_{a_1} > b_{a_3+1} > b_{a_3+2} = c_{a_3+2}$
by the definition $J$-type and Lemma~\ref{EPhi}.
It follows that $\phi(T)$ is a valid transversal of $\mathcal{Y}$.

Next, we prove that $h_F(\phi(T)) = (a_1,a_2,a_3+1)$.  It is clear
that $(a_1,a_2,a_3+1) \in V(\phi(T))$ and $S(a_1,a_2,a_3+1) = (a_3,a_1,a_2) = \#(a)$.
Let $(d_1,d_2,d_3) \in V(\phi(T))$
and suppose for sake of contradiction that $S(d) > \#(a)$ in the lexicographic order.
If $d_3 > a_3+1$ and $b_{d_3} > b_{a_3+1}$, then $(a_3+1,a_3+2,d_3) \in V(T)$ (because
$a_3 +1 \in D$ due to the fact that $\mathcal{Y}$ is 1-alternating), which contradicts
the separability of $T$ because $S(a_3+1,a_3+2,d_3) \ge (a_3+2,a_1,0)$ in the lexicographic order.
If $d_3 < a_3 + 1$, $b_{d_3} < b_{a_3+1}$, and $d_1,d_2 \not= a_2$, then
we have $b_{d_i} = c_{d_i}$ for all $i$, which implies that $(d_1,d_2,d_3) \in V(T)$, contradiction.
If $d_3 > a_3 + 1$, $b_{d_3} < b_{a_3+1}$, and $d_1 = a_2$, then Lemma~\ref{EPhi}
yields that $d_2 > a_3$, which implies that $(a_3,d_2,d_3) \in V(T)$, contradiction.
If $d_3 > a_3 + 1$, $b_{d_3} < b_{a_3+1}$, and $d_1 = a_2$, then $(d_1,a_3,d_2) \in V(T)$,
contradiction.  Hence, we may assume that $d_3 \le a_3 + 1$.  Because $S(d) > \#(a)$
in the lexicographic order, we have ($d$ is of $F$-type 2, $a_3 - 1 \in D$, $d_2 = a_3 - 2$, and $d_3 = a_3 - 1$) or
$d_3 = a_3 + 1$.  In the former case, Lemma~\ref{JType3L1} implies that
$b_{d_3} \notin [b_{a_3}, b_{a_1}]$, and the fact that $a_3-1 > a_2$ then yields that $b_{d_3} > b_{a_1}$ by Lemma~\ref{EPhi}.  Because $d_2 = a_3 - 2 > a_2$, we have $b_{d_2} \ge b_{a_2}$ by Lemma~\ref{EPhi}.
Regardless, if $b_{d_1} < b_{a_1}$, then we have $d_1 > a_2$, and it follows that
$c_{d_1} < c_{d_2}$ by Lemma~\ref{JType3L1}, contradiction.  Hence, we have $b_{d_1} > b_{a_1}$,
which implies that $b_{d_1} = c_{d_1}$.
It is clear that $b_{d_2} \le \max\{c_{d_2}, b_{a_1}\}$, and therefore, we have $(d_1,d_2,d_3) \in V(T)$,
contradiction.  Hence, we may assume that $d_3 = a_3 + 1$.
If $d_2 = a_3$, then by Lemma~\ref{EPhi} we have $d_1 \ge a_2$, but this implies that $d_1 \in \Gamma_{[a_2,a_3)}^{[b_{a_3},b_{a_1}]}(T)$,
which contradicts Lemma~\ref{JType3L1}.  We may assume that $d_2 \not= a_3$,
which implies that $d$ is of $J$-type 3 and $S(d) = (a_3,d_1,d_2)$.  The fact that
$S(d) > \#(a)$ in the lexicographic order implies that $d_1 \ge a_1$. If $c_{d_1} > b_{a_3},$, then by Lemma~\ref{EPhi}, we have that $d_1 < a_1$, contradiction.  If $c_{d_1} = b_{a_3}$, then
we have $b_{d_2} = c_{d_2} < b_{a_3}$ and $a_2 < d_2 < a_3$, but the existence of such a $d_2$
contradicts Lemma~\ref{EPhi}.  If $c_{d_1} < b_{a_3}$, then we have $c_{d_i} = b_{d_i}$ for $i = 1,2$
and $(d_1,d_2,a_3+1) \in V(T)$, contradiction.  Thus, we have $h_F(\phi(T)) = (a_1,a_2,a_3+1)$.
It is clear that $\phi(T)$ is of $F$-type 3.  We have $c_{a_2} = b_{a_3}$ and $c_{a_3+1} = b_{a_1}$,
which implies that
\[\psi(\phi(T)) = \theta_{[a_2,a_3]}^{[b_{a_3},b_{a_1}]}\left(\theta_{[1,a_1] \cup \{a_3+1\}}^{[b_{a_3},b_{a_1}]}\left( \omega_{[1,a_1] \cup \{a_3+1\}}^{[b_{a_3},b_{a_1}]}\left(\omega_{[a_2,a_3]}^{[b_{a_3},b_{a_1}]}(T)\right)\right)\right) = T,\]
as desired.

We prove that if $e = (e_1,e_2,e_3) \in h_J(\phi(T))$, then $e_3 > a_3$.  Suppose for sake
of contradiction that $e_3 \le a_3$.
If $c_{e_3} > b_{a_1}$, then $(e_1,e_2,e_3) \in U(T)$, contradiction.
If $b_{a_2} \le c_{e_3} < b_{a_1},$ then by Lemma~\ref{EPhi},
we have $a_2 < e_3 < a_3$.  By Lemma~\ref{JType3L1} and because $d_2 > d_3$ with $c_{d_2} > c_{d_3}$,
we have $b_{d_2} > b_{a_1}$ or $d_2 < a_2$; however, the latter case implies that $b_{d_2} > b_{a_1}$
by Lemma~\ref{EPhi} again.  Then, $(d_1,d_2,d_3) \in U(T)$, contradiction./
If $b_{a_3} < c_{e_3} < b_{a_2}$, then by Lemma~\ref{EPhi} we have $e_3 \le a_1$.
By Lemma~\ref{JType3L1}, we have $b_{d_2} > b_{a_1}$, which implies that
$(d_1,d_2,d_3) \in U(T)$, contradiction.
If $c_{e_3} = b_{a_3}$, then
by Lemma~\ref{JType3L1}, we have $b_{d_1} > b_{a_1}$ (because if $b_{d_1} < b_{a_1},$
then $c_{d_1} < c_{d_2}$ by Lemma~\ref{JType3L1}, contradiction).  This implies
that $(d_1,a_1,a_2) \in U(T)$, contradiction.
If $c_{e_3} < b_{a_3}$, then by Lemma~\ref{EPhi}, we have $e_3 < a_2$.
By Lemma~\ref{JType3L1} and because $e_1,e_2 < a_2$, at most one of $c_{e_1},c_{e_2}$ can be an element
of $[b_{a_3},b_{a_1}]$, which implies that $(e_1,e_2,e_3) \in U(T)$, contradiction.  The separability
of $\phi(T)$ follows.

Let $m = \min \Gamma_{[1,a_1]}^{[b_{a_3},b_{a_1}]}(T)$.
We have $b_i = c_i$ for all $i < m$.  By Lemma~\ref{JType3L1} and because $b_{a_1} > b_{a_2} > b_{a_3+1}$ by Lemma~\ref{EPhi}, we have $b_m > b_{a_3+1} = c_{i_1}$, and therefore,
$(b_1,b_2,\ldots,b_n) > (c_1,c_2,\ldots,c_n)$ in the lexicographic order, as desired.
\end{asparaenum}
\end{proof}

\section{Proof of Proposition~\ref{FullPsi}}
\label{sec:PropFullPsi}

The proof is similar to the proof of Proposition~\ref{FullPhi}.  First, we define
$E_\psi(T) \subseteq Y$, which is the analogue $E_\phi(T)$.  Let $T = \{(i,b_i)\}$
be a separable, valid transversal of $\mathcal{Y}$ that contains $F_3$, and let
$h_F(T) = (a_1,a_2,a_3).$  Then, let
\[E_\psi(T) = \left(\left([1,a_1) \times [b_{a_1}, Y_{a_3}]\right) \cup \left((a_1,a_2) \times [b_{a_2},b_{a_3}]\right)
    \cup \left((a_2,a_3) \times [1,b_{a_1}]\right) \cup \left((a_3, \infty) \times (b_{a_1}, \infty)\right)\right) \cap Y.\]
Once again, the critical property of $E_\psi(T)$ is the following lemma.

\begin{lemma} \label{EPsi}
If $T$ is a separable valid transversal of $\mathcal{Y}$ that contains $F_3$, then
$E_\psi(T)$ does not contain any element of $T$.
\end{lemma}
\begin{proof}
If $(i,b_i) \in [1,a_1) \times [b_{a_1}, Y_{a_3}]$, then
$(i,a_1,a_2) \in U(T)$, which contradicts the separability of $T$.
If $(i,b_i) \in (a_1,a_2) \times [b_{a_2},b_{a_3}]$,
then $(i,a_2,a_3) \in V(T)$, and $S(i,a_2,a_3) > S(a_1,a_2,a_3)$ in the lexicographic order.
If $(i,b_i) \in (a_2,a_3) \times [1,b_{a_1}]$,
then $(a_1,i,a_3) \in V(T)$, and $S(a_1,i,a_3) > S(a_1,a_2,a_3)$ in the lexicographic order.
Both contradict the definition of $h_F$.
If $(i,b_i) \in (a_3, \infty) \times (b_{a_1}, \infty)$, then
$v = (a_2,a_3,i)$ is a copy of $F_3$ in $T$.  If $i \in A$, replace
$v$ by $(a_2,a_3,i+1)$.  Then, we have $v \in V(T)$, and
$S(v) > S(h_F(T))$ in the lexicographic order,
which contradicts the definition of $h_F$.
\end{proof}

The analogue of Lemma~\ref{JType3L1} is the following lemma, which will be used repeatedly
in the proof of Proposition~\ref{FullPsi} for the case in which $T$ has $F$-type 3.

\begin{lemma}
\label{FType3L1}
Let $T = \{(i,b_i)\}$ be a separable, valid transversal of $\mathcal{Y}$ of $F$-type 3,
and let $\psi(T) = \{(i,c_i)\}$.
\begin{enumerate}[(a)]
\item
Let $\Gamma_{[1,a_1]}^{[b_{a_3},b_{a_1}]}(T) = \{i_1 < i_2 < \cdots < i_k\}$; then $b_{i_1} < b_{i_2} < \cdots < b_{i_k}$ and $c_{i_1} < c_{i_2} < \cdots < c_{i_k}$.  In particular, if $i \in \Gamma_{[1,a_1]}^{[b_{a_2},b_{a_3}]}(T)$, then $b_i \le b_{a_1}.$
\item Let $\Gamma_{[a_2,a_3)}^{[b_{a_2},b_{a_3})}(T) = \{i_1 < i_2 < \cdots < i_k\}$, then $b_{a_1} < b_{i_1} < b_{i_2} < \cdots < b_{i_k}$ and $c_{i_1} < c_{i_2} < \cdots < c_{i_k}$.  Furthermore, we have
    $b_{a_1} < c_{i_2} < \cdots < c_{i_k} < c_{a_3}$.
\end{enumerate}
\end{lemma}
\begin{proof}
First, we prove part (a).  If $j < j'$ with $b_{i_j} > b_{i_j'}$, then $(i_j,i_{j'},a_3-1) \in U(T)$,
which contradicts the separability of $T$.
Because $c_{i_j} = b_{i_{j-1}}$, to prove that $c_{i_1} < c_{i_2} < \cdots < c_{i_k}$
it suffices to prove that $c_{i_1} < c_{i_2}$.  But, if $c_{i_1} > c_{i_2}$, we have
$b_{a_3} = c_{i_1} > c_{i_2} = b_{i_1}$, and therefore $(i_1,a_3-1,a_3) \in V(T)$, but
$S(i_1,a_3-1,a_3) = (a_3+1,i_1+1,0)$, which contradicts the definition of $h_F$.
The last sentence follows because $b_{i_k} = b_{a_1}$.

The proof of part (b) is similar.  Let $i_0 = a_1$.  If $j < j'$ with $b_{i_j} > b_{i_{j'}}$, then
$(i_j,i_{j'},a_3-1) \in U(T)$, which contradicts the separability of $T$.  To prove that $c_{i_1} < c_{i_2} < \cdots < c_{i_k}$, it suffices to prove that $c_{i_1} < c_{i_2}$, but this is clear because $c_{i_1} = b_{a_3} < b_{i_1} = c_{i_2}$.  Let $i_{k+1} = a_3$.  The last sentence follows because $c_{i_{j+1}} = b_{i_j}$ for $j \in [k]$.
\end{proof}

The following additional lemma will be also used in proof of Proposition~\ref{FullPsi}
for $T$ of $F$-type 3.

\begin{lemma}
\label{FType3L2}
Let $T$ be a separable, valid transversal of $\mathcal{Y}$ of $F$-type 3,
let $h_F(T) = (a_1,a_2,a_3)$,
and let $m = \min \Gamma_{[1,a_1]}^{[b_{a_2},b_{a_3}]}(T)$.  Then,
the set $(a_3,\infty) \times (b_m,\infty) \cap Y$ does not contain any element of $T$.
\end{lemma}
\begin{proof}
Suppose for sake of contradiction that $(i, b_i) \in (a_3,\infty) \times (b_m,\infty) \cap Y$.
Then, $(m,a_2,i) \in V(T)$ and $S(m,a_2,i) \ge (a_3,m,0) > (a_3-1,a_1,a_2) = S(h_F(T))$ in the lexicographic
order, which contradicts the definition of $h_F$.
\end{proof}

\begin{proof}[Proof of Proposition~\ref{FullPsi}]
If $x = (x_1,x_2,x_3)$ is a copy of $F_3$ in $T$, then either $x_3 \notin A$
and $x \in V(T)$ or $x_3 \in A$ and $(x_1,x_2,x_3+1) \in V(T)$.  Thus, if $T$
contains $F_3$, then $\psi(T)$ is defined.
We do casework on the $F$-type of $T$.  Let $h_F(T) = a = (a_1,a_2,a_3)$.

\begin{asparaenum}
\renewcommand{\labelenumi}{\bf{$F$-type \arabic{enumi}.}}
\item
\begin{figure}
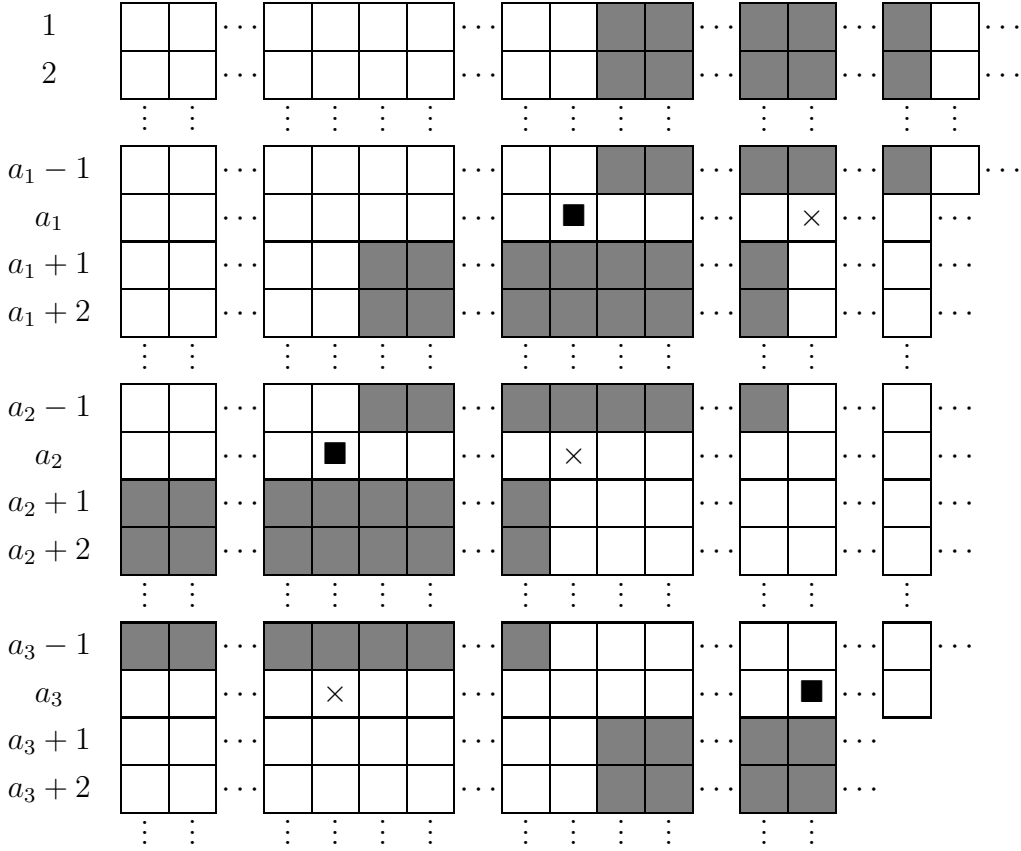

\ytableausetup{boxsize=normal}
\begin{ytableau}
\none[1] &\none & &&\none[\cdots]& &&&&\none[\cdots] && & *(gray) & *(gray) &\none[\cdots] & *(gray) & *(gray)  & \none[\cdots] & *(gray) & & \none[\cdots]\\
\none[2] &\none & &&\none[\cdots]& &&&&\none[\cdots] && & *(gray) & *(gray) &\none[\cdots] & *(gray) & *(gray)  & \none[\cdots] & *(gray) & & \none[\cdots]\\
\none &\none & \none[\vdots] &\none[\vdots] &\none &\none[\vdots] &\none[\vdots] &\none[\vdots] & \none[\vdots] & \none & \none[\vdots] & \none[\vdots] & \none[\vdots] & \none[\vdots] & \none & \none[\vdots] & \none[\vdots] &\none &\none[\vdots] &\none[\vdots]\\
\none[a_1-1] &\none & &&\none[\cdots]& &&&&\none[\cdots] && & *(gray) & *(gray) &\none[\cdots] & *(gray) & *(gray)  & \none[\cdots] & *(gray) & & \none[\cdots]\\
\none[a_1] &\none &&& \none[\cdots] &&&&&\none[\cdots]&& \blacksquare & & & \none[\cdots] & & \times & \none[\cdots] && \none[\cdots]\\
\none[a_1+1] &\none &&& \none[\cdots] &&& *(gray) & *(gray) &\none[\cdots] & *(gray) & *(gray) & *(gray) & *(gray) & \none[\cdots] & *(gray) & & \none[\cdots] && \none[\cdots]\\
\none[a_1+2] &\none &&& \none[\cdots] &&& *(gray) & *(gray) &\none[\cdots] & *(gray) & *(gray) & *(gray) & *(gray) & \none[\cdots] & *(gray) & & \none[\cdots] && \none[\cdots]\\
\none &\none & \none[\vdots] &\none[\vdots] &\none &\none[\vdots] &\none[\vdots] &\none[\vdots] & \none[\vdots] & \none & \none[\vdots] & \none[\vdots] & \none[\vdots] & \none[\vdots] & \none & \none[\vdots] & \none[\vdots] &\none &\none[\vdots]\\
\none[a_2-1] &\none &&& \none[\cdots] &&& *(gray) & *(gray) &\none[\cdots] & *(gray) & *(gray) & *(gray) & *(gray) & \none[\cdots] & *(gray) & & \none[\cdots] && \none[\cdots]\\
\none[a_2] &\none &&& \none[\cdots] &&\blacksquare &&& \none[\cdots] && \times&&&\none[\cdots] &&& \none[\cdots] && \none[\cdots]\\
\none[a_2+1] &\none &*(gray) &*(gray) & \none[\cdots] &*(gray) &*(gray) & *(gray) & *(gray) & \none[\cdots] & *(gray) &&&&\none[\cdots] &&& \none[\cdots] && \none[\cdots]\\
\none[a_2+2] &\none &*(gray) &*(gray) & \none[\cdots] &*(gray) &*(gray) & *(gray) & *(gray) & \none[\cdots] & *(gray) &&&&\none[\cdots] &&& \none[\cdots] && \none[\cdots]\\
\none &\none & \none[\vdots] &\none[\vdots] &\none &\none[\vdots] &\none[\vdots] &\none[\vdots] & \none[\vdots] & \none & \none[\vdots] & \none[\vdots] & \none[\vdots] & \none[\vdots] & \none & \none[\vdots] & \none[\vdots] &\none &\none[\vdots]\\
\none[a_3-1] &\none &*(gray) &*(gray) & \none[\cdots] &*(gray) &*(gray) & *(gray) & *(gray) & \none[\cdots] & *(gray) &&&&\none[\cdots] &&& \none[\cdots] && \none[\cdots]\\
\none[a_3] &\none &&& \none[\cdots] &&\times &&&\none[\cdots]&&&&&\none[\cdots] && \blacksquare & \none[\cdots] &\\
\none[a_3+1] &\none &&& \none[\cdots] &&&&&\none[\cdots]&& & *(gray) &*(gray) &\none[\cdots] &*(gray) &*(gray) & \none[\cdots] \\
\none[a_3+2] &\none &&& \none[\cdots] &&&&&\none[\cdots]&& & *(gray) &*(gray) &\none[\cdots] &*(gray) &*(gray) & \none[\cdots] \\
\none &\none & \none[\vdots] &\none[\vdots] &\none &\none[\vdots] &\none[\vdots] &\none[\vdots] & \none[\vdots] & \none & \none[\vdots] & \none[\vdots] & \none[\vdots] & \none[\vdots] & \none & \none[\vdots] & \none[\vdots]\\
\end{ytableau}
\caption{The squares marked with a solid black box are the elements of the chosen copy of $F_3$ for a separable, valid transversal $T$ of $F$-type 1.  The crosses mark new elements of $\psi(T)$, i.e. elements of $\psi(T) \setminus T$, and the gray squares are free of elements of $T$ (and $\psi(T)$).}
\label{fig:FType1}
\end{figure}
See Figure~\ref{fig:FType1}.
First, we prove that $\psi(T)$ is a valid transversal of $\mathcal{Y}$.
Because $Y_{a_3} \ge b_{a_1}$, the set $\psi(T)$ is a transversal of $Y$.
If the sets
$\{i,i+1\}$ and $\{a_1,a_3\}$ are disjoint, then $b_i = c_i$ and $b_{i+1} = c_{i+1}$,
which implies that $i$ is an ascent (resp. descent) of $\psi(T)$ if and only if
it is an ascent (resp. descent) of $T$.  By Lemma~\ref{EPsi}, we have
$b_{a_1-1},b_{a_1+1}\notin (b_{a_1},b_{a_3}) = (b_{a_1},c_{a_1}))$, and
$b_{a_2-1},b_{a_2+1} \notin (b_{a_2},c_{a_2})$.  Thus, $t = a_1-1$ is an ascent
(resp. descent) of $\psi(T)$ if and only if it is an ascent (resp. descent) of $T$.
If $a_2 \not= a_1 + 1$, then the same holds for $t = a_1$ and $t = a_2 - 1$,
and if $a_2 = a_1 + 1$, then $a_1 = a_2 - 1$ is a descent of both $T$ and $\psi(T)$.
If $a_2 \not= a_3 - 1$, then $a_2$ is an ascent (resp. descent) of $\psi(T)$ if and only if
it is an ascent (resp. descent) of $T$.  By definition of $F$-type and $V(T)$, we have $a_3-1, a_3 \notin A$,
and the fact that $\mathcal{Y}$ is 1-alternating implies that $a_3-1,a_3 \notin D$.  Furthermore,
if $a_2 = a_3 - 1$, then $a_2 \notin A, D$.  It follows that $\psi(T)$ is a valid transversal of $\mathcal{Y}$.

Next, we prove that $h_J(\psi(T)) = (a_1,a_2,a_3)$.  It is clear that $(a_1,a_2,a_3) \in U(\psi(T))$,
and suppose for sake of contradiction that $d \in U(\psi(T))$ with $\#(d) < \#(a)$ in the lexicographic
order.  If $d_3 < a_1$ or $b_{d_3} > b_{a_3}$, then $b_{d_i} = c_{d_i}$ for all $i \in [3]$
and $d \in U(T)$, contradiction.  If $d_3 = a_1$, then $c_{d_3} > b_{d_3}$ and
$b_{d_i} = c_{d_i}$ for all $i\in [2]$, which implies that $d \in U(T)$, contradiction.
If $d_3 = a_2$ and $d_1 = a_1$, then $b_{a_1} < b_{d_2} = c_{d_2} < b_{a_3}$ with $a_1 < d_2  < a_2$,
which contradicts Lemma~\ref{EPsi}.  Hence, if $d_3 = a_2$, we have $b_{d_1} = c_{d_1}$,
as well as $c_{d_2} \ge b_{d_2} > b_{a_2}$, which implies that $d \in U(T)$, contradiction.
If $a_1 < d_3 < a_2$ with $b_{d_3} < b_{a_3}$, then Lemma~\ref{EPsi} yields that $b_{d_3} < b_{a_2}$.
In this case, if $d_1 = a_1$, then Lemma~\ref{EPsi} yields that $b_{d_2} < b_{a_2}$ and therefore
$d \in U(T)$, contradiction.  Furthermore, if $a_1 < d_3 < a_2$ with $b_{d_3} < b_{a_3}$
and $d_1 \not= a_1$, then we have $b_{d_1} = c_{d_1}$ and \[c_{b_2} \ge b_{d_2} \ge \min\{b_{a_1},c_{d_2}\} > c_{d_3} = b_{d_3},\]
which implies that $d \in U(T)$, contradiction.  If $a_2 < d_3 < a_3$ with $b_{d_3} < b_{a_3}$
and $d_1 = a_1$, then by Lemma~\ref{EPsi}
we have $c_{d_3} = b_{d_3} > b_{a_1} = c_{d_2}$, which yields that $d_2 \not= a_2$.  Furthermore,
by Lemma~\ref{EPsi} again and because $b_{a_1} < b_{d_2} < c_{a_1} = b_{a_3}$, we have $d_2 > a_2$,
but the fact that $(a_1,d_2,a_3) \in V(T)$ contradicts the definition of $h_F$.  If $a_2 < d_3 < a_3$ with $d_2 = a_1$, then we have $(d_1,a_1,a_2) \in U(T)$, contradiction.  If $a_2 < d_3 < a_3$ with $a_1 \notin \{d_1,d_2\}$,
then $d \in U(T)$, contradiction.  It is then clear that $\psi(T)$ is of $J$-type 1, and the fact
that $\phi(\psi(T)) = T$ follows.

We prove that if $e = (e_1,e_2,e_3) \in V(\psi(T))$, then $S(e) \le S(a)$ in the lexicographic order.
Suppose for sake of contradiction that $S(e) > S(a)$ in the lexicographic order.
First, suppose that $e_3 > a_3$ in the lexicographic order;
it is clear that $b_{e_3} = c_{e_3}$.  Furthermore, because $a_3 \notin A$, the first
component of $S(e)$ must be greater than $a_3$.
By Lemma~\ref{EPsi}, we have $b_{e_3} < b_{a_1}$.  If $e_1 = a_3$, then
we have $(a_2,e_2,e_3) \in V(T)$, but the first component of $S(a_2,e_2,e_3)$
is greater than $a_3$, which contradicts the definition of $h_F$.  If $e_2 = a_3$,
then we have $b_{a_2} < c_{e_1} < b_{a_1}$, which yields that $b_{e_1} = c_{e_1}$.
By Lemma~\ref{EPsi}, we have $e_1 < a_1$, and hence $(e_1,a_2,e_3) \in V(T)$, but
the first component of $S(e_1,e_2,e_3)$ is greater than $a_3$, contradiction.  If $a_3 \notin \{e_1,e_3\}$,
then $b_{e_i} = c_{e_i}$ for all $i$ and thus $e \in V(T)$, contradiction.  Hence, we may assume
that $e_3 \le a_3$.  Because $S(e) > S(a)$ in the lexicographic order, either
($F$ is of $F$-type 2, $a_3 - 2 \in A$, $e_3 = a_3 - 1$ and $e_2 = a_3 - 2$) or ($F$ is of
$F$-type 1 and $e_3 = a_3$).  In the former case, because $a_3 - 2 \ge a_2$, we have $a_3 - 1 \in D$, we have $b_{a_3} < b_{a_3-1}$, which implies that $b_{a_3-1} = c_{a_3-1}$.
Additionally, by Lemma~\ref{EPsi} and because $c_{a_2} = b_{a_1}$, we have
$c_{a_3 - 2} \ge b_{a_1}$, and because $S(e) > S(a)$ in the lexicographic order, we have
$e_1 > a_1$.  Therefore, we have $c_{e_1} > b_{a_1}$ and thus $c_{e_1} = b_{e_1}$; it is also
clear that $b_{e_2} \le c_{e_2}$.  It follows that $e \in U(T)$, contradiction.
Hence, we may assume that $e_3 = a_3$ and $F$ is of $F$-type 1.  Because $b_{e_i} < c_{e_3} = b_{a_2}$
for $i \in [2]$, we have $b_{e_i} = c_{e_i}$ for $i \in [2]$, which implies that $e \in U(T)$, contradiction.
The separability of $\psi(T)$ follows.

For $i < a_1$, we have $b_i = c_i$, and $c_{a_1} = b_{a_3} > b_{a_1}$.  Thus, we have
$(b_1,b_2,\ldots,b_n) < (c_1,c_2,\ldots,c_n)$ in the lexicographic order, as desired.

\item
\begin{figure}
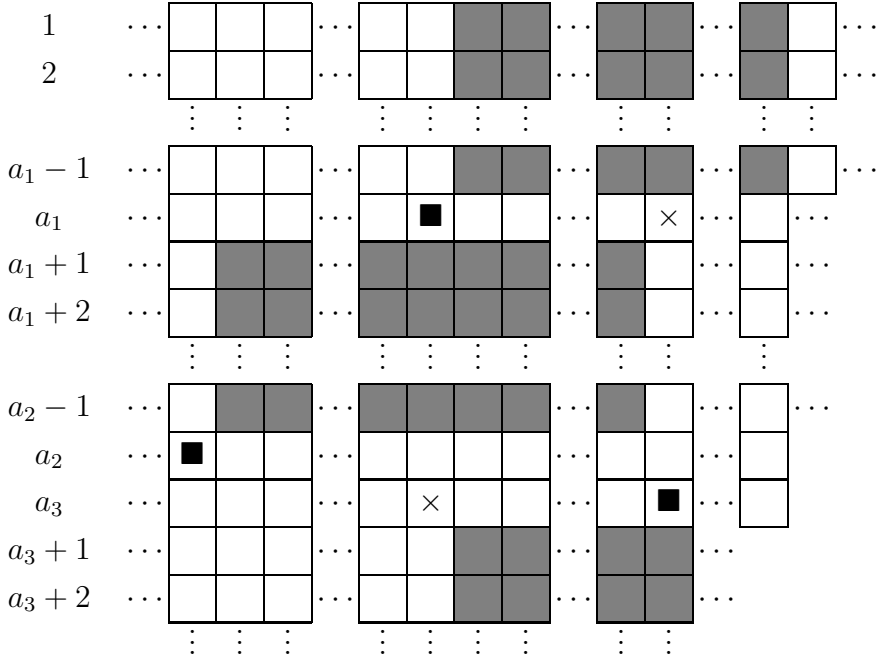

\begin{ytableau}
\none[1] &\none & \none[\cdots] &&&&\none[\cdots] && & *(gray) & *(gray) &\none[\cdots] & *(gray) & *(gray)  & \none[\cdots] & *(gray) & & \none[\cdots]\\
\none[2] &\none & \none[\cdots] &&&&\none[\cdots] && & *(gray) & *(gray) &\none[\cdots] & *(gray) & *(gray)  & \none[\cdots] & *(gray) & & \none[\cdots]\\
\none &\none & \none &\none[\vdots] &\none[\vdots] &\none[\vdots] &\none &\none[\vdots] & \none[\vdots] & \none[\vdots] & \none[\vdots] & \none & \none[\vdots] & \none[\vdots] & \none & \none[\vdots] & \none[\vdots] & \none\\
\none[a_1-1] &\none & \none[\cdots] &&&&\none[\cdots] && & *(gray) & *(gray) &\none[\cdots] & *(gray) & *(gray)  & \none[\cdots] & *(gray) & & \none[\cdots]\\
\none[a_1] &\none & \none[\cdots] &&&&\none[\cdots]&& \blacksquare & & & \none[\cdots] & & \times & \none[\cdots] && \none[\cdots]\\
\none[a_1+1] &\none & \none[\cdots] && *(gray) & *(gray) &\none[\cdots] & *(gray) & *(gray) & *(gray) & *(gray) & \none[\cdots] & *(gray) & & \none[\cdots] && \none[\cdots]\\
\none[a_1+2] &\none & \none[\cdots] && *(gray) & *(gray) &\none[\cdots] & *(gray) & *(gray) & *(gray) & *(gray) & \none[\cdots] & *(gray) & & \none[\cdots] && \none[\cdots]\\
\none &\none & \none &\none[\vdots] &\none[\vdots] &\none[\vdots] &\none &\none[\vdots] & \none[\vdots] & \none[\vdots] & \none[\vdots] & \none & \none[\vdots] & \none[\vdots] & \none & \none[\vdots]\\
\none[a_2-1] &\none & \none[\cdots] && *(gray) & *(gray) &\none[\cdots] & *(gray) & *(gray) & *(gray) & *(gray) & \none[\cdots] & *(gray) & & \none[\cdots] && \none[\cdots]\\
\none[a_2] &\none & \none[\cdots] &\blacksquare &&& \none[\cdots] &&&&&\none[\cdots] &&& \none[\cdots] &\\
\none[a_3] &\none & \none[\cdots] &&&&\none[\cdots]&& \times &&&\none[\cdots] && \blacksquare & \none[\cdots] &\\
\none[a_3+1] &\none & \none[\cdots] &&&&\none[\cdots]&& & *(gray) &*(gray) &\none[\cdots] &*(gray) &*(gray) & \none[\cdots]\\
\none[a_3+2] &\none & \none[\cdots] &&&&\none[\cdots]&& & *(gray) &*(gray) &\none[\cdots] &*(gray) &*(gray) & \none[\cdots]\\
\none &\none & \none &\none[\vdots] &\none[\vdots] &\none[\vdots] &\none &\none[\vdots] & \none[\vdots] & \none[\vdots] & \none[\vdots] & \none & \none[\vdots] & \none[\vdots] \\
\end{ytableau}
\caption{The squares marked with a solid black box are the elements of the chosen copy of $F_3$ for a separable, valid transversal $T$ of $F$-type 2, and the crosses mark new elements of $\psi(T)$, i.e. elements of $\psi(T) \setminus T$.  The gray squares are free of elements of $T$ (and $\psi(T)$).}
\label{fig:FType2}
\end{figure}
See Figure~\ref{fig:FType2}.
First, we prove that $\psi(T)$ is a valid transversal of $\mathcal{Y}$.
It is clear that $\psi(T)$ is a transversal of $Y$.  If the sets
$\{i,i+1\}$ and $\{a_1,a_3\}$ are disjoint, then $b_i = c_i$ and $b_{i+1} = c_{i+1}$,
which implies that $i$ is an ascent (resp. descent) of $\psi(T)$ if and only if
it is an ascent (resp. descent) of $T$.  By Lemma~\ref{EPsi}, we have $b_{a_1-1}, b_{a_1+1}
\notin (b_{a_2},b_{a_3}) \supseteq (b_{a_1},c_{a_1})$.  It follows that $a_1-1$
is an ascent (resp. descent) of $\phi(T)$ if and only if it is an ascent (resp. descent)
of $T$, and the same for $a_1$.  We also have $a_3 - 1 \in A$, but we also have
$c_{a_3-1} = b_{a_2} < b_{a_1} = c_{a_3}$ and thus $a_3-1$ is an ascent of $\psi(T)$.
Because $\mathcal{Y}$ is 1-alternating,
we have $a_3 \in D$.  However, by Lemma~\ref{EPsi}, we have $c_{a_3+1} = b_{a_3+1} < b_{a_1} = c_{a_3}$,
and thus $a_3$ is a descent of $\psi(T)$.  It follows that $\psi(T)$ is a valid transversal of $\mathcal{Y}$,
as desired.

Next, we prove that there is an integer $y$ such that $h_J(T) = (a_1,y,a_3+1)$.
First, because $Y_{a_3+1} = Y_{a_3} \ge b_{a_3} = c_{a_1}$ and $c_{a_3} = b_{a_1} > c_{a_3+1}$
(which follows from Lemma~\ref{EPsi}),
we have $(a_1,a_3,a_3+1) \in U(\psi(T))$.
Suppose for sake of contradiction that $d = (d_1,d_2,d_3) \in U(\psi(T))$ with $\#(d) < (a_3+1,a_1,0)$
in the lexicographic order. If $d_3 < a_1$ or ($a_1 < d_3 < a_3-1$ and $a_1 \notin\{d_1,d_2\}$), then we
have $b_{d_i} = c_{d_i}$ for all $i \in [3]$, and thus $d \in U(T)$, which contradicts the separability
of $T$.  If $d_3 = a_1$, then because $b_{a_1} < b_{a_3} = c_{a_1}$, we have $d \in U(T)$,
contradiction.  If $d_i = a_1$ for some $i \in [2]$ and $a_3 \notin \{d_1,d_2,d_3\}$, then by Lemma~\ref{EPsi}, we have $b_{d_{i+1}} \notin (b_{a_3},b_{a_1})$, which implies that $(d_1,d_2,d_3) \in U(T)$, contradiction.
If $d_3 = a_3$ and $a_1 \notin \{d_1,d_2\}$, it is clear that
$d_2 \not= a_3 - 1$.  Thus, we have $(d_1,d_2,a_3-1) \in U(T)$ because
$b_{d_i} = c_{d_i}$ for $i \in [2]$ and $c_{a_3} > c_{a_3-1} = b_{a_3-1}$, contradiction.
If $d_3 = a_3$, it is impossible that $d_1 = a_1$ because if $d_1 = a_1$, then
$a_1 < d_2 < a_3$ with $b_{d_2} \in (b_{a_2},b_{a_3})$, which contradicts Lemma~\ref{EPsi}.
If $d_3 = a_3$ and $d_2 = a_1$, then we have \[b_{d_1} = c_{d_1} > c_{a_1} = b_{a_3} > b_{a_1} = c_{a_3} > c_{a_3-1} = b_{a_3-1},\] where the last inequality follows from Lemma~\ref{EPsi}.  Hence, we have $(d_1,d_2,a_3-1) \in U(T)$, contradiction.  Hence, we may assume
that $d_3 = a_3 + 1$.  Because $\#(d) < (a_3+1,a_1,0)$, we may also assume that $d_1 < a_1$,
which implies that $b_{d_1} = c_{d_1}$.  Lemma~\ref{EPsi} yields that $b_{d_1} < b_{a_1}$,
and thus $b_{d_i} = c_{d_i}$ for all $i \in [3]$.  It follows that $d \in U(T)$, contradiction.
It is clear that $\psi(T)$ is of $J$-type 2 and that $\phi(\psi(T)) = T$.

We prove that if $e = (e_1,e_2,e_3) \in V(\psi(T))$, then $S(e) \ge S(h_F(T))$ in the lexicographic
order.  If $e_3 = a_3$, then $b_{e_i} = c_{e_i}$ for all $i \in [2]$, and $b_{e_3} > c_{e_3}$.
It follows that $e \in V(T)$, which contradicts the definition of $h_F$.  Hence, we may
assume that $e_3 > a_3$, and it follows that $b_{e_3} = c_{e_3}$.  Lemma~\ref{EPsi} yields
that $b_{e_3} < b_{a_1}$, and thus $c_{e_i} < b_{a_1}$ for all $i$, which yields that
$b_{e_i} = c_{e_i}$ for all $i$. This implies that $e \in V(T)$, contradiction.  The separability
of $\psi(T)$ follows.

For $i < a_1$, we have $b_i = c_i$.  Because $b_{a_1} < b_{a_3} = c_{a_1}$, we have
$(b_1,b_2,\ldots,b_n) < (c_1,c_2,\ldots,c_n)$ in the lexicographic order, as desired.

\item
\begin{figure}
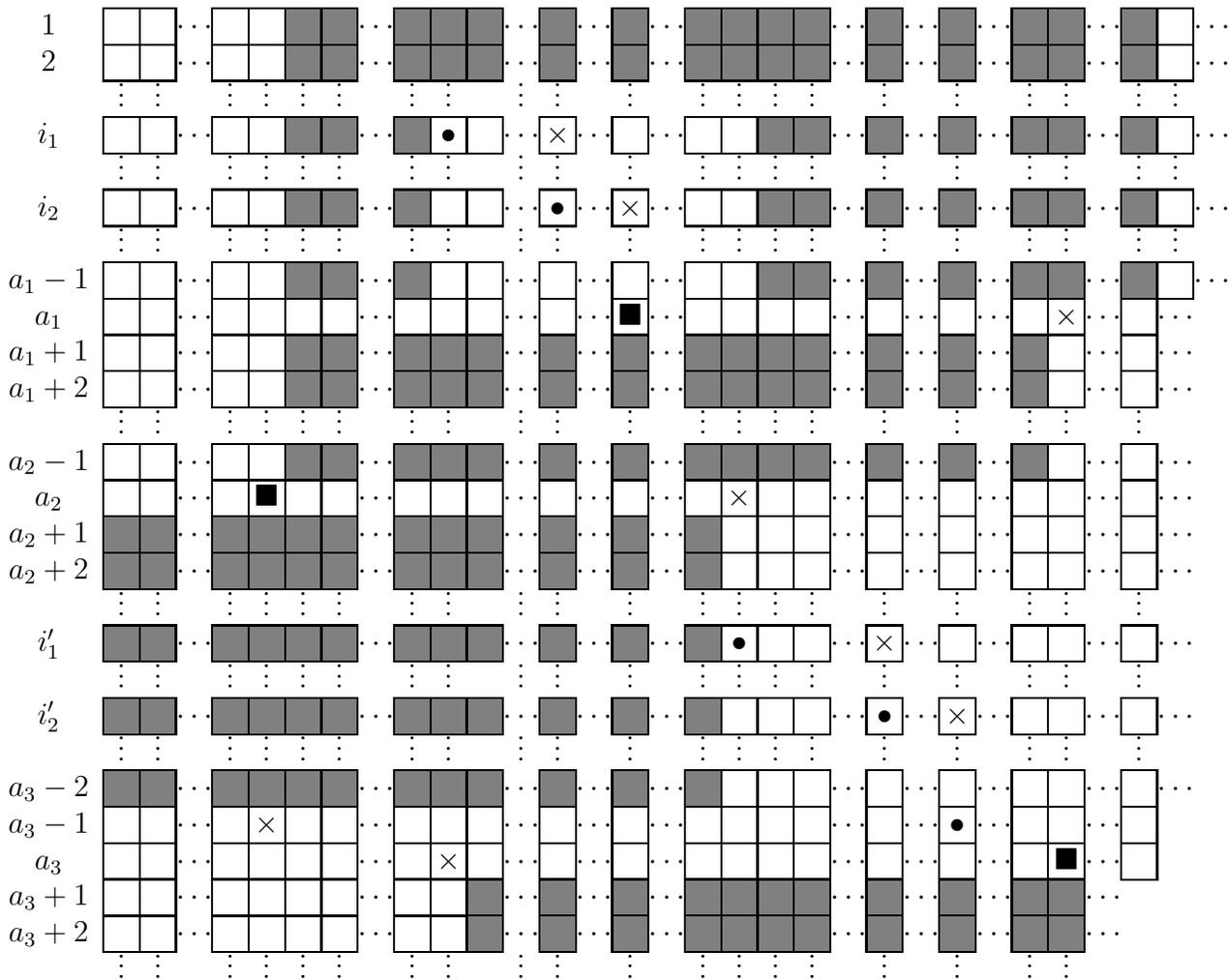

\ytableausetup{boxsize = 1.175em}
\begin{ytableau}
\none[1] &\none & &&\none[\cdots]& &&*(gray)&*(gray)&\none[\cdots] &*(gray)&*(gray)&*(gray)&\none[\cdots]&*(gray)&\none[\cdots]&*(gray)&\none[\cdots] &*(gray)&*(gray) & *(gray) & *(gray) &\none[\cdots] &*(gray) &\none[\cdots] &*(gray) &\none[\cdots] & *(gray) & *(gray)  & \none[\cdots] & *(gray) & & \none[\cdots]\\
\none[2] &\none & &&\none[\cdots]& &&*(gray)&*(gray)&\none[\cdots] &*(gray)&*(gray)&*(gray)&\none[\cdots]&*(gray)&\none[\cdots]&*(gray)&\none[\cdots] &*(gray)&*(gray) & *(gray) & *(gray) &\none[\cdots] &*(gray) &\none[\cdots] &*(gray) &\none[\cdots] & *(gray) & *(gray)  & \none[\cdots] & *(gray) & & \none[\cdots]\\
\none &\none & \none[\vdots] &\none[\vdots] &\none &\none[\vdots] &\none[\vdots] &\none[\vdots] & \none[\vdots] & \none & \none[\vdots]&\none[\vdots]&\none&\none[\vdots]&\none[\vdots]&\none&\none[\vdots]&\none& \none[\vdots] & \none[\vdots] & \none[\vdots] & \none[\vdots] & \none & \none[\vdots] & \none & \none[\vdots] & \none & \none[\vdots] & \none[\vdots] &\none &\none[\vdots] &\none[\vdots]\\
\none[i_1] &\none & &&\none[\cdots]& &&*(gray)&*(gray)&\none[\cdots] &*(gray)&\bullet& &\none[\cdots]&\times&\none[\cdots]&&\none[\cdots] && & *(gray) & *(gray) &\none[\cdots] &*(gray) &\none[\cdots] &*(gray) &\none[\cdots] & *(gray) & *(gray)  & \none[\cdots] & *(gray) & & \none[\cdots]\\
\none &\none & \none[\vdots] &\none[\vdots] &\none &\none[\vdots] &\none[\vdots] &\none[\vdots] & \none[\vdots] & \none & \none[\vdots]&\none[\vdots]&\none&\none[\vdots]&\none[\vdots]&\none&\none[\vdots]&\none& \none[\vdots] & \none[\vdots] & \none[\vdots] & \none[\vdots] & \none & \none[\vdots] & \none & \none[\vdots] & \none & \none[\vdots] & \none[\vdots] &\none &\none[\vdots] &\none[\vdots]\\
\none[i_2] &\none & &&\none[\cdots]& &&*(gray)&*(gray)&\none[\cdots] &*(gray)&&& \none[\cdots]&\bullet&\none[\cdots]&\times&\none[\cdots] && & *(gray) & *(gray) &\none[\cdots] &*(gray) &\none[\cdots] &*(gray) &\none[\cdots] & *(gray) & *(gray)  & \none[\cdots] & *(gray) & & \none[\cdots]\\
\none &\none & \none[\vdots] &\none[\vdots] &\none &\none[\vdots] &\none[\vdots] &\none[\vdots] & \none[\vdots] & \none & \none[\vdots]&\none[\vdots]&\none&\none[\vdots]&\none[\vdots]&\none&\none[\vdots]&\none& \none[\vdots] & \none[\vdots] & \none[\vdots] & \none[\vdots] & \none & \none[\vdots] & \none & \none[\vdots] & \none & \none[\vdots] & \none[\vdots] &\none &\none[\vdots] &\none[\vdots]\\
\none[a_1-1] &\none & &&\none[\cdots]& &&*(gray)&*(gray)&\none[\cdots] &*(gray)&&&\none[\cdots]&&\none[\cdots] &&\none[\cdots] && & *(gray) & *(gray) &\none[\cdots] &*(gray) &\none[\cdots] &*(gray) &\none[\cdots] & *(gray) & *(gray)  & \none[\cdots] & *(gray) & & \none[\cdots]\\
\none[a_1] &\none &&& \none[\cdots] &&&&&\none[\cdots] &&&&\none[\cdots]&&\none[\cdots]&\blacksquare&\none[\cdots]&&&&& \none[\cdots] &&\none[\cdots] &&\none[\cdots] &&\times & \none[\cdots] && \none[\cdots]\\
\none[a_1+1] &\none &&& \none[\cdots] &&& *(gray) & *(gray) &\none[\cdots] &*(gray)&*(gray)&*(gray)&\none[\cdots]&*(gray)&\none[\cdots]&*(gray)&\none[\cdots] & *(gray) & *(gray) & *(gray) & *(gray) &\none[\cdots] &*(gray) &\none[\cdots] &*(gray) &\none[\cdots] & *(gray) & & \none[\cdots] && \none[\cdots]\\
\none[a_1+2] &\none &&& \none[\cdots] &&& *(gray) & *(gray) &\none[\cdots] &*(gray)&*(gray)&*(gray)&\none[\cdots]&*(gray)&\none[\cdots]&*(gray)&\none[\cdots] & *(gray) & *(gray) & *(gray) & *(gray) &\none[\cdots] &*(gray) &\none[\cdots] &*(gray) &\none[\cdots] & *(gray) & & \none[\cdots] && \none[\cdots]\\
\none &\none & \none[\vdots] &\none[\vdots] &\none &\none[\vdots] &\none[\vdots] &\none[\vdots] & \none[\vdots] & \none & \none[\vdots]&\none[\vdots]&\none&\none[\vdots]&\none[\vdots]&\none&\none[\vdots]&\none& \none[\vdots] & \none[\vdots] & \none[\vdots] & \none[\vdots] & \none & \none[\vdots] & \none & \none[\vdots] & \none & \none[\vdots] & \none[\vdots] &\none &\none[\vdots]\\
\none[a_2-1] &\none &&& \none[\cdots] &&& *(gray) & *(gray) &\none[\cdots] &*(gray)&*(gray)&*(gray)&\none[\cdots]&*(gray)&\none[\cdots]&*(gray)&\none[\cdots] & *(gray) & *(gray) & *(gray) & *(gray) &\none[\cdots] &*(gray) &\none[\cdots] &*(gray) &\none[\cdots] & *(gray) & & \none[\cdots] && \none[\cdots]\\
\none[a_2] &\none &&& \none[\cdots] &&\blacksquare&&& \none[\cdots] &&&&\none[\cdots]&&\none[\cdots]&&\none[\cdots] && \times&&&\none[\cdots] &&\none[\cdots] &&\none[\cdots] &&& \none[\cdots] && \none[\cdots]\\
\none[a_2+1] &\none &*(gray) &*(gray) & \none[\cdots] &*(gray) &*(gray) &*(gray) &*(gray) & \none[\cdots] &*(gray)&*(gray)&*(gray)&\none[\cdots]&*(gray)&\none[\cdots]&*(gray)&\none[\cdots] &*(gray) &&&&\none[\cdots] &&\none[\cdots] &&\none[\cdots] &&& \none[\cdots] && \none[\cdots]\\
\none[a_2+2] &\none &*(gray) &*(gray) & \none[\cdots] &*(gray) &*(gray) &*(gray) &*(gray) & \none[\cdots] &*(gray)&*(gray)&*(gray)&\none[\cdots]&*(gray)&\none[\cdots]&*(gray)&\none[\cdots] &*(gray) &&&&\none[\cdots] &&\none[\cdots] &&\none[\cdots] &&& \none[\cdots] && \none[\cdots]\\
\none &\none & \none[\vdots] &\none[\vdots] &\none &\none[\vdots] &\none[\vdots] &\none[\vdots] & \none[\vdots] & \none & \none[\vdots]&\none[\vdots]&\none&\none[\vdots]&\none[\vdots]&\none&\none[\vdots]&\none& \none[\vdots] & \none[\vdots] & \none[\vdots] & \none[\vdots] & \none & \none[\vdots] & \none & \none[\vdots] & \none & \none[\vdots] & \none[\vdots] &\none &\none[\vdots]\\
\none[i'_1] &\none &*(gray) &*(gray) & \none[\cdots] &*(gray) &*(gray) &*(gray) &*(gray) & \none[\cdots] &*(gray)&*(gray)&*(gray)&\none[\cdots]&*(gray)&\none[\cdots]&*(gray)&\none[\cdots] &*(gray) &\bullet&&&\none[\cdots] &\times&\none[\cdots] &&\none[\cdots] &&& \none[\cdots] && \none[\cdots]\\
\none &\none & \none[\vdots] &\none[\vdots] &\none &\none[\vdots] &\none[\vdots] &\none[\vdots] & \none[\vdots] & \none & \none[\vdots]&\none[\vdots]&\none&\none[\vdots]&\none[\vdots]&\none&\none[\vdots]&\none& \none[\vdots] & \none[\vdots] & \none[\vdots] & \none[\vdots] & \none & \none[\vdots] & \none & \none[\vdots] & \none & \none[\vdots] & \none[\vdots] &\none &\none[\vdots]\\
\none[i'_2] &\none &*(gray) &*(gray) & \none[\cdots] &*(gray) &*(gray) &*(gray) &*(gray) & \none[\cdots] &*(gray)&*(gray)&*(gray)&\none[\cdots]&*(gray)&\none[\cdots]&*(gray)&\none[\cdots] &*(gray) &&&&\none[\cdots] &\bullet&\none[\cdots] &\times&\none[\cdots] &&& \none[\cdots] && \none[\cdots]\\
\none &\none & \none[\vdots] &\none[\vdots] &\none &\none[\vdots] &\none[\vdots] &\none[\vdots] & \none[\vdots] & \none & \none[\vdots]&\none[\vdots]&\none&\none[\vdots]&\none[\vdots]&\none&\none[\vdots]&\none& \none[\vdots] & \none[\vdots] & \none[\vdots] & \none[\vdots] & \none & \none[\vdots] & \none & \none[\vdots] & \none & \none[\vdots] & \none[\vdots] &\none &\none[\vdots]\\
\none[a_3-2] &\none &*(gray) &*(gray) & \none[\cdots] &*(gray) &*(gray) &*(gray) &*(gray) & \none[\cdots] &*(gray)&*(gray)&*(gray)&\none[\cdots]&*(gray)&\none[\cdots]&*(gray)&\none[\cdots] &*(gray) &&&&\none[\cdots] &&\none[\cdots] &&\none[\cdots] &&& \none[\cdots] && \none[\cdots]\\
\none[a_3-1] &\none &&& \none[\cdots] &&\times &&&\none[\cdots]&&&&\none[\cdots]&&\none[\cdots]&&\none[\cdots] &&&&&\none[\cdots]  &&\none[\cdots] &\bullet&\none[\cdots] &&& \none[\cdots] &\\
\none[a_3] &\none &&& \none[\cdots] &&&&&\none[\cdots]&&\times&&\none[\cdots]&&\none[\cdots]&&\none[\cdots] &&&&&\none[\cdots]  &&\none[\cdots] &&\none[\cdots] &&\blacksquare& \none[\cdots] &\\
\none[a_3+1] &\none &&& \none[\cdots] &&&&&\none[\cdots]&&&*(gray)&\none[\cdots]&*(gray)&\none[\cdots] &*(gray)&\none[\cdots]&*(gray)&*(gray)& *(gray) &*(gray) &\none[\cdots] &*(gray) &\none[\cdots] &*(gray) &\none[\cdots] &*(gray) &*(gray) & \none[\cdots] \\
\none[a_3+2] &\none &&& \none[\cdots] &&&&&\none[\cdots]&&&*(gray)&\none[\cdots]&*(gray)&\none[\cdots] &*(gray)&\none[\cdots]&*(gray)&*(gray)& *(gray) &*(gray) &\none[\cdots] &*(gray) &\none[\cdots] &*(gray) &\none[\cdots] &*(gray) &*(gray) & \none[\cdots] \\
\none &\none & \none[\vdots] &\none[\vdots] &\none &\none[\vdots] &\none[\vdots] &\none[\vdots] & \none[\vdots] & \none & \none[\vdots]&\none[\vdots]&\none&\none[\vdots]&\none[\vdots]&\none&\none[\vdots]&\none& \none[\vdots] & \none[\vdots] & \none[\vdots] & \none[\vdots] & \none & \none[\vdots] & \none & \none[\vdots] & \none & \none[\vdots] & \none[\vdots]\\
\end{ytableau}
\caption{The squares marked with a solid black box are the elements of the chosen copy of $F_3$ for a separable, valid transversal $T = \{(i,b_i)\}$ of $F$-type 3.  The bullets mark some other elements of $T$, while the crosses mark new elements of $\psi(T)$, i.e. elements of $\phi(T) \setminus T$.  The gray squares are free of elements of $T$ (and $\psi(T)$). We suppose that $\Gamma_{[1,a_1)}^{[b_{a_2},b_{a_3}]}(T) = \{i_1,i_2\}$ and $\Gamma_{(a_2,a_3-1)}^{[b_{a_2},b_{a_3}]}(T) = \{i'_1,i'_2\}$.}
\label{fig:FType3}
\end{figure}
See Figure~\ref{fig:FType3}.
First, we prove that $\psi(T)$ is a valid transversal of $\mathcal{Y}$.
This paragraph is similar to the first paragraph of the proof of Proposition~\ref{FullPhi}
for the case of $J$-type 3.
If $b_i > b_{a_3}$ or $i > a_3$, it is clear that $b_i = c_i$, and therefore $T$
is a transversal of $Y$.
If the sets $\{i,i+1\}$ and $\Gamma_{[1,a_1] \cup [a_2,a_3]}^{[b_{a_2},b_{a_3}]}(T)$ are disjoint,
then $b_i = c_i$ and $b_{i+1} = c_{i+1},$ and thus
$i$ is an ascent (resp. descent) of $\phi(T)$ if and only if it is an ascent (resp. descent) of $T$.
If $x, x+1 \in \Gamma_{[1,a_1] \cup [a_2,a_3-1]}^{[b_{a_2},b_{a_3}]}(T)$ with $x \not= a_1,a_3-2$.
then by Lemma~\ref{FType3L1}, we have $b_x < b_{x+1}$ and $c_x < c_{x+1}$.
If $x \in \Gamma_{[1,a_1] \cup [a_2,a_3-1]}^{[b_{a_2},b_{a_3}]}(T)$,
but $x+1$ is not (and $x \not= a_1,a_3-1$), then we have $b_{a_2} < b_x,c_x < b_{a_1}$ and $b_{x+1} \notin [b_{a_3},b_{a_3+1}],$ which implies that $x$ is an ascent (resp. descent) of $\phi(T)$ if and only if
it is (resp. descent) of $T$.  Similar logic holds if we replace $x+1$ by $x-1$ and require that
$x \not= a_2$.
Because $b_{a_1+1}, b_{a_2-1} \notin [b_{a_2},b_{a_3}]$ by Lemma~\ref{EPsi},
we have $b_{a_1+1} < b_{a_1}$ if and only if $b_{a_1+1} < c_{a_1}$, and $b_{a_2-1} < b_{a_2}$
if and only if $b_{a_2-1} < c_{a_2}$.
If $a_2 \not= a_1+1$, then
we have $c_{a_1+1} = b_{a_1+1}$ and $c_{a_2-1} = b_{a_2-1}$, which implies that
that $a_1$ is an ascent (resp. descent) of $\psi(T)$ if and only if it is an ascent (resp. descent)
of $T$, and similarly for $a_2-1$.
If $a_2 = a_1 + 1$, then $a_1 = a_2-1$ is a descent of both $T$ and $\psi(T)$.
Also, we have $a_3-1 \in A$, and
because $\mathcal{Y}$ is 1-alternating, we have $a_3 \in D$ and $a_3-2 \notin A$.
If $a_3-2 \in D$, we have $b_{a_3-2} > b_{a_3-1}$, and Lemma~\ref{FType3L1} implies that $a_3-2 \notin \Gamma_{[a_2,a_3-1]}^{[b_{a_2},b_{a_3}]}$.  Therefore, we have $c_{a_3-2} = b_{a_3-2} > b_{a_3-1} > c_{a_3-1}$.
Regardless, we have $c_{a_3-1} = b_{a_2}< c_{a_3+1}$.  By Lemma~\ref{FType3L2} and the definition of $\psi$,
we have $c_{a_3+2} = b_{a_3+2} < c_{a_3+1}$.
It follows that $\phi(T)$ is a valid transversal of $\mathcal{Y}$.

Next, we prove that $h_J(\psi(T)) = (a_1,a_2,a_3-1)$.  It is clear that
$(a_1,a_2,a_3-1) \in U(\psi(T))$.  Suppose for sake of contradiction that $(d_1,d_2,d_3) \in U(\psi(T))$
with $\#(d) < (a_3-1,a_1,a_2)$ in the lexicographic order.
If $b_{d_3} > b_{a_3}$, then we have $b_{d_i} = c_{d_i}$ for all $i \in [3]$,
and thus $d \in U(T)$, contradiction.
If $d_3 < a_3 - 1$ and $b_{a_2} < b_{d_3} < b_{a_3}$, or $d_3 = a_2$, then by Lemma~\ref{EPsi},
we have $d_3 \ge a_2$, and Lemma~\ref{FType3L1} yields that that $c_{d_3} > b_{a_1}$.
Then, by Lemma~\ref{FType3L1}, we have $d_2 < a_2$ or  $b_{d_2} > b_{a_1}$.
In the former case, because $c_{d_2} > c_{d_3} > b_{a_1}$, we have $d_2 = a_1$
or $b_{d_2} > b_{a_1}$, and if $d_2 = a_1$, then $(d_1,a_1,a_2) \in U(T)$, contradiction.
If $d_3 < a_3 - 1$ and $b_{d_2} > b_{a_1}$, then we have $(d_1,d_2,d_3) \in U(T)$, contradiction.
If $d_3 < a_3 - 1$ and $b_{a_2} < b_{d_3} < b_{a_1}$, then by Lemma~\ref{EPsi} we have
$d_3 \le a_1$.  Lemma~\ref{FType3L1} implies that $b_{d_2} > b_{a_3}$
it follows that $c_{d_2} = b_{d_2} > b_{a_3} > b_{d_3}$.  Therefore, we have $(d_1,d_2,d_3) \in U(T)$,
contradiction.  If $d_3 < a_3 - 1$ and $b_{d_3} < b_{a_2}$, then by Lemma~\ref{EPsi} we have $d_3 < a_2$.
By Lemma~\ref{FType3L1}, at most one of $d_1,d_2$ can be in $\Gamma_{[1,a_1]}^{[b_{a_2},b_{a_3}]}(T)$,
while by Lemma~\ref{EPsi}, any $d_i \notin \Gamma_{[1,a_1]}^{[b_{a_2},b_{a_3}]}(T)$ must satisfy
$b_{d_i} \notin [b_{a_2},b_{a_3}]$.  It follows that $(d_1,d_2,d_3) \in U(T)$, contradiction.  Hence,
we may assume that $d_3 = a_3 - 1$.  If $d_1 < a_1$, then it follows from Lemma~\ref{FType3L1}
that $c_{d_1} \notin (b_{a_1},b_{a_3})$, and by Lemma~\ref{EPsi}, this implies that
$c_{d_1} \notin (b_{a_1},Y_{a_3})$.  Therefore, we have $c_{d_1} \le b_{a_1}$, which implies
that $c_{d_2} \in (b_{a_2},b_{a_1})$.  By Lemma~\ref{FType3L1}, we have $b_{d_2} \in (b_{a_2},b_{a_1})$,
and Lemma~\ref{EPsi} yields that $d_2 < a_1$.  Applying Lemma~\ref{FType3L1} again yields that
$b_{d_2} > b_{d_1}$, contradiction.  Hence, we may assume that $d_1 = a_1$.  The fact that
$d_2 \ge a_2$ follows from Lemma~\ref{EPsi}.  It is clear that $\psi(T)$ is of $J$-type 3.
Because $c_{a_1} = b_{a_3+1}$ and $c_{a_3} = b_{a_2}$, we have
\[\phi(\psi(T)) = \omega_{[1,a_1] \cup \{a_3\}}^{[b_{a_2},b_{a_3}]}\left(\omega_{[a_2,a_3-1]}^{[b_{a_2},b_{a_3}]}\left(\theta_{[a_2,a_3-1]}^{[b_{a_2},b_{a_3}]}\left(\theta_{[1,a_1] \cup \{a_3\}}^{[b_{a_2},b_{a_3}]}(T)\right)\right)\right) = T,\]
as desired.

We prove that if $e = (e_1,e_2,e_3) \in V(\psi(T))$, then $S(e) \ge S(h_F(T))$ in the lexicographic
order.  First, we prove that $e_3 \le a_3$.  Suppose for sake of contradiction
that $e_3 > a_3$.  Let $m = \min \Gamma_{[1,a_1]}^{[b_{a_2},b_{a_3}]}(T)$;
by Lemma~\ref{FType3L2}, we have $c_{e_3} = b_{e_3} < b_m = c_{a_3+1}$.
If $e_1 > a_3$, then $e \in V(T)$, contradiction, and thus we may assume
that $e_1 \le a_3$, and Lemmata~\ref{EPsi} and~\ref{FType3L1} imply
that $e_1 = a_3$ or $e_1 < a_2$.  In the former case, $(a_2,e_2,e_3) \in V(T)$, contradiction.
In the latter case, by Lemma~\ref{EPsi} and because $b_m < c_x$ for all $x \in \Gamma_{[1,a_1]}^{[b_{a_2},b_{a_1}]}$,
we have $b_{e_1} < b_{a_2}$.  This implies that $b_{e_i} = c_{e_i}$ for all $i$, and thus $(e_1,e_2,e_3) \in U(T)$,
contradiction.  Thus, we have $e_3 \le a_3$, as desired.  Because $e_3 \not= a_3$ and
the first component of $S(e)$ is at least $a_3$, either ($e$ is of $F$-type 2, $a_3 - 2 \in D$, $e_2 = a_3 - 3$ and $e_3 = a_3 - 2$)
or $e_3 = a_3$.  In the former case, because $b_{a_3-2} > b_{a_3-1}$,
Lemma~\ref{FType3L1} implies that $b_{a_3-2} > b_{a_3}$ and therefore $b_{a_3-2} = c_{a_3-2}$. By Lemma~\ref{EPsi}, we have $b_{a_3-3} \ge b_{a_1}$,
and it follows that $b_{a_3-3} \le \max\{b_{a_3},c_{a_3-3} < c_{a_3-2} = b_{a_3-2}$.  Lemma~\ref{EPsi}
implies that $b_{e_1} > b_{a_3}$ or $e_1 > d_2$, but in the latter case, the fact that $c_{e_1} > c_{e_2}$
implies that $b_{e_1} > b_{a_3}$ as well.  Thus, $b_{e_1} = c_{e_1}$ and $(e_1,e_2,e_3) \in V(T)$,
which contradicts the definition of $h_F$.  Hence, we may assume that $e_3 = a_3$.  If $e_2 = a_3 - 1$,
then we have $b_m = c_{a_3} > c_{e_1} > c_{e_2} = b_{a_2}$, and by Lemma~\ref{EPsi} and the definition
of $\psi$, we have $e_1 < a_1$.  Thus, $e_1 \in \Gamma_{[1,a_1]}^{[b_{a_2},b_{a_3}]}$ with
$c_{e_1} < b_m < c_m$, which contradicts Lemma~\ref{FType3L1}.  Therefore, we may assume
that $e_2 < a_3 - 1$ and $e$ is of $F$-type 3.  The fact that $S(e) > S(h_F(T))$ in the lexicographic
order implies that $e_1 \ge a_1$, but $c_{a_1} > c_{a_3}$ and thus we may in fact assume that
$e_1 > a_1$.  By Lemma~\ref{EPsi} and the definition of $\psi$, we have $c_{e_1} < b_{a_2}$,
which implies that $b_{e_i} = c_{e_i}$ for $i \in [2]$, and thus $(e_1,e_2,e_3) \in V(T)$, contradiction.
The separability of $\psi(T)$ follows.

For all $i < m$, we have $b_i = c_i$.
If $m < a_1$, then let $m' = \min \Gamma_{(m,a_1]}^{[b_{a_3},b_{a_1}]}(T)$.  By Lemma~\ref{FType3L1},
we have $c_m = b_{m'} > b_m$.  If $m = a_1$, then we have $c_m = b_{a_3} > b_m$.  It follows
that $(b_1, b_2, \ldots, b_n) > (c_1,c_2,\ldots, c_n)$ in the lexicographic order, as desired.
\end{asparaenum}
\end{proof}

\section{Brute-Force Enumerations}
\label{sec:enumerate}

We computed $|A_n(q)|$ for small $n$ and short-length $q$ by brute-force computer enumeration.
This data, shown in Tables~\ref{table:6Even} and~\ref{table:6Odd}, formed the basis of our results and conjectures.  In an attempt to check
if Conjecture~\ref{SesaGeneral} is true, we computed $|S_{\mathcal{Y}}(F_k)$ and $|S_\mathcal{Y}(J_k)|$
for $k \le 10$ and all $\mathcal{Y}$ whose Young diagram has at most 10 rows.  The equality
$|S_{\mathcal{Y}}(F_k) = |S_\mathcal{Y}(J_k)|$ held for all such $\mathcal{Y}$ and $k$;
since the output data is very large, we do not give it in this paper.

We also computed $|D_n^k(q)|$ for small $n, k$ and short-length $q$ by similar enumerations.  This
data, shown in Table~\ref{table:4Rep}, forms the basis for our theorems of Section~\ref{Generalized
alternating permutations}.

\begin{table}
\center{
\begin{tabular}{|l|c|c|c|c|c|c|}
\hline
Patterns & 2 & 4 & 6 & 8 & 10 & 12\\
\hline
(634521, 652341), (534621, 651342) & 1 & 5 & 61 & 1385 & 47860 & 2202236\\
\hline
(564321, 654312), 645321, 653421, & 1 & 5 & 61 & 1385 & 47860 & 2201540\\
(456321, 654123), (345621, 651234), &&&&&&\\
(234561, 612345), (165432, 543216), &&&&&&\\
(216543, 432165), (126543, 432156), &&&&&&\\
321654, (213654, 321465), 123456 &&&&&&\\
(123654, 321456), (213465, 213465), &&&&&&\\
(123465, 213456) &&&&&&\\
\hline
(312654, 321564), (213564, 312465) & 1 & 5 & 61 & 1385 & 47860 & 2198859\\
(123564, 312456) &&&&&&\\
\hline
(215643, 431265), (125643, 431256) & 1 & 5 & 61 & 1385 & 47860 & 2197690\\
\hline
(214563, 412365), (124563, 412356) & 1 & 5 & 61 & 1385 & 47860 & 2197299\\
\hline
(214653, 421365), (124653, 421356) & 1 & 5 & 61 & 1385 & 47860 & 2195798\\
\hline
(143265, 215436), (125436, 143256) & 1 & 5 & 61 & 1344 & 44386 & 1954114\\
\hline
(132654, 321546), (124365, 214356), & 1 & 5 & 61 & 1344 & 44377 & 1951843\\
 (132465, 213546), (123546, 132456) &&&&&&\\
 (124356, 124356), 214365 &&&&&&\\
\hline
(564231, 645312), (456231, 645123) & 1 & 5 & 61 & 1344 & 44377 & 1951757\\
\hline
(564312, 564312), (456312, 564123), & 1 & 5 & 61 & 1344 & 44377 & 1951429\\
(345612, 561234), 456123 &&&&&&\\
\hline
(465312, 564213), (456213, 465123) & 1 & 5 & 61 & 1344 & 44342 & 1943735\\
\hline
(215634, 341265), (125634, 341256) & 1 & 5 & 61 & 1344 & 44333 & 1940841\\
\hline
(216534, 342165), (126534, 342156) & 1 & 5 & 61 & 1344 & 44333 & 1940623\\
\hline
(546312, 564132), (456132, 546123) & 1 & 5 & 61 & 1344 & 44324 & 1940209\\
\hline
(231654, 321645), (213645, 231465), & 1 & 5 & 61 & 1344 & 44306 & 1937196\\
 (123645, 231456) &&&&&&\\
\hline
(216453, 423165), (126453, 423156) & 1 & 5 & 61 & 1344 & 44306 & 1936673\\
\hline
(216345, 234165), (126345, 234156) & 1 & 5 & 61 & 1344 & 44306 & 1935009\\
\hline
(142365, 214536), (124536, 142356) & 1 & 5 & 61 & 1344 & 44289 & 1935152\\
\hline
(134265, 215346), (125346, 134256) & 1 & 5 & 61 & 1344 & 44289 & 1934933\\
\hline
(214635, 241365), (124635, 241356) & 1 & 5 & 61 & 1344 & 44280 & 1932468\\
\hline
(216435, 243165), (126435, 243156) & 1 & 5 & 61 & 1344 & 44280 & 1931424\\
\hline
(215364, 314265), (125364, 314256) & 1 & 5 & 61 & 1344 & 44271 & 1930657\\
\hline
(215463, 413265), (125463, 413256) & 1 & 5 & 61 & 1344 & 44271 & 1929874\\
\hline
(216354, 324165), (126354, 324156) & 1 & 5 & 61 & 1344 & 44253 & 1926893\\
\hline
\end{tabular}}
\caption{The size of $A_{2n}(q)$ is given for all patterns $q \in S_6$ that
participate in a nontrivial equivalence for even-length alternating permutations.
Parentheses indicate trivial equivalences.}
\label{table:6Even}
\end{table}

\begin{table}
\center{
\begin{tabular}{|l|c|c|c|c|c|c|c|}
\hline
Patterns & 1 & 3 & 5 & 7 & 9 & 11 & 13\\
\hline
(654321, 123456), (654312, 213456), & 1 & 2 & 16 & 272 & 7936 & 329098 & 17316208\\
(654123, 321456), (651234, 432156), &&&&&&&\\
(612345, 543216) &&&&&&&\\
\hline
(634521, 125436), (634512, 215436) & 1 & 2 & 16 & 272 & 7622 & 300499 & 15125692\\
\hline
(653421, 124356), (653412, 214356) & 1 & 2 & 16 & 272 & 7622 & 300430 & 15106854\\
\hline
(645321, 123546), (645312, 213546), & 1 & 2 & 16 & 272 & 7622 & 300430 & 15106113\\
(645123, 321546) &&&&&&&\\
\hline
(564321, 123465), (456321, 123654), & 1 & 2 & 16 & 272 & 7622 & 300430 & 15102362\\
(345621, 126543), (234561, 165432), &&&&&&&\\
(564312, 213465), (456312, 213654), &&&&&&&\\
(345612, 216543), (564123, 321465), &&&&&&&\\
(456123, 321654), (561234, 432165) &&&&&&&\\
\hline
(564213, 312465), (456213, 312654) & 1 & 2 & 16 & 272 & 7622 & 300172 & 15038858\\
\hline
(435621, 126534), (435612, 216534) & 1 & 2 & 16 & 272 & 7622 & 300103 & 15012608\\
\hline
(465321, 123564), (465312, 213564), & 1 & 2 & 16 & 272 & 7622 & 300094 & 15023874\\
(465123, 321564) &&&&&&&\\
\hline
(346521, 125643), (346512, 215643) & 1 & 2 & 16 & 272 & 7622 & 300025 & 15004212\\
\hline
(436521, 125634), (436512, 215634) & 1 & 2 & 16 & 272 & 7622 & 300025 & 14998611\\
\hline
(546321, 123645), (546312, 213645), & 1 & 2 & 16 & 272 & 7622 & 299916 & 14987084\\
 (546123, 321645) &&&&&&&\\
\hline
(365421, 124563), (365412, 214563) & 1 & 2 & 16 & 272 & 7622 & 299897 & \\
\hline
(543621, 126345), (543612, 216345) & 1 & 2 & 16 & 272 & 7622 & 299768 & \\
\hline
(635421, 124536), (635412, 214536) & 1 & 2 & 16 & 272 & 7622 & 299708 & \\
\hline
(356421, 124653), (356412, 214653) & 1 & 2 & 16 & 272 & 7622 & 299698 & \\
\hline
(643521, 125346), (643512, 215346) & 1 & 2 & 16 & 272 & 7622 & 299668 & \\
\hline
(534621, 126435), (534612, 216435) & 1 & 2 & 16 & 272 & 7622 & 299658 & \\
\hline
(536421, 124635), (536412, 214635) & 1 & 2 & 16 & 272 & 7622 & 299639 & \\
\hline
(563421, 124365), (563412, 214365) & 1 & 2 & 16 & 266 & 7164 & 270463 & 13077672\\
\hline
(564231, 132465), (456231, 132654) & 1 & 2 & 16 & 266 & 7164 & 270463 & 13077275\\
\hline
(564132, 231465), (456132, 231654) & 1 & 2 & 16 & 266 & 7156 & 268940 & 12868164\\
\hline
(354621, 126453), (354612, 216453) & 1 & 2 & 16 & 266 & 7156 & 268876 & \\
\hline
(463521, 125364), (463512, 215364) & 1 & 2 & 16 & 266 & 7148 & 267642 & \\
\hline
(453621, 126354), (453612, 216354) & 1 & 2 & 16 & 266 & 7148 & 267590 & \\
\hline
(364521, 125463), (364512, 215463) & 1 & 2 & 16 & 266 & 7148 & 267539 & \\
\hline
\end{tabular}}
\caption{The size of $A_{2n+1}(q)$ is given for all patterns $q \in S_6$ that
participate in a nontrivial equivalence for odd-length alternating permutations.
Parentheses indicate trivial equivalences.}
\label{table:6Odd}
\end{table}

\begin{table}
\center{
\begin{tabular}{|l|c|c|c|c|c|c|c|c|c|}
\hline
Patterns & 1 & 2 & 3 & 4 & 5 & 6 & 7 & 8 & 9\\
\hline
1342 & 1 & 1 & 1 & 2 & 5 & 9& 20 & 64 & 143\\
\hline
1243 & 1 & 1 & 1 & 2 & 5 & 9& 21 & 68 & 153\\
\hline
1423 & 1 & 1 & 1 & 3 & 6 & 9& 42 & 93 & 143\\
\hline
3124 & 1 & 1 & 1 & 3 & {\bf 9} & {\bf 9}& 44 & {\bf 143} & {\bf 143}\\
\hline
2134 & 1 & 1 & 1 & 3 & {\bf 9} & {\bf 9}& 44 & {\bf 153} & {\bf 153}\\
4123 &&&&&&&&&\\
\hline
\end{tabular}}
\caption{This table shows $|D_n^3(q)|$ for selected $q \in S_4$.  Note in particular
the repeated values of $|D^3_n(q)|$ for $q = 2134, 4123$.}
\label{table:4Rep}
\end{table}

\bibliography{finalpaper}
\bibliographystyle{abbrv}

\end{document}